%% file: collapse.tex
\documentclass[12pt]{amsart}

\usepackage{amscd,amssymb,amsthm,verbatim}

\usepackage{epsfig,color,graphicx}

\numberwithin{equation}{section}
\theoremstyle{plain}
\newtheorem{assumption}[equation]{Standing Assumption}
\newtheorem{lemma}[equation]{Lemma}
\newtheorem{theorem}[equation]{Theorem}
\newtheorem{proposition}[equation]{Proposition}
\newtheorem{corollary}[equation]{Corollary}

\newtheorem{sublemma}[equation]{Sublemma}

\theoremstyle{definition}
\newtheorem{definition}[equation]{Definition}

\theoremstyle{remark}
\newtheorem{remark}[equation]{Remark}

\newcommand{\adjust}{\operatorname{adjust}}

\newcommand{\al}{\alpha}

\newcommand{\be}{\beta}

\newcommand{\convexhull}{\operatorname{Hull}}

\newcommand{\const}{\operatorname{const.}}

\newcommand{\D}{\partial}
\newcommand{\De}{\Delta}

\newcommand{\diam}{\operatorname{diam}}

\newcommand{\e}{{\mathcal E}}
\newcommand{\edge}{\operatorname{edge}}
\newcommand{\End}{\operatorname{End}}

\newcommand{\f}{{\mathcal  F}}

\newcommand{\Ga}{\Gamma}
\newcommand{\gen}{\operatorname{gen}}
\newcommand{\genus}{\operatorname{genus}}
\newcommand{\gh}{\operatorname{GH}}

\newcommand{\Gr}{\operatorname{Gr}}

\newcommand{\Id}{\operatorname{Id}}

\newcommand{\im}{\operatorname{im}}
\newcommand{\Image}{\operatorname{Im}}

\newcommand{\inj}{\operatorname{InjRad}}

\newcommand{\Int}{\operatorname{int}}

\newcommand{\Isom}{\operatorname{Isom}}

\renewcommand{\L}{{\mathcal L}}
\newcommand{\La}{\Lambda}

\newcommand{\lra}{\longrightarrow}

\newcommand{\N}{{\mathbb N}}

\newcommand{\Om}{\Omega}

\newcommand{\R}{{\mathbb R}}

\renewcommand{\r}{\frak r}

\newcommand{\restr}{\mbox{\Large \(|\)\normalsize}}

\newcommand{\Rm}{\operatorname{Rm}}

\newcommand{\si}{\sigma}

\newcommand{\SO}{\operatorname{SO}}

\newcommand{\spann}{\operatorname{span}}

\newcommand{\slim}{\operatorname{slim}}
\newcommand{\supp}{\operatorname{supp}}

\renewcommand{\th}{\theta}

\newcommand{\twostratum}{\operatorname{2-stratum}}

\newcommand{\vol}{\operatorname{vol}}
\newcommand{\Z}{{\mathbb Z}}
\newcommand{\zeroball}{\operatorname{0-stratum}}
\newcommand{\cangle}{\widetilde{\angle}}
\newcommand{\ctits}{C_T}
\newcommand{\de}{\delta}

\newcommand{\eps}{\epsilon}
\newcommand{\ga}{\gamma}
\newcommand{\la}{\lambda}
\newcommand{\ol}{\overline}
\newcommand{\ra}{\rightarrow}
\newcommand{\Si}{\Sigma}

\begin{document}

\begin{abstract}
We prove that a $3$-dimensional compact
Riemannian manifold which is locally collapsed,
with respect to a lower curvature bound, is a graph manifold.  This
theorem was stated by Perelman and was used in his proof of
the geometrization conjecture.
\end{abstract}

\title{Locally Collapsed 3-Manifolds}
\author{Bruce Kleiner}
\address{Courant Institute of Mathematical Sciences\\
251 Mercer St.
New York, NY 10012}
\email{bkleiner@cims.nyu.edu}

\author{John Lott}
\address{Department of Mathematics\\
University of California at Berkeley\\ 
Berkeley, CA 94720}
\email{lott@math.berkeley.edu}
\thanks{Research supported by NSF grants 
DMS-0903076 and DMS-1007508}
\date{May 28, 2014}
\maketitle

\setlength{\parskip}{\smallskipamount}

\tableofcontents

\setlength{\parskip}{\medskipamount}

\section{Introduction}

\subsection{Overview}
In this paper we prove that a $3$-dimensional Riemannian manifold 
which is locally collapsed,
with respect to a lower curvature bound, is a graph manifold.  This
result was stated  without
proof by Perelman in \cite[Theorem 7.4]{Perelman2},
where it was used to show that certain collapsed manifolds arising
in his proof of the geometrization conjecture are graph 
manifolds. 
Our goal is to provide a
proof of Perelman's collapsing theorem which is streamlined,
self-contained and accessible.
Other proofs of Perelman's theorem
appear in \cite{bbbmp,Cao-Ge,Morgan-Tian,Shioya-Yamaguchi}.

In the rest of this introduction we state the main result and
describe some of the issues involved in proving it.
We then give an outline of the proof. We finish by
discussing the history of the problem.

\subsection{Statement of results} \label{statement}

We begin by defining an intrinsic local scale
function for a Riemannian manifold.

\begin{definition} \label{def2}
Let $M$ be a complete Riemannian manifold. Given
$p \in M$, the {\em curvature scale} $R_p$ at $p$ is defined as
follows. If the connected component of $M$ containing $p$
has nonnegative sectional curvature then
$R_p = \infty$.
Otherwise, $R_p$
is the (unique) number $r > 0$ such that
the infimum of the sectional curvatures on $B(p,r)$
equals $- \: \frac{1}{r^2}$.
\end{definition}

We need one more definition.

\begin{definition} \label{graphdef}
Let $M$ be a compact orientable 
$3$-manifold (possibly with boundary).
Give $M$ an arbitrary Riemannian metric.  We say that $M$ 
is a {\em graph manifold} if there is a finite
disjoint collection of embedded $2$-tori $\{T_j\}$ in 
the interior of $M$
such that
each connected component of the metric closure of
$M - \bigcup_j T_j$ is the total space of
a circle bundle over a surface (generally with boundary).
\end{definition}

For simplicity, in this introduction we state the main theorem in the case of
closed manifolds. For the general case of manifolds with boundary,
we refer the reader to
Theorem \ref{thmmain2}.

\begin{theorem}(cf. \cite[Theorem 7.4]{Perelman2}) 
\label{thmmain}
Let $c_3$ denote the volume
of the unit ball in $\R^3$ and
let $K \ge 10$ be a fixed integer. 
Fix a function $A:(0, \infty) \ra (0,\infty)$. Then there is a 
$w_0 \in (0,c_3)$  such that the following holds.
 
Suppose that $(M, g)$ is a closed 
orientable Riemannian $3$-manifold.   Assume in addition that for every
$p \in M$,
\begin{enumerate}
\item $\vol(B(p, R_p)) \le w_0 R_p^3$ and
\item For every  
$w'\in [w_0,c_3)$, $k \in[0,K]$,
and $r \leq R_p$ such that $\vol(B(p,r))\geq w'r^3$, the inequality
\begin{equation}
\label{eqn-derivboundsmain}
| \nabla^k\Rm| \le A(w^\prime) \:  r^{-(k+2)}
\end{equation}
holds in the ball $B(p,r)$.
\end{enumerate}

Then $M$ is a graph manifold.
\end{theorem}

\subsection{Motivation}

Theorem \ref{thmmain}, or more
precisely the version for manifolds with boundary, is
essentially the same as Perelman's
\cite[Theorem 7.4]{Perelman2}.  Either
result
can be used to complete 
the Ricci flow proof of Thurston's geometrization conjecture.
We explain this in Section \ref{geom}, following the presentation
of Perelman's work in \cite{Kleiner-Lott}.

To give a brief explanation, let $(M, g(\cdot))$ be a
Ricci flow with surgery whose initial manifold is compact, orientable and
three-dimensional. Put $\widehat{g}(t) = \frac{g(t)}{t}$.
Let $M_t$ denote the time $t$ manifold. (If $t$ is a surgery time
then we take $M_t$ to be the postsurgery manifold.) For any $w > 0$, the 
Riemannian manifold $(M_t, \widehat{g}(t))$ has a decomposition into
a $w$-thick part and a $w$-thin part.  
(Here the terms ``thick'' and ``thin''
are suggested by the Margulis thick-thin decomposition but the
definition is somewhat different. In the case of hyperbolic manifolds,
the two notions are essentially equivalent.)
As
$t \rightarrow \infty$, the $w$-thick part of $(M_t, \widehat{g}(t))$
approaches the $w$-thick part of a complete finite-volume
Riemannian manifold of constant curvature $- \: \frac14$, whose
cusps (if any) are incompressible in $M_t$. Theorem \ref{thmmain}
implies that for large $t$, the $w$-thin part of $M_t$ is a graph manifold.
Since graph manifolds are known to have a geometric decomposition
in the sense of Thurston, this proves the geometrization conjecture.

Independent of Ricci flow considerations, Theorem \ref{thmmain} fits
into the program in Riemannian geometry of understanding 
which manifolds can collapse.
The main geometric assumption in Theorem \ref{thmmain} is the first one,
which is a local collapsing statement,
as we discuss in the next subsection.
The second assumption of Theorem \ref{thmmain} is more technical in nature. 
In the application
to the geometrization conjecture, the validity of the second assumption
essentially arises from the smoothing effect of the Ricci flow equation.

In fact, Theorem \ref{thmmain}
holds without the second assumption. 
In order to prove this stronger result, 
one must use the highly nontrivial Stability Theorem of 
Perelman \cite{Perelman,Kapovitch}. As
mentioned in \cite{Perelman2}, if one does make the second
assumption then one can effectively replace the Stability Theorem
by standard
$C^K$-convergence
of Riemannian manifolds. 
Our proof of Theorem \ref{thmmain} is set up so that it extends to a proof of
the stronger theorem, without the second assumption,
provided that one invokes the Stability Theorem in relevant places;
see Sections \ref{subsec-removingbounds} and
\ref{sec-shioyayamaguchi}.

\subsection{Aspects of the proof}

The strategy in proving Theorem \ref{thmmain} is to first understand
the local geometry and topology of the manifold $M$. One then glues
these local descriptions together to give an explicit
decomposition of $M$ that shows it to be a graph manifold.
This strategy is common to \cite{Shioya-Yamaguchi,Morgan-Tian,Cao-Ge} 
and the present paper. In this subsection we describe the strategy in a
bit more detail.  Some of the new features of the present paper will be
described more fully in Subsection \ref{sec-overview}.

\subsubsection{An example}
The following simple example gives a useful illustration of the strategy
of the proof. 

Let $P\subset H^2$ be a compact convex polygonal domain in the
two-dimensional hyperbolic space. Embedding $H^2$ in the
four-dimensional hyperbolic space $H^4$, let $N_s(P)$
be the metric $s$-neighborhood around $P$ in $H^4$.
Take $M$ to be the boundary $\D N_s(C)$.  If $s$ is sufficiently small
then one can check that the hypotheses of Theorem \ref{thmmain} are
satisfied.

Consider
the structure of $M$ when $s$ is small.  There is a region
$M^{\twostratum}$, lying at 
distance $\ge \const s$ from the boundary $\D P$, which is the total space
of a circle bundle.
At scale comparable to $s$, a suitable neighborhood of a point in 
$M^{\twostratum}$ is nearly isometric to a product
of a planar region with $S^1$.  There is also a region $M^{\edge}$
lying at distance $\le \const s$
from an edge of $P$, but away from the vertices of $P$, which is the
total space of a $2$-disk bundle. At scale
comparable to $s$, a suitable neighborhood of a point in $M^{\edge}$ is
nearly isometric to
the product of an interval with a $2$-disk.  Finally, there is a region
$M^{\zeroball}$ lying at
distance $\le \const s$ from the vertices of $P$.
A connected component of $M^{\zeroball}$ is diffeomorphic to a $3$-disk.

We can choose 
$M^{\twostratum}$, $M^{\edge}$ and $M^{\zeroball}$ so that
there is a decomposition
$M = M^{\twostratum} \cup M^{\edge} \cup M^{\zeroball}$ with the
property that on 
interfaces, fibration structures are compatible.
Now $M^{\edge} \cup M^{\zeroball}$
is a finite union of $3$-disks and $D^2\times I$'s, which is 
homeomorphic to a solid torus.  Also, $M^{\twostratum}$
is a circle bundle over
a $2$-disk, i.e. another solid torus, and
$M^{\twostratum}$ intersects $M^{\edge} \cup M^{\zeroball}$ in a
$2$-torus.
So using this geometric decomposition, we recognize that $M$
is a graph manifold. (In this case 
$M$ is obviously diffeomorphic to $S^3$, being the boundary of a convex set in 
$H^4$, and so it is a graph manifold; 
the point is that one can recognize this using the geometric structure
that comes from the local collapsing.)

\subsubsection{Local collapsing} \label{loccoll}

The statement of Theorem \ref{thmmain} is in terms of 
a {\em local} lower curvature bound, as evidenced by
the appearance of the curvature scale $R_p$.
Assumption (1) of Theorem \ref{thmmain} can be considered to be
a local collapsing statement.
(This is in contrast
to a global collapsing condition, where one assumes that the
sectional curvatures are at least $-1$
and $\vol(B(p,1)) < \epsilon$ for every $p \in M$.)
To clarify the local collapsing statement, we make one more definition.

\begin{definition} \label{def1}
Let $c_3$ denote
the volume of the Euclidean
unit ball in $\R^3$.
Fix $\bar w \in (0, c_3)$.   Given $p \in M$,
the {\em $\bar w$-volume scale at $p$} is 
\begin{equation}
r_p(\bar w) = \inf \{ r > 0 \: : \: \vol(B(p,r))=\bar w\,r^3\}.
\end{equation} 
If there is no such $r$ then we say that the 
$\bar w$-volume scale is infinite.
\end{definition}

There are two ways to look at hypothesis (1) of 
Theorem \ref{thmmain}, at the
curvature scale or at the volume scale.  Suppose first that
we rescale the ball $B(p, R_p)$ to have radius one. Then the
resulting ball
will have sectional curvature bounded below by $-1$ and volume
bounded above by $w_0$. 
As $w_0$ will be small, we can say that on the curvature scale, the manifold
is locally volume collapsed with respect to a lower curvature bound.
On the other hand, suppose that we rescale $B(p, r_p(w_0))$ to
have radius one. Let 
$B'(p, 1)$
denote the rescaled ball. Then 
$\vol(B'(p, 1)) = w_0$. 
Hypothesis (1) of Theorem \ref{thmmain}
implies that there is a big number ${\mathcal R}$ so that
the sectional curvature on
the radius ${\mathcal R}$-ball 
$B'(p, {\mathcal R})$ 
(in the rescaled manifold) is bounded below by
$- \: \frac{1}{{\mathcal R}^2}$.
Using this, we deduce that 
on the volume scale, a large neighborhood of $p$ 
is well approximated by a large region in a complete 
nonnegatively curved $3$-manifold $N_p$. 
This gives a local model for the geometry of $M$.
Furthermore, if $w_0$ is small then we can say that at the
volume scale, the neighborhood of $p$ is close in a coarse sense to
a space of dimension less than three.

In order to prove Theorem \ref{thmmain}, one must first choose on which
scale to work.  We could work on the curvature scale, or the
volume scale, or some intermediate scale (as is done in
\cite{Morgan-Tian,Shioya-Yamaguchi,Cao-Ge}). In this paper we will work consistently
on the volume scale.  This gives a uniform and simplifying approach.

\subsubsection{Local structure}

At the volume scale, the local geometry of $M$ is well approximated by
that of a nonnegatively curved $3$-manifold. 
(That we get a $3$-manifold instead of a $3$-dimensional
Alexandrov space comes from the second assumption in Theorem 
\ref{thmmain}.)  
The topology of nonnegatively curved $3$-manifolds
is known in the compact case by work of Hamilton
\cite{Hamilton3d,Hamiltonnnn} and in the noncompact case by work
of Cheeger-Gromoll \cite{Cheeger-Gromoll}. In the latter case, the
geometry is also well understood.  Some relevant examples
of such manifolds are: 
\begin{enumerate}
\item $\R^2 \times S^1$,
\item $\R \times S^2$, 
\item $\R \times \Sigma$, where $\Sigma$ is a noncompact nonnegatively
curved surface which is diffeomorphic to $\R^2$ and has a
cylindrical end, and
\item $\R \times_{\Z_2} S^2$.
\end{enumerate}

If a neighborhood of a point $p \in M$ is modeled by 
$\R^2 \times S^1$ at the volume scale then the length of the
circle fiber is comparable to $w_0$. Hence if $w_0$ is small then
the neighborhood looks almost like a $2$-plane. Similarly,
if the neighborhood is modeled by $\R \times S^2$ then it looks
almost like a line.  On the other hand, if the neighborhood is
modeled by $\R \times \Sigma$ then for small $w_0$, the surface $\Sigma$
looks almost like a half-line and the neighborhood looks almost
like a half-plane.  Finally, if the neighborhood is modeled by
$\R \times_{\Z_2} S^2$ then it looks almost like a half-line.

\subsubsection{Gluing}

The remaining issue is use the
local geometry to deduce the global topology of $M$. This is
a gluing issue, as the local models need to be glued together
to obtain global information.

One must determine which local models should be
glued together. 
We do this by means of a stratification of $M$.
If $p$ is a point in $M$ then for $k \le 2$, 
we say that $p$ is a $k$-stratum point if on the volume scale,
a large ball around $p$
approximately splits off an $\R^k$-factor metrically,
but not an $\R^{k+1}$-factor. 

For $k \in \{1,2\}$, 
neighborhoods
of the $k$-stratum points will glue together 
in order to produce the total space of a fibration
over a $k$-dimensional manifold. For example, neighborhoods of the
$2$-stratum points will glue together to form a circle bundle over
a surface. 
Neighborhoods of the $0$-stratum points play a somewhat
different role.  They will be inserted as ``plugs''; for example,
neighborhoods of
the exceptional fibers in a Seifert fibration will arise in this way.

By considering how $M$ is decomposed into these various subspaces that
fiber, we will be able to show that $M$ is a graph manifold.

\subsection{Outline of the proof}
\label{sec-overview}

We now indicate the overall structure of the proof of
Theorem \ref{thmmain}. In this subsection we suppress parameters
or denote them by $\const$
In the paper we will use some minimal facts about pointed
Gromov-Hausdorff convergence and Alexandrov spaces, which
are recalled in Section \ref{preliminaries}.

\subsubsection{Modified volume scale}
\label{subsec-modifiedvolumescale}

The first step is to replace the volume scale by a slight modification
of it.
The motivation for this step is
the fluctuation of the volume scale. Suppose that 
$p$ and $q$ are points in overlapping local models.  As these local
models are at the respective volume scales, there will be
a problem in gluing
the local models together if $r_q(\bar w)$ differs wildly from
$r_p(\bar w)$. We need control on how the volume
scale fluctuates on a ball of the form $B(p,  \const r_p(\bar w))$.
We deal with this problem
by replacing the volume scale $r_p(\bar w)$ by a modified scale
which has better properties.
We assign a scale $\r_p$ to each point $p\in M$ such that:
\begin{enumerate}
\item $\r_p$ is much less than the curvature scale
$R_p$.
\item The function $p\mapsto\r_p$ is smooth and has  Lipschitz
constant $\Lambda\ll 1$. 
\item The ball $B(p,\r_p)$ has volume lying in the
interval $[w'\r_p^3,\bar w\r_p^3]$, where $w'<\bar w$ are
suitably chosen constants lying in the interval $[w_0,c_3]$.
\end{enumerate}

The proof of the existence of the scale function $p\mapsto\r_p$ 
follows readily from the local collapsing assumption, the Bishop-Gromov 
volume comparison theorem, and an argument similar to McShane's extension
theorem for real-valued Lipschitz functions; see 
Section \ref{sec-scalefunction}.

\subsubsection{Implications of compactness}
\label{subsec-compactnessimplications}

Condition (1) above implies that the rescaled manifold $\frac{1}{\r_p}M$,  
in the vicinity of $p$, has almost nonnegative
curvature. Furthermore, condition (3) implies that
it looks collapsed but not too collapsed, in the
sense that the volume of the unit ball around $p$ in 
the rescaled manifold
$\frac{1}{\r_p}M$ is small but not too small. Thus by 
working at the scale $\r_p$, we are able to retain the local collapsing
assumption (in a somewhat weakened form) while
gaining improved behavior of the scale function.

Next, the bounds (\ref{eqn-derivboundsmain})
extend to give 
bounds on the derivatives of the curvature tensor of the form
\begin{equation}
\label{eqn-derivboundsrescaled}
|\nabla^k\Rm| \: \leq A'(C,w')
\end{equation}
for $0\leq k\leq K$, when restricted to
balls $B(p,C)$
in $\frac{1}{\r_p}M$.  Using
(\ref{eqn-derivboundsrescaled}) and
standard compactness
theorems for pointed Riemannian manifolds, we get:
\begin{enumerate}
\setcounter{enumi}{3}
\item 
For every $p\in M$,
the rescaled pointed manifold $(\frac{1}{\r(p)}M,p)$  is
close in the pointed 
$C^{K}$-topology
to a pointed nonnegatively curved
$C^{K}$-smooth
Riemannian $3$-manifold $(N_p,\star)$.  
\item 
For every $p\in M$,
the pointed manifold $(\frac{1}{\r(p)}M,p)$ is 
close in the pointed Gromov-Hausdorff
topology to a pointed nonnegatively curved Alexandrov
space $(X_p,\star)$ of dimension at most $2$.
\end{enumerate}

\subsubsection{Stratification}

The next step is to define a partition of $M$ into
$k$-stratum points, for $k \in \{0,1,2\}$. The partition
is in terms of the number of $\R$-factors that approximately
split off in $(\frac{1}{\r(p)}M,p)$.

Let $0 < \beta_1 < \beta_2$ be new parameters.
Working at scale $\r_p$, we classify points in $M$ 
as follows
(see Section \ref{sec-stratification}):

\begin{itemize}
\item  A point
$p\in M$ lies in the {\bf $2$-stratum} if 
$(\frac{1}{\r(p)}M,p)$ is $\beta_2$-close to $(\R^2, 0)$ in the pointed
Gromov-Hausdorff topology.
\item  A point $p\in M$ lies
in the {\bf $1$-stratum}, if it does not lie in the $2$-stratum, 
but $(\frac{1}{\r(p)}M,p)$ is $\beta_1$-close to 
$\left( \R \times Y_p, \left(0, \star_{Y_p} \right) \right)$ in
the pointed Gromov-Hausdorff topology, where $Y_p$ is a point,
a circle, an interval or a half-line, 
and $\star_{Y_p}$ is a basepoint in $Y_p$.
\item  A point lies in the {\bf $0$-stratum} if it does not lie
in the $k$-stratum for $k\in \{1,2\}$.
\end{itemize}

We now discuss the structure near points in the
different strata in more detail, describing the model spaces
$X_p$ and $N_p$.

\subsubsection*{$2$-stratum points} 
(Section \ref{sec-loc2stratum}). 
If $\beta_2$ is small and
$p\in M$ is a $2$-stratum
point
then $X_p$ is isometric to $\R^2$, while $N_p$
is isometric to a product $\R^2\times S^1$ where the $S^1$
factor is small.  Since the pointed rescaled manifold 
$(\frac{1}{\r_p}M,p)$ is close to $(N_p,\star)$, we can
transfer the projection map $N_p\simeq \R^2\times S^1\ra \R^2$
to a map $\eta_p$ defined on a large ball  $B(p,C)\subset\frac{1}{\r_p}M$,
where it defines a circle fibration.

\subsubsection*{$1$-stratum points}
(Sections \ref{sec-edgepoints} and \ref{sec-locslim}).
If $\beta_1$ is small and $p$ is a $1$-stratum point then 
$X_p=\R\times Y_p$, where $Y_p$
is a point, a circle, an interval or a half-line.
The 
$C^K$-smooth
model space $N_p$ will be an isometric product 
$\R \times \bar N_p$, where
$\bar N_p$ is a complete nonnegatively curved orientable surface.
As in the
$2$-stratum case, we can transfer the projection map
$N_p\simeq \R\times \bar N_p\ra \R$ to a map $\eta_p$
defined on a large  ball $B(p,C)\subset \frac{1}{\r_p}M$, 
where it defines a submersion.

We further 
classify the $1$-stratum points according to the diameter of the
cross-section $Y_p$.  If the diameter of $Y_p$ is not too large then
we say that $p$ lies in the {\em slim $1$-stratum}.
(The motivation for the terminology is that in this
case $(\frac{1}{\r_p}M,p)$ appears slim, being at moderate Gromov-Hausdorff
distance from a line.)  For slim $1$-stratum points, the cross-section
$\bar N_p$ is diffeomorphic to $S^2$ or $T^2$. Moreover, 
in this case the
submersion $\eta_p$ will be a fibration with fiber
diffeomorphic to $\bar N_p$.  

We also distinguish another type
of $1$-stratum point, the {\em edge points}.  A $1$-stratum
point $p$ is an edge point if $(X_p,\star)=(\R\times Y_p,\star)$
can be taken to be pointed isometric to a flat Euclidean half-plane whose
basepoint lies on the edge.  Roughly speaking, we show that near $p$,
the set  $E'$ of edge points looks like a $1$-dimensional manifold at scale
$\r_p$. Furthermore, there is a smooth function $\eta_{E'}$ which behaves
like the ``distance to the edge'' and which combines with $\eta_p$
to yield ``half-plane coordinates'' for $\frac{1}{\r_p}M$ near $p$.
When restricted to an appropriate sublevel set of $\eta_{E'}$, 
the map $\eta_p$ defines a fibration with fibers diffeomorphic to a
compact surface  with boundary $F_p$.  
Using the fact that for edge points 
$N_p=\R\times \bar N_p$ is Gromov-Hausdorff close to a half-plane, 
one sees that
the pointed surface $(\bar N_p,\star)$ is Gromov-Hausdorff close to a pointed
ray $([0,\infty),\star)$. This allows one to conclude that 
$F_p$ is diffeomorphic to a closed $2$-disk.

\subsubsection*{$0$-stratum points} 
(Section \ref{sec-loc0stratum}).  We know by (4) that if $p$
is a $0$-stratum point, then $(\frac{1}{\r_p}M,p)$ is 
$C^{K}$-close
to a nonnegatively curved 
$C^K$-smooth
$3$-manifold $N_p$. 
The idea for analyzing the structure of $M$ near a $0$-stratum point $p$
is to use the fact that nonnegatively curved manifolds look
asymptotically like cones, and are diffeomorphic to any
sufficiently large ball in them (centered at a fixed basepoint).
More precisely, we find a scale $r_p^0$ with
$\r_p\leq r_p^0\leq \const\r_p$ so that:
\begin{itemize}
\item
The pointed rescaled manifold $(\frac{1}{r^0_p}M,p)$ is close in the
$C^{K}$-topology to a $C^K$-smooth
nonnegatively curved $3$-manifold $(N_p',\star)$.
\item 
The distance function $d_p$ in 
$\frac{1}{\r_p^0}M$
has no critical points in the metric annulus 
$A(p,\frac{1}{10},10) = \overline{B(p,10)} - B(p,\frac{1}{10})$,
and $B(p,1)\subset \frac{1}{r_p^0}M$ is diffeomorphic to $N_p'$.
\item The pointed space $(\frac{1}{r^0_p}M,p)$ is close in the 
pointed Gromov-Hausdorff topology to  a Euclidean cone (in fact
the Tits cone of $N_p'$).  
\item $N_p'$ has at most one end.
\end{itemize}
\noindent
The proof of the existence of the scale $r_p^0$ is based on
the fact that nonnegatively curved
manifolds are asymptotically conical,  the critical point
theory of Grove-Shiohama \cite{Grove-Shiohama},
and a compactness argument.  Using the approximately conical
structure, one obtains a smooth function $\eta_p$ on $\frac{1}{r_p^0}M$
which, when restricted to the
metric annulus $A(p,\frac{1}{10},10)\subset \frac{1}{r_p^0}M$,
behaves like the radial function on a cone.
In particular, for $t\in [\frac{1}{10},10]$,
the sublevel sets $\eta_p^{-1}[0,t)$  are diffeomorphic to  $N_p'$.

The soul theorem \cite{Cheeger-Gromoll},
together with
Hamilton's classification of closed nonnegatively curved
$3$-manifolds \cite{Hamilton3d,Hamiltonnnn}, implies that
$N_p'$ is diffeomorphic to one of the following: 
a manifold $W/\Gamma$ where $W$ is either $S^3$, $S^2\times S^1$
or $T^3$ equipped with a standard Riemannian metric
and $\Gamma$ is a finite group of isometries;
$S^1\times \R^2$; $S^2\times \R$, $T^2\times \R$;
or a twisted line bundle over $\R P^2$ or the Klein bottle. Thus we know the possibilities for the 
topology of $B(p,1)\subset \frac{1}{r_p^0}M$.

\subsubsection{Compatibility of the local structures}
Having determined the local structure of $M$ near each point, 
we examine how these local structures fit together on their
overlap.  For example, consider the slim $1$-stratum
points corresponding to an $S^2$-fiber. A neighborhood of
the set of
such points looks like a union of cylindrical regions. If 
the axes of overlapping cylinders are very well-aligned then
the process of gluing them together will be simplified. 
It turns out that such compatibility is automatic from our choice of
stratification.

To see this,  suppose that $p,q\in M$ are $2$-stratum 
points with 
$B(p, \const \r_p) \cap B(q, \const \r_q)\neq\emptyset$.
Then provided that 
$\La$ is small,  we know that $\r_p\approx \r_q$.
Suppose now that $z \in B(p, \const \r_p) \cap B(q, \const \r_q)$.
We have two $\R^2$-factors at $z$, coming from the
approximate splittings at $p$ and $q$. 
If the parameter $\beta_2$ is small then these $\R^2$-factors
must align well at $z$. If not then we would get two misaligned
$\R^2$-factors at $p$, which would generate an approximate
$\R^3$-factor at $p$, contradicting the local collapsing assumption.
Hence the maps $\eta_p$ and $\eta_q$, which arose
from approximate $\R^2$-splittings, are nearly ``aligned'' with
each other on their overlap, so that $\eta_p$ and $\eta_q$
are affine functions of each other, up to arbitrarily small $C^1$-error. 

Now fix $\beta_2$.
Let $p,q \in M$ be  $1$-stratum points.
At any $z \in B(p, \const \r_p) \cap B(q, \const \r_q)$,
there are two $\R$-factors, coming from the approximate $\R$-splittings
at $p$ and $q$.  If $\beta_1$ is small then these two $\R$-factors
must align well at $z$, or else we would get two misaligned
$\R$-factors at $p$, contradicting the fact that $p$ is not a
$2$-stratum point.
Hence the functions $\eta_p$
and $\eta_q$ are also affine functions of each other, up to 
arbitrarily small
$C^1$-error.  

One gets additional compatibility properties for pairs
of points of
different types.  For example, if $p$ lies in the $0$-stratum
and $q\in A(p, \frac{1}{10}, 10) \subset \frac{1}{\r_p^0}M$
belongs to the $2$-stratum then the
radial function $\eta_p$,
when appropriately rescaled,
agrees with an affine function of $\eta_q$
in $B(q,10)\subset \frac{1}{\r_q}M$ 
up to small $C^1$-error.

\subsubsection{Gluing the local pieces together}
(Sections \ref{sec-mapping}-\ref{sec-extracting}).
To begin the gluing process, we select a separated
collection of points of each type in $M$:
$
\{p_i\}_{i\in I_{\twostratum}}$,
$\{p_i\}_{i\in I_{\slim}}$,
$\{p_i\}_{i\in I_{\edge}}$,
$\{p_i\}_{i\in I_{\zeroball}}$,
so that 
\begin{itemize}
\item $\bigcup_{i \in I_{\twostratum}} B(p_i, \const 
\r_{p_i})$ covers the $2$-stratum points,
\item $\bigcup_{i \in I_{\slim} \cup I_{\edge}} B(p_i, \const 
\r_{p_i})$ covers the $1$-stratum points, and
\item
$\bigcup_{i \in I_{zeroball}} B(p_i, \const 
r^0_{p_i})$ covers the $0$-stratum points.
\end{itemize}

Our next objective is to combine the  $\eta_{p_i}$'s so as to define
global fibrations for each of the different types of points, 
and ensure that these fibrations are compatible on overlaps.
To do this, we borrow an idea from the proof of the Whitney
embedding theorem (as well as proofs of
Gromov's compactness theorem \cite[Chapter 8.D]{Gromov},
\cite{Katsuda}):
we define a smooth map 
$\e^0:M\ra H$ into a 
high-dimensional Euclidean space $H$. The components
of $\e^0$ are functions of the $\eta_{p_i}$'s, the edge
function $\eta_{E'}$, and the scale
function $p\mapsto \r_p$, cutoff appropriately so that they define
global smooth functions.  

Due to the pairwise compatibility of the 
$\eta_{p_i}$'s discussed above, it turns out that the image 
under $\e^0$ of $\bigcup_{i \in I_{\twostratum}} B(p_i, \const 
\r_{p_i})$ is a 
subset $S\subset H$ which, when viewed at the right scale, is everywhere
locally close (in the pointed Hausdorff sense)
to a $2$-dimensional affine subspace.  We call such a set a 
{\em cloudy $2$-manifold}.   By an elementary argument, we show
in Appendix \ref{sec-cloudy} that a cloudy manifold of any dimension can be 
approximated by  a core manifold $W$ whose
normal injectivity radius is controlled. 

We adjust the map $\e^0$  by ``pinching'' it 
into the manifold core of $S$, thereby upgrading  $\e^0$ to a new map
$\e^1$ which is a circle fibration near the $2$-stratum.
The new map $\e^1$ is $C^1$-close to $\e^0$.
We then perform similar adjustments near the edge points
and slim $1$-stratum points, to obtain a map $\e:M\ra H$
which yields fibrations when restricted to certain regions
in $M$, see Section \ref{sec-adjusting}.  For example, 
we obtain 
\begin{itemize}
\item A $D^2$-fibration of a
region of $M$ near the edge set $E'$,
\item  
$S^2$ or $T^2$-fibrations of a region containing the slim stratum, and 
\item
A surface
fibration collaring (the boundary of) the region near $0$-stratum points.
\end{itemize}
Furthermore, it is a feature built into the construction
that where the fibered regions overlap, they do so in surfaces with boundary
along which the two fibrations are compatible.
For instance, 
the interface between the edge fibration and the $2$-stratum
fibration is a surface which inherits the same circle fibration
from the edge fibration and the $2$-stratum fibration.
Similarly, the
interface between the $2$-stratum fibration and the slim 
$1$-stratum fibration is a surface
with boundary which inherits a circle fibration from the $2$-stratum.
See Proposition \ref{prop-decompositionprop} for the properties
of the fibrations.

\subsubsection{Recognizing the graph manifold structure} (Section 
\ref{sec-proof})
At this stage of the argument, one has a decomposition of $M$
into domains with disjoint interiors,
where each domain is a compact $3$-manifold with corners
carrying a fibration of a specific kind, with compatibility of
fibrations on overlaps.  Using the topological
classification of the  fibers and the  $0$-stratum domains, one readily 
reads off the graph manifold structure.  This completes the proof
of Theorem \ref{thmmain}.

\subsubsection{Removing the bounds on derivatives of 
curvature}
\label{subsec-removingbounds}
(Section \ref{sec-shioyayamaguchi}).
The proof of Theorem \ref{thmmain}
uses the derivative bounds (\ref{eqn-derivboundsmain})
only for 
$C^{K}$-precompactness
results.
In turn these are essentially used  only to determine the
topology of the $0$-stratum balls and the fibers of the edge
fibration.  Without the derivative bounds (\ref{eqn-derivboundsmain}),
one can appeal to similar compactness arguments.
However, one ends up with a sequence of pointed Riemannian manifolds
$\{(M_k,\star_k)\}$ which converge in the pointed Gromov-Hausdorff
topology to a pointed $3$-dimensional nonnegatively curved
Alexandrov space 
$(M_\infty,\star_\infty)$,
rather than having 
$C^K$-convergence to a $C^K$-smooth 
limit. By invoking Perelman's Stability Theorem
\cite{Perelman,Kapovitch}, one can
relate the topology of the limit space to those of the approximators.
The only remaining step is to determine the topology of the 
nonnegatively curved
Alexandrov spaces that arise as limits in this fashion.  In the
case of noncompact limits, this was done by Shioya-Yamaguchi
\cite{Shioya-Yamaguchi}.   In the compact case, it follows
from Simon \cite{Simon} or, alternatively, from the Ricci flow
proof of the elliptization conjecture (using the finite time 
extinction results of Perelman and Colding-Minicozzi).
For more details, we refer the reader to Section 
\ref{sec-shioyayamaguchi}.

\subsubsection{What's new in this paper}
\label{sec-whatsnew}

The
proofs of the collapsing theorems in 
\cite{Shioya-Yamaguchi1,Shioya-Yamaguchi,bbbmp,Morgan-Tian,Cao-Ge},
as well as the proof in this paper, 
all begin by comparing the local geometry at a certain scale
with the geometry of a nonnegatively curved manifold, and then
use this structure to deduce that one has a graph manifold.  The paper
\cite{bbbmp} follows a rather different 
line from the  other proofs, in that it uses the least amount
of the information available from the nonnegatively curved models,
and proceeds with a covering argument based on the theory
of simplicial volume, as well as Thurston's proof of the geometrization
theorem for Haken manifolds.  The papers
\cite{Shioya-Yamaguchi1,Shioya-Yamaguchi,Morgan-Tian,Cao-Ge} and this paper
have a common overall
strategy, which is to use more of the theory of manifolds with
nonnegative sectional curvature -- Cheeger-Gromoll theory 
\cite{Cheeger-Gromoll} and
critical point theory \cite{Grove-Shiohama} --
to obtain a more refined version of the local models. Then the
local models are spliced together to obtain a decomposition of
the manifold into fibered regions from which one can recognize
a graph manifold.

Overall, our proof uses a minimum of material beyond the theory
of nonnegatively curved manifolds.
It is essentially elementary in flavor.
We now comment on some specific new points in our approach.

\subsubsection*{The  scale function $\r_p$}
The existence of a scale function $\r_p$ with the properties 
indicated in Section \ref{subsec-compactnessimplications} makes
it apparent that the theory of local collapsing is, at least 
philosophically, no different than the global version of collapsing.

We work consistently at the scale $\r_p$, which streamlines the argument.
In particular,
the structure theory of 
Alexandrov spaces,
which enters if one works at the curvature scale,
is largely eliminated.
Also, in the selection argument, one considers ball covers where
the radii are linked to the scale function $\r$, 
so one easily obtains
bounds on the intersection multiplicity from the fact that
the radii of
intersecting balls are comparable (when the scale function
$p\mapsto\r_p$ has small Lipschitz constant).
The technique of constructing a scale function with
small Lipschitz constant could help in other geometric
gluing problems.

\subsubsection*{The stratification}

Stratifications 
have a long history in geometric analysis, especially for singular
spaces such as convex sets, minimal varieties, Alexandrov spaces, and Ricci limit spaces, 
where one typically looks at the number of $\R$-factors that split off in a
tangent cone. 
The particular stratification that we use, based on the
number of $\R$-factors that approximately split off in a manifold, 
was not used in collapsing theory before, to our knowledge.  
Its implications for
achieving alignment may be useful in other settings.

\subsubsection*{The gluing procedure}

Passing from local models to global fibrations
involves some kind of gluing process.  Complications
arise from the fact one has to construct a global base space
for the fibration at the same time as one glues together
the fibration maps; in addition, one has to make the 
fibrations from the different strata compatible. 
The most obvious
approach is to add fibration patches inductively,
by using small isotopies and the fact that on overlaps, the
fibration maps are nearly affinely equivalent.  Then one
must perform further isotopies to make fibrations from the
different strata compatible with one another.  We find
the gluing procedure used here to be more elegant; moreover, it
produces fibrations which are automatically compatible.

Embeddings into a Euclidean space were used before to
construct fibrations in a collapsing setting
\cite{Fukaya}.  However, there is the important difference that
in the earlier work the base of the fibration was already
specified, and this base was embedded into a Euclidean space.
In contrast, in the present paper we must produce the
base at the same time as the fibrations, so we produce it as
a submanifold of the Euclidean space.

\subsubsection*{Cloudy manifolds}
The notion of cloudy manifolds, and the proof that they
have a good manifold core, may be of independent interest.
Cloudy manifolds are similar to objects that have  been
encountered before,  in the work
of Reifenberg \cite{Reifenberg} in geometric measure theory 
and also in  \cite{Pugh}.  However
the clean elementary argument for the existence
of a smooth core given in Appendix \ref{sec-cloudy}, 
using the universal bundle and transversality,
seems to be new.

\subsubsection{A sketch of the history}
The theory of collapsing was first developed by 
Cheeger and Gromov \cite{Cheeger-Gromov,Cheeger-Gromov2}, 
assuming both upper and lower bounds
on sectional curvature.   Their work characterized the degeneration
that can occur when one drops the injectivity radius bound
in Gromov's compactness theorem, generalizing Gromov's theorem
on almost flat manifolds \cite{Gromovalmostflat}. 
The corresponding local collapsing structure was used by 
Anderson and Cheeger-Tian in work on Einstein manifolds 
\cite{Anderson,Cheeger-Tian}.
As far as we know, the
first results on collapsing with a lower curvature bound
were announced by Perelman in the early 90's,
as an application of the theory of Alexandrov spaces,
in particular his Stability Theorem from
\cite{Perelman} (see also \cite{Kapovitch}); however, these results
were never published.
Yamaguchi \cite{Yamaguchi} established a fibration
theorem for manifolds close to Riemannian manifolds, under
a lower curvature bound.
Shioya-Yamaguchi \cite{Shioya-Yamaguchi1} studied 
collapsed $3$-manifolds with a diameter bound
and showed that they are graph manifolds, apart from an exceptional
case.
In \cite{Perelman2}, Perelman formulated without proof
a theorem equivalent
to our Theorem \ref{thmmain}.
A short time later, Shioya-Yamaguchi \cite{Shioya-Yamaguchi} proved
that -- apart from an exceptional case --
sufficiently collapsed $3$-manifolds are graph manifolds,
this time without assuming a diameter bound.  This result (or
rather the localized version they discuss in their appendix)  may be used
in lieu of \cite[Theorem 7.4]{Perelman2} to complete the 
proof of the geometrization conjecture.  Subsequently, 
Bessi\`eres-Besson-Boileau-Maillot-Porti \cite{bbbmp} gave a 
different approach to the last part of the proof of the geometrization 
conjecture, which involves collapsing as well as refined results
from $3$-dimensional topology.  Morgan-Tian \cite{Morgan-Tian}  gave 
a proof
of Perelman's collapsing result along the lines of
Shioya-Yamaguchi \cite{Shioya-Yamaguchi}.
We also mention the paper [CG] by Cao-Ge which 
relies on more sophisticated Alexandrov space results.

\subsection{Acknowledgements}  We thank
Peter Scott for some references to the $3$-manifold literature.

\section{Notation and conventions}
\label{sec-indexnotation}

\subsection{Parameters and constraints}
The rest of the paper develops a lengthy construction,
many steps of which generate new constants; we will
refer to these as {\em parameters}.
Although the  parameters remain fixed after being introduced, one 
should view different sets of parameter values as defining
different potential instances of the construction.  This is
necessary, because several arguments involve
consideration of sequences of values for certain parameters,
which one should associate with a sequence of 
distinct instances of the construction.  

Many steps of the argument assert that certain statements
hold provided that certain constraints on the
parameters are satisfied.   
By convention, each time we refer to such a constraint, we will
assume for the remainder of the paper that the 
inequalities in question are satisfied.  Constraint functions
will be denoted with a bar, e.g. $\be_E<\overline{\be}_E(\be_1,\si)$
means that $\be_E\in (0,\infty)$ satisfies an upper bound which is a function
of $\be_1$ and $\si$. 
By convention, all constraint functions take values in $(0,\infty)$.

At the end of the proof of Theorem \ref{thmmain}, 
we will verify that the constraints on the
various parameters can be imposed consistently.  Fortunately,
we do not have to carefully adjust each parameter in terms of the
others; the constraints are rather of the form that one parameter
is sufficiently small (or large) in terms of some others.  Hence
the only issue is the order in which the parameters are considered.

We follow
Perelman's convention that a condition like $a > 0$ means that
$a$ should be considered to be a small parameter, while a condition
like $A < \infty$ means that $A$ should be considered to be a large
parameter. This convention is only for expository purposes and
may be ignored by a logically minded reader.

\subsection{Notation} \label{notation}
We will use the following compact notation
for cutoff functions with prescribed support. 
Let $\phi \in C^\infty(\R)$ be a nonincreasing function so that
$\phi \Big|_{(-\infty,0]} = 1$,  $\phi \Big|_{[1,\infty)} = 0$
and $\phi((0,1)) \subset (0,1)$.
Given $a,b \in \R$ with $a < b$, we define $\Phi_{a,b} \in C^\infty(\R)$ by
\begin{equation}
\Phi_{a,b}(x) = \phi(a+(b-a)x),
\end{equation}
so that
$\Phi_{a,b} \Big|_{(-\infty,a]} = 1$ and 
$\Phi_{a,b} \Big|_{[b,\infty)} = 0$. Given
$a,b,c,d \in \R$ with $a<b<c<d$, we define 
$\Phi_{a,b,c,d} \in C^\infty(\R)$ by
\begin{equation}
\Phi_{a,b,c,d}(x) = \phi_{-b,-a}(-x) \: \phi_{c,d}(x),
\end{equation}
so that
$\Phi_{a,b,c,d} \Big|_{(-\infty,a]} = 0$,
$\Phi_{a,b,c,d} \Big|_{[b,c]} = 1$ and
$\Phi_{a,b,c,d} \Big|_{[d,\infty)} = 0$.

If $X$ is a metric space and $0 < r \le R$ then the annulus
$A(x, r, R)$ is $\overline{B(x,R)} - B(x,r)$.
The dimension of a metric space will always mean the
Hausdorff dimension.
For notation, if $C$ is a metric cone with basepoint at the
vertex $\star$ then we
will sometimes just write $C$ for the pointed metric space $(C, \star)$.
(Recall that a metric cone is
a pointed metric space $(Z,\star)$, which is a union of rays leaving
the basepoint $\star$, such that the union of any two such rays
is isometric to the union of two rays leaving the origin in 
$\R^2$.)

If $Y$ is a subset of $X$ and 
$t:Y\ra (0,\infty)$ is a function then we
write $N_t(Y)$ for the neighborhood of $Y$ with
variable thickness $t$:  $N_t(Y)=\bigcup_{y\in Y}\;B(y,t(y))$.

If $(X,d)$ is a metric space and $\lambda > 0$ then we write $\lambda X$
for the metric space $(X, \lambda d)$. For notational simplicity, we write
$B(p,r)\subset \la X$ to denote the $r$-ball around $p$ in the
metric space $\la X$.

Throughout the paper, a product  of metric spaces $X_1\times X_2$
will be endowed with the distance function given by the Pythagorean formula,
i.e. if $(x_1,x_2),\,(y_1,y_2)\in X_1\times X_2$ then
$d_{X_1 \times X_2}((x_1,x_2),(y_1,y_2))=
\sqrt{d^2_{X_1}(x_1,y_1)+d^2_{X_2}(x_2,y_2)}$.

\subsection{Variables}
For the reader's convenience, we 
tabulate the variables in this paper, listed by the
section in which they first appear. \\ \\
\noindent
Section \ref{statement}: $R_p$ \\
Section \ref{loccoll}: $r_p(\cdot)$ \\
Section \ref{sec-scalefunction}: $\Lambda$, $\sigma$, $\r_p$, $\bar{w}$,
$w'$\\
Section \ref{subsec-k-stratum}: $\beta_1$, $\beta_2$, $\beta_3$, $\De$ \\
Section \ref{adapted2}: $\varsigma_{\twostratum}$, $\eta_p$ and $\zeta_p$
(for $2$-stratum points) \\
Section \ref{2selection}: ${\mathcal M}$ \\
Section \ref{subsec-edgepoints}: 
$\beta_E$, $\si_E$, $\beta_{E'}$, $\si_{E'}$ \\
Section \ref{subsec-regularizationofrhoe'}: $d_{E'}$, $\rho_{E'}$,
$\varsigma_{E'}$ \\
Section \ref{subsec-edgetangent}: $\varsigma_{\edge}$, $\eta_p$ and $\zeta_p$
(for edge points) \\
Section \ref{subsec-additionalcutoff}: $\zeta_{\edge}$, 
$\zeta_{E'}$ \\
Section \ref{adaptedslim}: $\varsigma_{\slim}$, $\eta_p$ and $\zeta_p$
(for slim $1$-stratum points)\\
Section \ref{subsecann}: $\Upsilon_0$, $\Upsilon_0'$, $\delta_0$, $r^0_p$ \\
Section \ref{subsec-0stratumscale}: 
$\varsigma_{\zeroball}$, $\eta_p$ and $\zeta_p$
(for $0$-stratum points)\\
Section \ref{subsecmapping}: $H$, $H_i$, 
$H_i'$, $H_i''$,
$H_{\zeroball}$, $H_{\slim}$,
$H_{\edge}$, $H_{\twostratum}$, $Q_1$, $Q_2$, $Q_3$, $Q_4$,
$\pi_i$, $\pi_i^\perp$, $\pi_{ij}$, $\pi_{H_i'}$, $\pi_{H_i''}$, $\e^0$ \\
Section \ref{subsecimagee0}: $\Omega_0$ \\
Section \ref{subsec-e0near2stratum}: $A_1$, $\widetilde{A}_1$,
$S_1$, $\widetilde{S}_1$, $\Omega_1$, $\Gamma_1$,
$\Sigma_1$, $r_1$, $\Omega_1$  \\
Section \ref{subsec-e0nearedgestratum}: $A_2$, $\widetilde{A}_2$,
$S_2$, $\widetilde{S}_2$, $\Omega_2$, $\Gamma_2$,
$\Sigma_2$, $r_2$, $\Omega_2$\\
Section \ref{subsec-e0nearslimstratum}: $A_3$, $\widetilde{A}_3$,
$S_3$, $\widetilde{S}_3$, $\Omega_3$, $\Gamma_3$,
$\Sigma_3$, $r_3$, $\Omega_3$\\
Section \ref{sec-adjusting}: $c_{\adjust}$ \\
Section \ref{subsecadjust2}: $W_1^0$, $\Xi_1$, $\psi_1$,
$\Psi_1$, $\Omega_1'$,
$\e^1$, $c_{\twostratum}$ \\
Section \ref{subsecadjustingedge}: $W_2^0$, $\Xi_2$, $\psi_2$,
$\Psi_2$, $\Omega_2'$,
$\e^2$, $c_{\edge}$ \\
Section \ref{subsecadjustslim}: $W_3^0$, $\Xi_3$, $\psi_3$,
$\Psi_3$, $\Omega_3'$,
$\e^3$, $c_{\slim}$ \\
Section \ref{subsecproofofprop}: $W_1$, $W_2$, $W_3$ \\
Section \ref{subsecdefzeroball}: $M^{\zeroball}$, $M_1$ \\
Section \ref{subsecdefslim}: $W_3'$, $U_3'$, $W_3''$, $M^{\slim}$, $M_2$ \\
Section \ref{subsecdefedge}: $W_2'$, $U_2'$, $W_2''$, $M^{\edge}$, $M_3$ \\
Section \ref{subsecdef2}: $W_1'$, $U_1'$, $W_1''$, $M^{\twostratum}$ \\
Section \ref{sec-proof}: $r_\partial$, $H_\partial$, $M_i^{\partial}$ \\
Appendix \ref{sec-cloudy}:  $S$, $\widetilde{S}$, $r(\cdot)$, $W$

\section{Preliminaries} \label{preliminaries}

We refer to \cite{Burago-Burago-Ivanov} for basics about
length spaces and Alexandrov spaces.

\subsection{Pointed Gromov-Hausdorff approximations}

\begin{definition} \label{def3}
Let $(X, \star_X)$ be a pointed metric space. 
Given $\de \in [0, \infty)$, two closed subspaces $C_1$ and $C_2$  are
{\em $\de$-close in the pointed Hausdorff sense} if
$C_1 \cap \overline{B(\star_X, \delta^{-1})}$ and 
$C_2 \cap \overline{B(\star_X, \delta^{-1})}$
have Hausdorff distance at most $\delta$.
\end{definition}

If $X$ is complete and proper (i.e. closed bounded
sets are compact) then the corresponding pointed Hausdorff topology,
on the set of closed subspaces of $X$, is 
compact and metrizable.

We now recall some definitions and basic results about the 
pointed Gromov-Hausdorff
topology  \cite[Chapter 8.1]{Burago-Burago-Ivanov}.

\begin{definition}
\label{def-hausdorffapproximation}
Let $(X,\star_X)$ and $(Y,\star_Y)$ be pointed metric spaces.
Given 
$\de\in [0,1)$,
 a pointed map $f:(X,\star_X)\ra (Y,\star_Y)$ is a 
{\em $\de$-Gromov-Hausdorff approximation} if for every 
$x_1,x_2\in B(\star_X,\de^{-1})$
and 
$y\in B(\star_Y,\de^{-1}-\de)$,
we have
\begin{equation}
|d_Y(f(x_1),f(x_2))-d_X(x_1,x_2)|\leq \de\quad \mbox{and}\quad
d_Y(y,f(B(\star_X,\de^{-1})))\leq \de\,.
\end{equation}
Two pointed metric spaces $(X,\star_X)$ and $(Y,\star_Y)$ are 
{\em $\de$-close in the pointed Gromov-Hausdorff topology}
(or $\de$-close for short) if
there is a $\de$-Gromov-Hausdorff approximation from
 $(X,\star_X)$ to $(Y,\star_Y)$.
We note that this does not define a metric space structure 
on the set of pointed metric spaces, but nevertheless defines a topology
which happens to be metrizable. 
A sequence $\{(X_i,\star_i)\}_{i=1}^\infty$
of pointed metric spaces
{\em Gromov-Hausdorff converges} to $(Y,\star_Y)$ if
there is a sequence $\{f_i:(X_i,\star_{X_i})\ra (Y,\star_Y)\}_{i=1}^\infty$
of $\de_i$-Gromov-Hausdorff approximations, where $\de_i\ra 0$.
  We will denote this by 
$(X_i,\star_{X_i})\stackrel{\gh}{\ra}(Y,\star_Y)$.

\end{definition} 
Note that a $\de$-Gromov-Hausdorff approximation 
is a $\de'$-Gromov-Hausdorff approximation for every $\de'\geq \de$.
A $\de$-Gromov-Hausdorff approximation $f$ has a {\em
quasi-inverse}
$\widehat{f} : (Y, \star_Y) \ra (X, \star_X)$ constructed by saying that
for $y \in B(\star_Y, \delta^{-1} - \delta)$, we choose some
$x \in B(\star_X, \delta^{-1})$ with $d_Y(y, f(x)) \le \delta$ and
put $\widehat{f}(y) = x$. There is a function $\delta^\prime = 
\delta^\prime(\delta) > 0$ with $\lim_{\de \ra 0} \delta^\prime = 0$
so that if $f$ is a $\de$-Gromov-Hausdorff approximation then
$\widehat{f}$ is a $\de^\prime$-Gromov-Hausdorff approximation
and $\widehat{f} \circ f$ (resp. $f \circ \widehat{f}$) 
is $\delta^\prime$-close
to the identity on $B(\star_X, (\delta^\prime)^{-1})$ 
(resp. $B(\star_Y, (\delta^\prime)^{-1})$).
The condition $(X_i,\star_{X_i})\stackrel{\gh}{\ra}(Y,\star_Y)$
is equivalent to the existence of a sequence 
$\{f_i: (Y,\star_Y)\ra(X_i,\star_{X_i})\}_{i=1}^\infty$
(note the reversal of domain and target)
of $\de_i$-Gromov-Hausdorff approximations, where $\de_i\ra 0$.

The relation of being $\de$-close is not symmetric. However,
this does not create a problem because only the associated notion of
convergence (i.e. the topology) plays a role in our discussion.

The pointed Gromov-Hausdorff
topology is a complete metrizable topology on the set of complete 
proper metric spaces
(taken modulo pointed isometry).  Hence we can talk about two 
such metric spaces
being having distance at most $\de$ from each other. 
There is a well-known criterion for a set of
pointed metric spaces to be precompact in the pointed Gromov-Hausdorff topology
\cite[Theorem 8.1.10]{Burago-Burago-Ivanov}. 
Complete proper {\em length spaces}, which are the main 
interest of this paper, form a closed subset of the set of
complete proper metric spaces.

\subsection{$C^K$-convergence}

\begin{definition}
\label{def-smoothapproximation}
Given $K \in \Z^+$,
let $(M_1, \star_{M_1})$ and $(M_2, \star_{M_2})$ be complete
pointed 
$C^{K}$-smooth Riemannian manifolds.
(That is, the manifold
transition maps are $C^{K+1}$ and the metric in local 
coordinates is $C^{K}$).
Given $\de\in [0,\infty)$, a pointed $C^{K+1}$-smooth map 
$f:(M_1, \star_{M_1})\ra 
(M_2, \star_{M_2})$ is a 
{\em $\de$-$C^{K}$ approximation} if it is a 
$\de$-Gromov-Hausdorff approximation and
the $C^K$-norm of 
$f^* g_{M_2} - g_{M_1}$, computed on 
$B(\star_{M_1}, \delta^{-1})$, is bounded above by $\delta$.
Two $C^K$-smooth Riemannian manifolds $(M_1, \star_{M_1})$ and 
$(M_2, \star_{M_2})$ are {\em $\de$-$C^K$ close} if
there is a $\de$-$C^K$ approximation from $(M_1, \star_{M_1})$ 
to $(M_2, \star_{M_2})$.
\end{definition} 

In what follows, we will always
take $K \ge 10$. We now state a $C^K$-precompactness result.

\begin{lemma} \label{smoothprecompactness} (cf. \cite[Chapter 10]{Petersen})
Given $v, r > 0$, $n \in \Z^+$ and 
a function $A : (0, \infty) \rightarrow (0, \infty)$,
the set of
complete pointed $C^{K+2}$-smooth $n$-dimensional Riemannian manifolds 
$(M,\star_M)$
such that 
\begin{enumerate}
\item $\vol(B(\star_M, r)) \ge v$ and 
\item $|\nabla^k \Rm| \le A(R)$ on
$B(\star_M, R)$, for all $0 \le k \le K$ and $R >0$,
\end{enumerate}
is precompact in the pointed 
$C^{K}$-topology.
\end{lemma}

The bounds on the derivatives of curvature in Lemma \ref{smoothprecompactness}
give uniform $C^{K+1}$-bounds on the Riemannian metric in harmonic
coordinates.  One then obtains limit metrics which are
$C^K$-smooth.  One can get improved regularity but we will not need it.

\subsection{Alexandrov spaces}
\label{subsec-alexandrov}

Recall that there is a notion of 
an Alexandrov
space of curvature at least $c$, or equivalently 
a complete length space $X$ having
curvature bounded below by $c \in \R$ on an open set $U \subset X$
\cite[Chapter 4]{Burago-Burago-Ivanov}.  
In this paper we will only be concerned with  Alexandrov
spaces of finite Hausdorff dimension, 
so this will be assumed implicitly without further
mention.   

We will also have occasion to work with 
incomplete, but locally complete spaces.  This situation typically arises
when one has a metric space $X$ where $X=B(p,r)$, and every closed ball $\ol{B(p,r')}$ with $r'<r$ is complete.
The version of Toponogov's theorem for Alexandrov spaces 
\cite[Chapter 10.3]{Burago-Burago-Ivanov},
in which one deduces global triangle comparison inequality from local ones,
also applies in the incomplete situation, provided that the geodesics
arising in the proof lie in an {\it a priori} complete part of the space.
In particular, if all sides of a geodesic triangle have
length $<D$ then triangle comparison is valid provided that the 
closed balls of radius $2D$ centered at the vertices are complete.

We recall the notion of a strainer (cf. 
\cite[Definition 10.8.9]{Burago-Burago-Ivanov}).

\begin{definition}
Given a point $p$ in an Alexandrov space $X$ of curvature
at least $c$, an {\em $m$-strainer at $p$ of quality
$\delta$ and scale $r$} is a collection
$\{(a_i, b_i)\}_{i=1}^m$ of pairs of points 
such that $d(p,a_i) = d(p, b_i) = r$ and in terms of comparison angles,
\begin{align}
\cangle_p(a_i,b_i) & > \pi - \delta, \\
\cangle_p(a_i,a_j) & > \frac{\pi}{2} - \delta, \notag \\
\cangle_p(a_i,b_j) & > \frac{\pi}{2} - \delta, \notag \\
\cangle_p(b_i,b_j) & > \frac{\pi}{2} - \delta \notag
\end{align}
for all $i,j \in \{1, \ldots, m\}$, $i \neq j$.  Note that the
comparison angles are defined using comparison triangles in the 
model space of constant curvature $c$.
\end{definition}

For facts about strainers, we refer to
\cite[Chapter 10.8.2]{Burago-Burago-Ivanov}. 
The Hausdorff dimension of $X$ equals its strainer number,
which is defined as follows.

\begin{definition}
The {\em strainer number of $X$} is the supremum of numbers $m$ such that
there exists an $m$-strainer of quality $\frac{1}{100m}$ at some point
and some scale.
\end{definition}

By ``dimension of $X$'' we will mean the Hausdorff dimension; this
coincides with its topological dimension, although we will not need this
fact.
If $(X, \star_X)$ is a pointed nonnegatively curved
Alexandrov space then there is a 
pointed Gromov-Hausdorff limit
$C_T X = \lim_{\lambda \rightarrow \infty} \left(
\frac{1}{\lambda} X, \star_X \right)$
called the {\em Tits cone} of $X$. 
It is a nonnegatively curved Alexandrov space 
which is a metric cone, as defined in
Subsection \ref{notation}.
We will consider Tits cones in the special case when $X$ is a 
nonnegatively curved Riemannian manifold.

A {\em line} in a length space $X$ is a curve $\gamma : \R \rightarrow X$
with the property that $d_X(\gamma(t), \gamma(t')) = |t-t'|$ for all
$t,t' \in \R$.
The splitting theorem (see \cite[Chapter 10.5]{Burago-Burago-Ivanov})
says that if a nonnegatively curved Alexandrov space $X$ contains a line then it
is isometric to $\R \times Y$ for some nonnegatively
curved Alexandrov space $Y$.

If $\gamma \: : \: [0, T] \rightarrow X$ is a minimal geodesic
in an Alexandrov space $X$, parametrized by arc-length, and $p \neq
\gamma(0)$ is a point in $X$ then the function $t \rightarrow
d_p(\gamma(t))$ is right-differentiable and
\begin{equation}
\lim_{t \rightarrow 0^+} \frac{d_p(\gamma(t)) - d_p(\gamma(0))}{t} =
- \cos \th,
\end{equation}
where $\th$ is the minimal angle between $\gamma$
and minimizing geodesics from $\gamma(0)$ to $p$
\cite[Corollary 4.5.7]{Burago-Burago-Ivanov}.

The proof of the next lemma is similar to that of
\cite[Theorem 10.7.2]{Burago-Burago-Ivanov}.

\begin{lemma} \label{alexlimit}
Given $n \in \Z^+$, 
let $\{(X_i, \star_{X_i})\}_{i=1}^\infty$ be a sequence of complete pointed
length spaces.
Suppose that $c_i \ra 0$ and $r_i \ra \infty$ are positive sequences such that
for each $i$, the ball
$B(\star_{X_i}, r_i)$ has curvature bounded below by $- c_i$
and dimension bounded above by $n$. Then
a subsequence of the $(X_i, \star_{X_i})$'s converges in the pointed Gromov-Hausdorff
topology to a pointed nonnegatively curved 
Alexandrov space of dimension at most $n$.
\end{lemma}

\subsection{Critical point theory} 
We recall a few facts about critical point theory here,
and refer the reader to \cite{Grove-Shiohama,Cheeger} for 
more information.

If $M$ is a complete Riemannian manifold and $p\in M$,
then a point $x\in M\setminus \{p\}$ is {\em noncritical}
if there is a nonzero vector $v\in T_xM$ making an angle
strictly larger than $\frac{\pi}{2}$ with the initial
velocity of every minimizing segment from $x$ to $p$.  If there are 
no critical points in the set $d_p^{-1}(a,b)$ then 
the level sets $\{d_p^{-1}(t)\}_{t\in(a,b)}$,
are pairwise isotopic Lipschitz
hypersurfaces, and $d_p^{-1}(a,b)$ is diffeomorphic to a 
product, as in the usual Morse lemma for smooth functions.
As with the traditional Morse lemma,
the proof proceeds by constructing a smooth vector field $\xi$
such that $d_p$ has uniformly positive directional derivative
in the direction $\xi$.

\subsection{Topology of nonnegatively curved 3-manifolds}

In this subsection we describe the topology of certain
nonnegatively curved manifolds.  We start with $3$-manifolds.

\begin{lemma} \label{topology}
Let $M$ be a complete connected orientable $3$-dimensional 
$C^K$-smooth
Riemannian manifold
with nonnegative sectional curvature. 
We have the following classification of the diffeomorphism
type of $M$, based on the number of ends : 
\begin{itemize}
\item
$0$ ends:\quad 
$S^1 \times S^2$, $S^1 \times_{\Z_2} S^2 =
\R P^3 \# \R P^3$,
$T^3/\Gamma$ (where $\Gamma$ is a finite subgroup of
$\Isom^+(T^3)$ which acts freely on $T^3$) or $S^3 / \Gamma$
(where $\Gamma$ is a finite 
subgroup
of $\SO(4)$ that acts freely on $S^3$). 
\item
$1$ end: \quad
$\R^3$, 
$S^1 \times \R^2$,
$S^2\times_{\Z_2} \R = \R P^3 - B^3$
or $T^2 \times_{\Z_2} \R = $ a flat $\R$-bundle over the Klein bottle. 
\item
$2$ ends:\quad
$S^2 \times \R$ or $T^2 \times \R$.
\end{itemize}

If $M$ has two ends then it splits off an $\R$-factor isometrically.
\end{lemma}
\begin{proof}
If $M$ has no end then it is compact and 
the result 
follows for $C^\infty$-metrics from \cite{Hamiltonnnn}.
For $C^K$-smooth metrics, one could adapt the argument in \cite{Hamiltonnnn}
or alternatively use \cite{Simon}.

If $M$ is noncompact then the 
Cheeger-Gromoll soul theorem says that $M$ is diffeomorphic
to the total space of a vector bundle over its soul, a closed 
lower-dimensional manifold with nonnegative sectional curvature
\cite{Cheeger-Gromoll}. 
(The proof in \cite{Cheeger-Gromoll}, which is for $C^\infty$-metrics,
goes through without change for $C^K$-smooth metrics.)
The possible dimensions of the soul are 0, 1 and 2. The
possible topologies of $M$ are listed in the lemma. 

If $M$ has two ends then it contains a line and
the Toponogov splitting theorem implies that $M$ splits off
an $\R$-factor isometrically.
\end{proof}

We now look at a pointed nonnegatively curved surface and
describe the topology of a ball in it 
which is pointed Gromov-Hausdorff close to an interval.

\begin{lemma} \label{surfacelemma}
Suppose that $(S,\star_S)$ is a pointed $C^K$-smooth 
nonnegatively curved complete orientable
Riemannian $2$-manifold. Let $\star_S\in S$ be a basepoint and
suppose that
the pointed ball $(B(\star_S,10),\star_S)$ has pointed Gromov-Hausdorff 
distance at most
$\delta$
from the pointed interval $([0, 10], 0)$.
\begin{enumerate}
\item Given $\th > 0$ there is some 
$\overline{\delta}(\th) > 0$
so that if 
$\delta<\overline{\delta}(\th)$
then for every $x\in 
\overline{B(\star_S, 9)} - B(\star_S, 1)$ the set $V_x$
of initial velocities of minimizing geodesic segments from
$x$ to $\star_S$ has diameter bounded above by $\th$.
\item 
There is some 
$\overline{\delta} > 0$ so that if 
$\delta<\overline{\delta}$
then for every $r\in [1,9]$
the ball  $\overline{B(\star_S, r)}$ is
homeomorphic to a closed $2$-disk.
\end{enumerate}
\end{lemma}
\begin{proof}
(1).
Choose a point $x^\prime$ with
$d_S(\star_S, x^\prime) = 9.5$. Fix a
minimizing geodesic $\gamma^\prime$ from $x$ to $x^\prime$
and a minimizing geodesic $\gamma^{\prime \prime}$ from
$\star_S$ to $x^\prime$. 
If $\gamma$ is a minimizing geodesic from $\star_S$ to $x$,
consider the geodesic triangle with edges $\gamma$, 
$\gamma^\prime$ and $\gamma^{\prime \prime}$. As
\begin{equation}
d(\star_S, x) + d(x, x^\prime) - d(\star_S, x^{\prime \prime})
\le \const \delta,
\end{equation}
triangle comparison implies that
the angle at $x$ between $\gamma$ and $\gamma^\prime$ is
bounded below by $\pi - a(\delta)$, where $a$ is a 
positive monotonic function with
$\lim_{\delta \rightarrow 0} a(\delta) = 0$. We take $\overline{\delta}$
so that $2a(\overline{\delta}) \le \theta$,

(2).  Suppose that 
$\delta<\overline{\delta}(\frac{\pi}{4})$.
By critical point theory, the distance function
$d_{\star_S}: A(\star_S,1,9) \ra [1,9]$ is 
a fibration with fibers diffeomorphic to a disjoint union of circles.
In particular, the closed
balls $\overline{B(\star_S,r)}$, for $r\in [1,9]$,
are pairwise homeomorphic.  When 
$\delta\ll 1$,
the fibers will be connected, since the diameter of
$d_{\star_S}^{-1}(5)$ will be comparable to $\delta$.
Hence $\overline{B(\star_S,1)}$
is homeomorphic to a surface with circle boundary.

Suppose that $\overline{B(\star_S, 1)}$ is not homeomorphic to a disk.
A complete connected orientable nonnegatively curved surface
is homeomorphic to $S^2$, $T^2$, $\R^2$, or $S^1\times \R$.
By elementary topology, the
only possibility is if 
$S$ is homeomorphic
to a $2$-torus, $\overline{B(\star_S, 1)}$ is homeomorphic to
the complement of a $2$-ball in $S$, and 
$S - B(\star_S,2)$ is homeomorphic to a disk. 
In this case, $S$ must be flat. However, 
the cylinder $A(\star_S,1,2)$ lifts to the
universal cover of $S$, which is isometric to the flat
$\R^2$. If 
$\delta$
is sufficiently small then the flat $\R^2$ would contain a 
metric ball of radius $\frac{1}{10}$ which is Gromov-Hausdorff close to an
interval, giving a contradiction. 
\end{proof}

\subsection{Smoothing Lipschitz functions}
\label{subsec-smoothinglipschitz}

The technique of smoothing Lipschitz functions was introduced in
Riemannian geometry by Grove and Shiohama \cite{Grove-Shiohama}.

If $M$ is a Riemannian manifold and $F$ is a Lipschitz function on $M$
then the generalized gradient of $F$ at $m \in M$ can be defined as follows.
Given $\epsilon  \in (0, \inj_m)$, 
if $x \in B(m, \epsilon)$ is a point of differentiability of $F$
then compute $\nabla_x F \in T_x M$ and parallel transport it along the 
minimizing
geodesic to $m$. Take the closed
convex hull of the vectors so obtained and then take the intersection as 
$\epsilon \ra 0$.
This gives a closed convex subset of $T_mM$, which is the generalized
gradient of $F$ at $m$ \cite{Clarke}; 
we will denote this set by
$\nabla_m^{\gen}F$.
The union $\bigcup_{m\in M}\,\nabla_m^{\gen} F\subset TM$ will be 
denoted $\nabla^{\gen}F$.

\begin{lemma} \label{lem-generalsmoothing}
Let  $M$ be a complete Riemannian manifold and let
$\pi : TM \rightarrow M$ be the projection map. 
Suppose that
$U\subset M$ is an open set, $C\subset U$ is a compact subset
and $S$
is an open fiberwise-convex subset of $\pi^{-1}(U)$.

Then for every $\eps>0$
and any Lipschitz function $F:M\ra \R$ whose generalized gradient
over $U$ lies in $S$, there is a Lipschitz function $\widehat F:M\ra \R$
such that:
\begin{enumerate}
\item $\widehat F$ is $C^\infty$ on an open set containing $C$.
\item  The generalized gradient of $\widehat{F}$, over $U$, lies in $S$.
(In particular, at every point in $U$ where $\widehat F$ is differentiable,
the gradient lies in $S$.)
\item $|\widehat{F} - F|_\infty \le \epsilon$.
\item $\widehat{F} \big|_{M-U} = F \big|_{M-U}$. 
\end{enumerate}
\end{lemma}

The proof of Lemma \ref{lem-generalsmoothing} proceeds by mollifying the
Lipschitz function $F$, as in \cite[Section 2]{Grove-Shiohama}. 
We omit the details.

\begin{corollary} \label{Lipsmoothing}
Suppose that $M$ is a compact Riemannian manifold. Given
$K < \infty$ and $\epsilon > 0$,
for any $K$-Lipschitz function $F$ on $M$
there is a $(K + \epsilon)$-Lipschitz function 
$\widehat{F} \in C^\infty(M)$ 
with  $|\widehat{F} - F|_{\infty} \le \epsilon$. 
\end{corollary}
\begin{proof}
Apply Lemma \ref{lem-generalsmoothing} with $C = U = M$, 
and
$S \: = \: \{v \in TM \: : \: |v| < K + \epsilon\}$. 
\end{proof}

\begin{corollary}\label{cor-1strained}
For all $\eps>0$ there is a $\th>0$ with the following
property.  

Let
$M$ be a complete Riemannian manifold, let $Y\subset M$
be a closed subset and let $d_Y:M\ra\R$ be the distance function
from $Y$. Given $p\in M - Y$,
let $V_p\subset T_pM$ be the set 
of initial  velocities  of minimizing
geodesics from $p$ to $Y$.  Suppose that
$U\subset M - Y$ is an open subset such that for all
$p \in U$, one has $\diam(V_p) <\th$. Let $C$ be a
compact subset of $U$.
Then for every $\eps_1>0$
there is a Lipschitz function $\widehat{F}:M\ra \R$ such that
\begin{enumerate}
\item $\widehat{F}$ is smooth on a neighborhood of $C$.
\item 
$\|\widehat{F}-d_Y\|_\infty\,<\,\eps_1$.
\item $\widehat{F} \restr_{M - U}=d_Y\restr_{M - U}$.
\item For every $p\in C$, the angle between $-(\nabla\widehat{F})(p)$
and $V_p$ is at most $\eps$.
\item $\widehat F-d_Y$ is  $\eps$-Lipschitz.
\end{enumerate}
\end{corollary}
\begin{proof}
First, note that if $p\in M - Y$ is a point of differentiability
of the distance function $d_Y$ then $\nabla_p\, d_Y=-V_p$.  Also,
the assignment $x\mapsto V_x$ is semicontinuous in the sense that
if $\{x_k\}_{k=1}^\infty$ is a sequence of points converging
to $x$ then by parallel transporting $V_{x_k}$ radially
to the fiber over $x$, we obtain a sequence 
$\{\bar V_{x_k}\}_{k=1}^\infty \subset T_xM$ which
accumulates on a subset of $V_x$.  It follows that the generalized
derivative of $d_Y$ at any point $p\in M - Y$ is precisely 
$-\convexhull(V_p)$,
where $\convexhull(V_p)$ denotes the convex hull of $V_p$.

Put $S' = \bigcup_{p \in U} \convexhull(-V_p)$.
Then $S'$ is a relatively closed fiberwise-convex
subset of $\pi^{-1}(U)$, with fibers of diameter less than $\th$.
We can fatten $S'$ slightly to form an open
fiberwise-convex set $S\subset \pi^{-1}(U)$
which contains $S'$, with fibers of diameter less than $2\th$. 

Now take $\th<\frac{\eps}{2}$ and
apply Lemma \ref{lem-generalsmoothing} to 
 $F=d_Y:M\ra \R$, with $S$ as in the preceding paragraph.
The resulting function $\widehat F:M\ra \R$ clearly satisfies
(1)-(4).  
   To see that (5) holds, note that if $p\in U$ is a point
of differentiability of both $\widehat F$ and $F$ then 
$\nabla\widehat F(p)$ and $\nabla F(p)$ both lie 
in the fiber $S\cap T_pM$, which has diameter less than $2\th$.
Hence the gradient of the difference satisfies
\begin{equation}
\|\nabla (\widehat F-F)(p)\|=\|\nabla\widehat F(p)-\nabla F(p)\|<2\th<\eps\,.
\end{equation}
Since $\widehat F$
coincides with $F$ outside $U$, this implies that
$\widehat F-F$ is $\eps$-Lipschitz.
\end{proof}

\begin{remark}
When we apply Corollary \ref{cor-1strained}, the hypotheses will be
verified using triangle comparison.
\end{remark}

\section{Splittings, strainers, and adapted coordinates}

This section is about the notion of a pointed metric space approximately
splitting off an $\R^k$-factor.  We first define an approximate
$\R^k$-splitting, along with the notion of compatibility between an
approximate 
$\R^k$-splitting
and an approximate $\R^j$-splitting.
We prove basic properties about approximate splittings.
In the case of a pointed Alexandrov space, we show that having an
approximate $\R^k$-splitting is equivalent to having a good
$k$-strainer. We show that if there is not an approximate
$\R^{k+1}$-splitting at a point $p$ then any approximate $\R^k$-splitting
at $p$ is nearly-compatible with any approximate $\R^j$-splitting at $p$,
for $j \le k$.

We then introduce the notion of coordinates adapted to an approximate
$\R^k$-splitting, in the setting of Riemannian manifolds with
a lower curvature bound, proving existence and (approximate) uniqueness
of such adapted coordinates.

\subsection{Splittings}

We start with the notion of a splitting.

\begin{definition}
\label{def-productstructure}
A product structure on a metric space $X$ is an isometry
$\al : X \ra X_1\times X_2$.  A {\em $k$-splitting}
of  $X$ is a product structure $\al : X \ra X_1\times X_2$ where
$X_1$ is isometric to $\R^k$. A {\em splitting} is a $k$-splitting
for some $k$.  Two $k$-splittings $\al : X \ra X_1\times X_2$ and
$\be : X \ra Y_1\times Y_2$ are {\em equivalent} if there are isometries
$\phi_i:X_i\ra Y_i$ such that $\be= (\phi_1,\phi_2) \circ \al$.
\end{definition}

In addition to equivalence of splittings, we can talk about compatibility
of splittings.
\begin{definition}
\label{def-compatiblesplitting}
Suppose that $j\leq k$.  A $j$-splitting $\al : X \ra X_1\times X_2$ 
is {\em compatible} with a $k$-splitting  
$\be : X \ra Y_1\times Y_2$ if there
is a $j$-splitting 
$\phi : Y_1 \ra \R^j\times\R^{k-j}$
such that $\al$ is equivalent to the $j$-splitting given by the
composition
\begin{equation}
X \stackrel{\beta}{\ra} Y_1 \times Y_2 \stackrel{(\phi, \Id)}{\ra}  
 (\R^j\times \R^{k-j})\times Y_2  \cong \R^j\times(\R^{k-j}\times Y_2).
\end{equation}
\end{definition}

\begin{lemma}
\label{lem-splitcompatible}
\mbox{}
\begin{enumerate}
\item
Suppose $\al:X\ra \R^k\times Y$ is a $k$-splitting of a metric
space $X$, and $\be:X\ra \R\times Z$
is a $1$-splitting.  Then either $\be$ is compatible with $\al$,
or there is a $1$-splitting $\ga:Y\ra \R\times W$ such that
$\be$ is compatible with the induced splitting
$X\ra (\R^k\times \R)\times W$.
\item Any two splittings of a metric space are compatible with a
third splitting.
\end{enumerate}
\end{lemma}

Before proving this, we need a sublemma.

Recall that a {\em line} in a metric space
is a globally minimizing complete geodesic, i.e. an isometrically
embedded copy of $\R$.  We will say that two lines are {\em parallel}
if their union is isometric to the union of two parallel lines in
$\R^2$.

\begin{sublemma}
\label{lem-linetrivia}
\mbox{}
\begin{enumerate}
\item A path $\ga:\R \ra X_1\times X_2$ in a product is a constant
speed geodesic if and only if the compositions $\pi_{X_i}\circ \ga:\R\ra X_i$
are constant speed geodesics.
\item Two lines in a metric space $X$ are parallel
if and only if they have 
constant speed parametrizations
$\ga_1:\R\ra X$ and $\ga_2:\R\ra X$ such that $d^2(\ga_1(s),\ga_2(t))$ is a 
quadratic
function of $(s-t)$.
\item If two lines $\ga_1$, $\ga_2$ 
in a product $\R^k\times X$ are parallel then either
$\pi_X(\ga_1),\,\pi_X(\ga_2)\subset X$
are parallel lines, or they are both points.
\item Suppose $\L$ is a collection of lines in a metric space $X$.
If $\bigcup_{\ga\in \L}\,\ga=X$, and every pair  $\ga_1,\ga_2\in\L$ 
is parallel,
then there is a $1$-splitting $\al:X\ra \R\times Y$ such that
$\L=\{\al^{-1}(\R\times \{y\})\}_{y\in Y}$.
\end{enumerate}
\end{sublemma}
\begin{proof}
(1).  It follows from the Cauchy-Schwarz inequality that
if $a,b,c\in X_1\times X_2$ satisfy the triangle
equation $d(a,c)=d(a,b)+d(b,c)$ then the same is true of their
projections $a_1,b_1,c_1\in X_1$ and $a_2,b_2,c_2\in X_2$, and
moreover 
$(d(a_1,b_1), d(a_2,b_2))$
and $(d(a_1,c_1), d(a_2,c_2))$ are linearly dependent in $\R^2$.  
This implies (1).

(2).  The parallel lines $y=0$ and $y=a$ in $\R^2$ can
be parametrized by $\gamma_1(s) = (s,0)$ and $\gamma_2(t) = (t,a)$,
with $d^2(\ga_1(s),\ga_2(t)) = (s-t)^2 + a^2$. Conversely,
suppose that lines $\gamma_1$ and $\gamma_2$ in a metric space
are such that $d^2(\ga_1(s),\ga_2(t))$ is quadratic in $(s-t)$.
After affine changes of $s$ and $t$, we can assume that
$d^2(\ga_1(s),\ga_2(t)) = (s-t)^2 + a^2$ for some $a \in \R$.
Then the union of $\gamma_1$ and $\gamma_2$ is isometric to the
union of the lines $y=0$ and $y=a$ in $\R^2$.

(3). By (2), we may assume that for $i\in \{1,2\}$
there are constant 
speed parametrizations $\ga_i:\R\ra \R^k\times X$, such that
$d^2(\ga_1(s),\ga_2(t))$ is a quadratic function of $(s-t)$.
The projections $\pi_{\R^k}\circ\ga_i$ are constant speed geodesics
in $\R^k$, and the quadratic function
$d^2(\pi_{\R^k}\circ\ga_1(s),\pi_{\R^k}\circ\ga_2(t))$ is
a function of $(s-t)$; otherwise 
$d^2(\pi_{\R^k}\circ\ga_1(s),\pi_{\R^k}\circ\ga_2(s))$ would be
unbounded in $s$, which contradicts that 
$d^2(\ga_1(s),\ga_2(s))$ is constant in $s$.
Therefore 
\begin{equation}
d^2(\pi_X\circ\ga_1(s),\pi_X\circ\ga_2(t))
=d^2(\ga_1(s),\ga_2(t))-d^2(\pi_{\R^k}\circ\ga_1(s),\pi_{\R^k}\circ\ga_2(t))
\end{equation}
is a quadratic
function of $(s-t)$.  By (2) we conclude that
$\pi_X\circ \ga_1$, $\pi_X\circ\ga_2$ are parallel.

(4). Let $\ga:\R\ra X$ be a unit speed parametrization of
some line in $\L$, and let $b:X\ra \R$ be the Busemann
function $\lim_{t\ra\infty}\;d(\ga(t),\cdot)-t$.  By assumption,
the elements of $\L$ partition $X$ into the cosets of an 
equivalence relation; the quotient $Y$ inherits a natural metric,
namely the Hausdorff distance.  The  map $(b,\pi_Y):X\ra\R\times Y$
defines a $1$-splitting -- one verifies that it is an isometry
using the fact that $\L$ consists of parallel lines.
\end{proof}

\bigskip

\begin{proof}[Proof of Lemma \ref{lem-splitcompatible}]
\mbox{}

(1). Consider the collection of lines
$\L_{\be}=\{\be^{-1}(\R\times\{z\})\mid z\in Z\}$.
By Lemma \ref{lem-linetrivia}(3) it follows that $\pi_Y\circ\al(\L_{\be})$
consists of parallel lines, or consists entirely of points.

{\em Case 1. $\pi_Y\circ\al(\L_{\be})$
consists of points.}  In this case, 
Sublemma \ref{lem-linetrivia}
implies that $\pi_{\R^k}\circ\al(\L_{\be})$ is a family of
parallel lines in $\R^k$.
Decomposing $\R^k$ into a product $\R^k\simeq \R^1\times \R^{k-1}$
in the direction defined by $\pi_{\R^k}\circ\al(\L_{\be})$, we
obtain a $1$-splitting
of $X$ which is easily seen to be equivalent
to $\be$.

{\em Case 2. $\pi_Y\circ\al(\L_{\be})$
consists of parallel lines.}  Since $\bigcup_{\ga\in \pi_Y\circ\al(\L_{\be})}
\;\ga=Y$, by Lemma \ref{lem-linetrivia}(4), there is a 
$1$-splitting $\ga:Y\ra \R\times W$ such that
$\{\ga^{-1}(\R\times\{w\})\mid w\in W\}
=\pi_Y\circ\al(\L_{\be})$.   Letting $\al':X\ra (\R^k\times\R)\times W$
be the $(k+1)$-splitting given by 
$X\stackrel{\al}{\ra}\R^k\times Y\ra (\R^k\times\R)\times W$,
we get that $\pi_W\circ \al'(\L_{\be})$ consists of points,
so by Case 1 it follows that $\be$ is compatible with $\al'$.

(2). Suppose that $\al:X\ra \R^k\times Y$ and $\be:X\ra \R^l\times Z$
are splittings.  We may apply part (1)
to $\al$ and the $1$-splitting obtained from the $i^{th}$
coordinate direction of $\be$, for successive values
of $i\in \{1,\ldots,l\}$.   This will enlarge $\al$ to a
splitting $\al'$ which is compatible with all of these  $1$-splittings,
and clearly $\al'$ is then compatible with $\be$.
\end{proof}

\subsection{Approximate splittings}

Next, we consider approximate splittings.

\begin{definition} \label{def-ghsplitting}
Given $k\in \Z^{\ge 0}$ and $\de\in [0,\infty)$,
a {\em $(k,\de)$-splitting} of a pointed
metric space $(X, \star_X)$ 
is a $\de$-Gromov-Hausdorff approximation 
$(X, \star_X) \ra (X_1, \star_{X_1}) \times (X_2, \star_{X_2})$, 
where $(X_1, \star_{X_1})$ is isometric to $(\R^k, \star_{\R^k})$.
(We allow $\R^k$ to have other basepoints than $0$.)
\end{definition}

There are ``approximate'' versions of equivalence and compatibility
of splittings.

\begin{definition}
\label{def-epscompatible}
Suppose  that $\al :  (X,\star_X) \ra (X_1, \star_{X_1}) \times 
(X_2, \star_{X_2})$ is a $(j,\de_1)$-splitting
and
$\be :  (X,\star_X) \ra (Y_1, \star_{Y_1}) \times (Y_2, \star_{Y_2})$
is an  $(k,\de_2)$-splitting.   Then
\begin{enumerate}
\item {\em $\al$ is $\eps$-close to $\be$} if $j=k$ and there are $\eps$-Gromov-Hausdorff
approximations $\phi_i:(X_i, \star_{X_i})\ra (Y_i,\star_{Y_i})$ such that the composition
$(\phi_1,\phi_2) \circ \al$ is $\epsilon$-close to $\be$, i.e. agrees with $\be$ on 
$B(\star_{X} ,\eps^{-1})$ up to
error at most $\eps$.  
\item {\em $\al$ is $\eps$-compatible with $\be$} if $j\leq k$ and 
there is a $j$-splitting $\ga:  (Y_1,\star_{Y_1}) \ra (\R^j, \star_{\R^j}) \times 
(\R^{k-j}, \star_{\R^{k-j}})$
such that the $(j,\de_2)$-splitting defined by the composition
\begin{equation}
X \stackrel{\be}{\ra} Y_1\times Y_2 \stackrel{(\ga, \Id)}{\ra} (\R^j\times \R^{k-j})\times Y_2
\cong \R^j\times (\R^{k-j}\times Y_2)
\end{equation}
is $\eps$-close to $\al$. 
\end{enumerate}
\end{definition}

\begin{lemma} \label{faralignment1}
Given $\delta > 0$ and $C < \infty$, there is a
$\delta^\prime = \delta^\prime(\delta, C) > 0$ with the
following property.  Suppose that 
$(X, \star_X)$ is a complete pointed metric space with a $(k, \delta^\prime)$-splitting $\alpha$.
Then for any $x \in B(\star_X, C)$, the pointed space
$(X, x)$ has a $(k, \delta)$-splitting coming from a change
of basepoint of $\alpha$.
\end{lemma}
\begin{proof}
In general, suppose that $f : (X, \star_X) \ra (Y, \star_Y)$ is a
$\delta^\prime$-Gromov-Hausdorff approximation. Given $x \in B(\star_X, C)$,
consider $x$ to be a new basepoint.
Note that
\begin{equation}
d(\star_Y, f(x)) \le d(\star_X, x) + \delta^\prime \le C + \delta^\prime.
\end{equation}
Suppose that $\delta$ satisfies
\begin{enumerate}
\item $\delta^{-1} \le (\delta^\prime)^{-1} - C$,
\item $\delta^{-1} - \delta \le 
(\delta^\prime)^{-1} - 2 \delta^\prime - C$ and
\item $\delta > 2 \delta^\prime$.
\end{enumerate}
We claim that
$f$ is a $\delta$-Gromov-Hausdorff between $(X, x)$ and $(Y, f(x))$.
To see this, first if
$x_1, x_2 \in B(x, \delta^{-1})$ then
$x_1, x_2 \in B(\star_X, (\delta^\prime)^{-1}$ and so
\begin{equation}
|d_Y(f(x_1),f(x_2))-d_X(x_1,x_2)|\leq \de^\prime \le \de.
\end{equation}

Next, if $y \in B(f(x), \delta^{-1} - \de)$ then
$y \in B(\star_Y, (\delta^\prime)^{-1} - \delta^\prime))$ and so
there
is some $\widehat{x} \in B(\star_X, (\delta^\prime)^{-1})$ with
$d(y, f(\widehat{x})) \le \de^\prime$.
Now 
\begin{equation}
d(x, \widehat{x}) \le d(f(x), f(\widehat{x})) + \de^\prime
\le d(f(x), y) + d(y, f(\widehat{x})) + \de^\prime \le
\delta^{-1} - \delta + 2\de^\prime < \delta^{-1},
\end{equation}
which proves the claim.

The lemma now follows provided that we specialize to the case when $Y$ 
splits off an $\R^k$-factor.
\end{proof}

\begin{lemma} \label{faralignment2}
Given $\delta > 0$ and 
$C < \infty$, there is a 
$\delta^\prime = \delta^\prime(\de, C) > 0$
with the following property.
Suppose that 
$(X, \star_X)$ is a complete pointed metric space
with a $(k_1, \delta^\prime)$-splitting $\alpha_1$ and 
a $(k_2, \delta^\prime)$-splitting $\alpha_2$. Suppose that
$\alpha_1$
is $\delta^\prime$-compatible with $\alpha_2$.
Given
$x\in B(\star_X, C)$, let $\alpha^\prime_1, \alpha^\prime_2$ be the
approximate splittings of $(X, x)$ coming from a change of basepoint
in $\alpha_1, \alpha_2$. Then $\alpha^\prime_1$ and $\alpha^\prime_2$ are
$\de$-compatible. 
\end{lemma}
\begin{proof}
The proof is similar to that of Lemma \ref{faralignment1}.
We omit the details.
\end{proof}

\subsection{Approximate splittings of Alexandrov spaces}

Recall the notion of a point in an Alexandrov space having a 
$k$-strainer
of a certain size and quality; see Subsection \ref{subsec-alexandrov}.
The next lemma shows that the notions of having a
good strainer and having a good approximate $\R^k$-splitting are
essentially equivalent 
for Alexandrov spaces.

\begin{lemma} \label{strainerlemma}
\begin{enumerate}
\item 
Given $k \in \Z^+$ and $\delta > 0$, there is a 
$\delta^\prime = 
\delta^\prime(k, \delta) > 0$
with the following property.
Suppose that $(X,\star_X)$ is a complete pointed 
nonnegatively curved Alexandrov space 
with a $(k, \delta^\prime)$-splitting.  Then $\star_X$ has a
$k$-strainer of quality $\delta$ at a scale
$\frac{1}{\delta}$.
\item 
Given $n \in \Z^+$ and $\delta > 0$, there is a 
$\delta^\prime = \delta^\prime(n, \delta) > 0$ with the following property.
Suppose that $(X,\star_X)$ is an complete pointed length space so that
$B \left( \star_X, \frac{1}{\delta^\prime} \right)$ has
curvature bounded below by $- \: \delta^\prime$ and
dimension bounded above by $n$. Suppose
that for some $k \le n$,
$\star_X$ has a $k$-strainer $\{p^\pm_i\}_{i=1}^k$
of quality $\delta^\prime$ at a scale
$\frac{1}{\delta^\prime}$. Then $(X,\star_X)$ has a $(k, \delta)$-splitting
$\phi:(X,\star_X)\ra (\R^k\times X',(0,\star_{X'}))$
where the composition $\pi_{\R^k}\circ \phi$ has $j^{th}$ component
$d_X(p^+_j, \star_X)-d_X(p^+_j, \cdot)$.
\end{enumerate}
\end{lemma}
\begin{proof}
The proof of (1) is immediate from the definitions.

Suppose that (2) were false.
Then for each $i \in \Z^+$, there
is an complete pointed length space $(X_i, \star_{X_i})$ so that 
\begin{enumerate}
\item $B(\star_{X_i}, i)$ has dimension
at most $n$,
\item $B(\star_{X_i}, i)$ has curvature bounded below by $- \: \frac{1}{i}$ and 
\item $\star_{X_i}$ has a $k$-strainer of quality $\frac{1}{i}$ at a scale
$i$ but
\item If $\Phi_i:(X_k,\star_{X_i})\ra \R^k$ has 
$j^{th}$ component defined as above, then $\Phi_i$
is not the $\R^k$  part of a $(k,\de)$-splitting for any $i$.
\end{enumerate}
After
passing to a subsequence, we can assume that $\lim_{i \rightarrow
\infty}(X_i, \star_{X_i}) = (X_\infty, \star_{X_\infty})$, for 
some pointed nonnegatively curved Alexandrov space
$(X_\infty, \star_{X_\infty})$ of dimension at most $n$,
the $k$-strainers yield $k$ pairs $\{\ga_j^\pm\}_{j=1}^k$
of opposite rays leaving $\star_{X_\infty}$, the
opposite rays $\ga_j^\pm$ fit together to form $k$ orthogonal
lines, and the $j^{th}$
components of the $\Phi_i$'s converge to 
the negative of the Busemann function of
$\ga_j^+$.
Using the Splitting Theorem \cite[Chapter 10.5]{Burago-Burago-Ivanov}
it follows that
$X_\infty$ splits off an $\R^k$-factor. 
This gives a contradiction.
\end{proof}

\begin{lemma} \label{limitsplits}
Given $k \le n \in \Z^+$,
suppose that $\{(X_i,\star_{X_i})\}_{i=1}^\infty$ 
is a sequence of complete pointed length spaces
and $\delta_i \rightarrow 0$ is a positive
sequence such that 
\begin{enumerate}
\item Each $B \left( \star_{X_i}, \frac{1}{\delta_i} \right)$ has 
curvature bounded below by $- \: \delta_i$ and dimension bounded above by $n$. 
\item Each $(X_i, \star_{X_i})$ has a $(k, \delta_i)$-splitting. 
\item $\lim_{i \rightarrow \infty} (X_i, \star_{X_i}) = (X_\infty, \star_{X_\infty})$ in the
pointed Gromov-Hausdorff topology. 
\end{enumerate}
Then $(X_\infty, \star_{X_\infty})$ is a
nonnegatively curved Alexandrov space with a $k$-splitting.
\end{lemma}
\begin{proof}
This follows from Lemma \ref{strainerlemma}.
\end{proof}

\subsection{Compatibility of approximate splittings}

Next, we show that the nonexistence of an approximate $(k+1)$-splitting implies that
approximate $j$-splittings are approximately compatible with $k$-splittings
for $j \le k$.

\begin{lemma} \label{alexcompatible}
Given 
$j \le k \le n\in\Z^+$ and $\beta_k^\prime, \beta_{k+1} > 0$,
there are numbers
$\de = \de(j,k,n,\beta_k^\prime, \beta_{k+1}) > 0$.
$\be_j = \beta_j(j,k,n,\beta_k^\prime, \beta_{k+1}) > 0$ and
$\be_k = \beta_k(j,k,n,\beta_k^\prime, \beta_{k+1}) > 0$ with the following property.
If $(X,\star_X)$ is a 
complete pointed length 
space such that 
\begin{enumerate}
\item The ball $B(\star_X,\de^{-1})$ has curvature bounded below by
$-\de$ and dimension bounded above by $n$, and
\item $(X,\star_X)$ does not admit a $(k+1,\be_{k+1})$-splitting 
\end{enumerate}
then any  $(j,\be_j)$-splitting of $(X,\star_X)$ is $\be_k^\prime$-compatible 
with any $(k,\be_k)$-splitting.
\end{lemma}
\begin{proof}
Suppose that the lemma is false. Then for some $j \le k \le n\in \Z^+$ and
$\be_k',\,\be_{k+1} > 0$, there are

\begin{enumerate}
\item A sequence $\{(X_i, \star_{X_i})\}_{i=1}^\infty$ of pointed complete length spaces,
\item A sequence $\{\al_i: (X_i, \star_{X_i}) \ra (X_{1,i},\star_{X_{1,i}}) \times (X_{2,i}, \star_{X_{2,i}})
 \}$ 
of $(j,i^{-1})$-splittings and
\item A sequence $\{\bar\al_i: (X_i, \star_{X_i}) \ra (Y_{1,i}, \star_{Y_{1,i}}) \times (Y_{2,i}, \star_{Y_{2,i}}) 
 \}$
of $(k,i^{-1})$-splittings
\end{enumerate}
such that
\begin{enumerate}
\setcounter{enumi}{3}
\item 
$B(\star_{X_i},i^{-1})$ has curvature 
bounded below by $-i^{-1}$,
\item $B(\star_{X_i},i^{-1})$ has dimension at most $n$,
\item $(X_i,\star_{X_i})$ does not admit a $(k+1,\be_{k+1})$-splitting
for any $i$ and
\item $\al_i$ is not $\be_k'$-compatible with $\bar\al_i$ for any $i$.
\end{enumerate}

By (4), (5) and Lemma \ref{alexlimit},
after passing to a subsequence we can assume that there is a 
pointed nonnegatively curved
Alexandrov space $(X_\infty, \star_{X_\infty})$,
and a sequence $\{\Phi_i : (X_i,\star_{X_i}) \ra (X_\infty, \star_{X_\infty}) \}$
of $i^{-1}$-Gromov-Hausdorff approximations.
In view of (2), (3) and Lemma \ref{limitsplits}, 
after passing to a further subsequence
we can also assume that there is a pointed $j$-splitting
$\al_\infty: (X_\infty, \star_{X_\infty}) \ra
(X_{\infty.1}, \star_{X_{\infty,1}}) \times (X_{\infty, 2}, \star_{X_{\infty,2}})$
and a pointed $k$-splitting
$\bar\al_\infty: (X_\infty, \star_{X_\infty}) \ra
(Y_{\infty,1}, \star_{Y_{\infty,1}}) \times (Y_{\infty,2}, \star_{Y_{\infty,2}})$
such that $\al_\infty \circ \Phi_i$ (respectively $\bar\al_\infty \circ \Phi_i$)
is $i^{-1}$-close to $\al_i$ (respectively $\bar\al_i$).  By Lemma \ref{lem-splitcompatible}
and the fact that $(X_\infty, \star_{X_\infty})$ does not admit a $(k+1)$-splitting,
we conclude that $\al_\infty$ is compatible with $\bar\al_\infty$.
It follows that $\al_i$ is $\be_k'$-compatible with $\bar\al_i$ for large $i$, 
which is a contradiction.
\end{proof}

\begin{remark}
Assumption (1) in Lemma \ref{alexcompatible} is probably not
necessary but it allows us to give a simple proof, and it will
be satisfied when we apply the lemma.
\end{remark}

\subsection{Overlapping cones}

In this subsection we prove a result about overlapping 
almost-conical regions that we will need later.
We recall that a pointed metric space $(X,\star)$ is a {\em metric
cone} if it is a union of rays leaving the basepoint $\star$,
and the union of any two rays $\ga_1,\,\ga_2$
leaving $\star$ is isometric to the 
union of two rays $\bar\ga_1,\,\bar\ga_2\subset\R^2$ leaving the origin
$o\in\R^2$.

\begin{lemma} \label{2cones}
If $(X, \star_X)$ is a conical nonnegatively curved
Alexandrov space and there is some $x \neq \star_X$ so that
$(X, x)$ is also a conical Alexandrov space then
$X$ has a $1$-splitting such that the segment from $\star_X$ to $x$
is parallel to the $\R$-factor.
\end{lemma}
\begin{proof}
Let $\al$ be a segment joining $\star_X$ to $x$.
Since $X$ is conical with respect to both $\star_X$ and $x$,
the segment $\al$ can be extended in both directions
as a geodesic $\ga$.  The cone structure implies that $\ga$ is a line.
The lemma now follows from the splitting theorem.
\end{proof}

\begin{lemma} \label{almost2cones}
Given $n \in \Z^+$ and $\de > 0$, there is a 
$\de^\prime = \de^\prime(n, \de) > 0$ with the
following property.  If
\begin{itemize}
\item $(X, \star_X)$ is a complete pointed length space,
\item $x \in X$ has $d(\star_X,  x) = 1$ and 
\item $(X, \star_X)$ and $(X, x)$ have pointed Gromov-Hausdorff 
distance
less than $\de^\prime$ from conical nonnegatively curved Alexandrov spaces  
$CY$ and $CY^\prime$, respectively,
of dimension at most $n$
\end{itemize}
then
$(X,x)$ has a $(1,\de)$-splitting.
\end{lemma}
\begin{proof}
If the lemma were false, then there would be
a $\de > 0$,  a positive sequence $\de^\prime_i \ra 0$,
a sequence of pointed complete length spaces
$\{(X_i, \star_{X_i})\}_{i=1}^\infty$, and 
points $x_i \in X_i$
such  that for every $i$:
\begin{itemize}
\item
$d(\star_{X_i}, x_i) = 1$.
\item $(X_i, \star_{X_i})$ and $(X_i, x_i)$ have
pointed Gromov-Hausdorff distance less than $\delta^\prime_i$ from  conical
nonnegatively curved
Alexandrov spaces $CY_i$ and $CY_i^\prime$, respectively, of 
dimension at most $n$.
\item $(X_i, x_i)$ does not have a $(1, \de)$-splitting.
\end{itemize}
After passing to a subsequence, we can assume that 
we have Gromov-Hausdorff limits
$\lim_{i \ra \infty} (X_i, \star_{X_i}) = (X_\infty, \star_{X_\infty})$, $\lim_{i \ra \infty} x_i =
x_\infty$ with
$d(\star_{X_\infty}, x_\infty) = 1$, and both $(X_\infty, \star_{X_\infty})$ and
$(X_\infty, x_\infty)$ are conical nonnegatively curved
Alexandrov spaces.  By Lemma \ref{2cones},
$(X_\infty, x_\infty)$ has a $1$-splitting.  This gives a
contradiction.
\end{proof}

\subsection{Adapted coordinates}

In this subsection
we discuss coordinate systems which arise from $(k,\de)$-splittings
of Riemannian manifolds, in the presence of a lower curvature bound.  The
basic construction combines the standard construction of
strainer coordinates \cite{Burago-Gromov-Perelman} with the 
smoothing result of Corollary \ref{cor-1strained}.

\begin{definition} \label{def4}
Suppose $0<\de'\leq \de$,
and let $\al$ be a $(k, \delta^\prime)$-splitting 
of a complete pointed Riemannian manifold $(M, \star_M)$.
Let $\Phi : B \left( \star_M, \frac{1}{\delta} \right) \ra
\R^k$ be the composition
$B \left( \star_M, \frac{1}{\delta} \right) \stackrel{\al}{\ra}
\R^k \times X_2 \ra \R^k$. 
Then a map $\phi:(B(\star_M,1),\star_M)\ra (\R^k,\phi(\star_M))$ defines
{\em  $\al$-adapted coordinates
of quality $\delta$} if
\begin{enumerate}
\item $\phi$ is  
smooth and
$(1+\delta)$-Lipschitz.
\item  The image of $\phi$ 
has Hausdorff distance at most $\delta$ from
$B(\phi(\star_M),1)\subset\R^k$.
\item
For all $m \in B(\star_M, 1)$ and 
$m^\prime \in B \left(\star_M, \frac{1}{\delta} \right)$
with $d(m, m^\prime) > 1$, the (unit-length) initial 
velocity vector
$v \in T_mM$
of any minimizing geodesic from $m$ to $m^\prime$ satisfies
\begin{equation} \label{adapted}
\left| D\phi(v) - \frac{\Phi(m^\prime)-\Phi(m)}{d(m, m^\prime)} \right| 
< \delta.
\end{equation}
\end{enumerate}
We will say that a map $\phi:(B(\star_M,1),\star_M)\ra (\R^k,0)$
defines {\em adapted coordinates of quality $\de$} if there exists
a $(k,\de)$-splitting $\al$ such that $\phi$ defines $\al$-adapted
coordinates of quality $\de$, as above.  Likewise, {\em  $(M,\star_M)$
admits $k$-dimensional adapted coordinates of quality $\de$}
if there is a map $\phi$ as above which defines adapted coordinates of 
quality $\de$.
\end{definition}

We now give a sufficient condition for an approximate splitting to have
good adapted coordinates.

\begin{lemma} \label{lem1}(Existence of adapted coordinates)
For all $n \in \Z^+$ and $\delta > 0$, there is a 
$\delta^\prime = \delta^\prime(n, \delta) >0$
with the following property. Suppose that $(M, \star_M)$ is an  
$n$-dimensional complete pointed Riemannian manifold with sectional curvature
bounded below by $- \delta^{\prime\,2}$ on 
$B \left(\star_M, \frac{1}{\delta^\prime} \right)$, which has a
$(k, \delta^\prime)$-splitting $\al$. Then
there exist $\al$-adapted coordinates of quality $\delta$.
\end{lemma}
\begin{proof} 
The idea of the proof is to use the approximate splitting to construct a 
strainer and then use the strainer to construct Lipschitz-regular coordinates,
which can be smoothed using Corollary \ref{cor-1strained}.

Fix $n \in \Z^+$ and $\de >0$.  
Suppose that $\de'<\de$ and $(M, \star_M)$ has a $(k,\de^\prime)$-splitting
$\al: \, (M,\star_M) \ra (\R^k,0) \times (Y,\star_Y)$.
Let $\{e_j\}_{j=1}^k$ be an orthonormal basis of $\R^k$. Given 
a parameter
$s \in (\frac{1}{\de},\frac{1}{10\de^\prime})$, 
choose $p_{j \pm} \in M$ so that
$\al(p_{j \pm})$ lies in the $10 \delta^\prime$-neighborhood of
$(\pm s e_j, \star_Y)$.

Define $\phi_0 \: : \: B(\star_M, 1) \rightarrow \R^k$ by
\begin{equation} \label{phimap}
\phi_0(m) \: = 
\left(\; d(\star_M, p_{1+}) - d(m, p_{1+}),
\ldots,
d(\star_M, p_{k+}) - d(m, p_{k+})
\right).
\end{equation}
We will show that if $s$ and $\de^\prime$ are chosen appropriately 
then we can smooth the component functions of $\phi_0$
using Corollary \ref{cor-1strained}, to obtain a map
$\phi:B(\star_M,1)\ra (\R^k,0)$ which defines $\al$-adapted
coordinates of quality $\de$. (Note
that $\al$ is also a $(k,\de)$-splitting
since $\de^\prime<\de$). 
We first estimate the left-hand side of (\ref{adapted}).
Recall that if $m$ is a point of differentiability of
$d_{p_{j+}}$
and $v \in T_mM$ is the initial vector of a unit-speed
minimizing geodesic $\overline{m m^\prime}$ then 
$Dd_{p_{j+}}(v) = - \cos (\cangle_{m} (m^\prime,p_{j+}))$.

\begin{sublemma}
There exists
$\bar s = \bar s(n,\delta) > \frac{1}{\delta}$ so that for
each $s \ge \bar s$, there is some
$\overline{\delta}^\prime=\overline{\delta}^\prime (n,s,\de) < \frac{1}{10s}$
such that if $\de^\prime < \overline{\delta}^\prime$
then the following holds.

Under the hypotheses of the lemma, 
suppose that $m \in B(\star_M, 1)$, $m^\prime \in 
B(\star_M, \frac{1}{\delta})$
and $d(m,m^\prime) \ge 1$.
Let
$\overline{m p_{j+}}$
and $\overline{mm^\prime}$ 
be minimizing geodesics. Then
\begin{equation} \label{ineq1}
\left| 
\cos (\cangle_m( m^\prime, p_{j+})) - 
\frac{d(m, p_{j+}) - d(m^\prime, p_{j+})}{d(m,m^\prime)} 
\right| <  \delta.
\end{equation}
\end{sublemma}

\begin{proof}
Suppose that the sublemma is not true.  Then for each
$\bar s > \frac{1}{\delta}$, one can find some
$s \ge \bar s$ so that there are (see Figure 1):
\begin{itemize}
\item A positive sequence $\delta_i^\prime \ra 0$,
\item A sequence 
$\{(M_i, \star_{M_i}) \}_{i=1}^\infty$ of $n$-dimensional complete pointed
Riemannian manifolds with sectional curvature bounded below by 
$- \delta_i^{\prime 2}$
on $B(\star_{M_i}, \frac{1}{\delta_i^\prime})$,
\item A $(k, \delta_i^\prime)$-splitting $\al_i^\prime$ of $M_i$ and
\item Points $m_i \in B(\star_{M_i}, 1)$ and $m_i^\prime
\in B(\star_{M_i}, \frac{1}{\delta})$ with
$d(m_i,m^\prime_i) \ge 1$ so that
\item For each $i$,
the inequality (\ref{ineq1}) fails.
\end{itemize}

\begin{figure}[h]
\begin{center}
 
\input{4.24.pspdftex}
 
\caption{}
\end{center}
\label{fig-4.24}

\end{figure}

Using Lemma \ref{limitsplits},
after passing to a subsequence, we can assume that 
$\lim_{i \rightarrow \infty} (M_i, \star_{M_i}) =  
(X_\infty, \star_{X_\infty})$ in the pointed
Gromov-Hausdorff topology for some pointed nonnegatively curved
Alexandrov space
$(X_\infty, \star_{X_\infty})$ which is an isometric product 
$\R^k \times Y_\infty$. 
After passing to a further
subsequence, we can assume 
\begin{itemize}
\item $\lim_{i \rightarrow \infty} m_i = x_\infty$ for some $x_\infty \in 
\overline{B(\star_{X_\infty}, 1)}$,
\item $\lim_{i \rightarrow \infty} m_i^\prime = x_\infty^\prime$ for some 
$x_\infty^\prime \in
 \overline{B(\star_{X_\infty}, 
\frac{1}{\delta})}$ with $d(x_\infty, x_\infty^\prime) \ge 1$,
\item The segments
$\overline{m_i m_i^\prime}$ converge to a segment 
$\overline{x_\infty  x_\infty^\prime}$ and
\item
$\lim_{i \ra \infty} p_{i,j+} = p_{\infty,j+}$ for some points
$p_{\infty,j+} \in X_\infty$
in $\R^k \times \{\star_{Y_\infty}\}$ of distance $s$ from
$\star_{X_\infty}$.
\end{itemize}
Then
\begin{itemize}
\item 
$\lim_{i \ra \infty} \cangle_{m_i} (m_i^\prime,p_{i,j+}) = 
\cangle_{x_\infty} (x_\infty^\prime,p_{\infty,j+})$,
\item $\lim_{i \ra \infty} d(m_i, m_i^\prime) = 
d(x_\infty, x_\infty^\prime)$ and
\item $\lim_{i \ra \infty} d(m_i^\prime,p_{i,j+}) = 
d(x_\infty^\prime, p_{\infty,j+})$.
\end{itemize}
Now a straightforward verification shows that since we are in the case of an
exact $\R^k$-splitting,
there is a function $\eta = \eta( \delta, s)$ with
$\lim_{s \ra \infty} \eta(\de, s) = 0$ so that
\begin{equation}
\left|
\cos (\cangle_{x_\infty} ({x_\infty}^\prime, p_{\infty,j+})) - 
\frac{d({x_\infty}, p_{\infty,j+}) - d({x_\infty}^\prime, p_{\infty,j+}) 
}{d({x_\infty},{x_\infty}^\prime)} \right| < \eta. 
\end{equation}
Taking $s$ large enough gives a contradiction, thereby proving the 
sublemma.
\end{proof}

Returning to the proof of the lemma,
with $s \ge \overline{s}$,
define $\phi_0 $ as in (\ref{phimap}).
We have shown that if $\de^\prime$ is sufficiently small then
$\phi_0$ satisfies (\ref{adapted}) with $\delta$ replaced by
$\frac12 \delta$.
By a similar contradiction argument, one can show that
if $\delta^\prime$ is sufficiently small, as a function of $n$, 
 $s$ and $\delta$, then
$\phi_0$ is a $(1+ \frac12 \delta)$-Lipschitz map whose
image is a $\de$-Hausdorff approximation
to $B(0, 1) \subset \R^k$.

If it were not for problems with cutpoints which could cause 
$\phi_0$ to be nonsmooth,
then we could take $\phi = \phi_0$. In general,
we claim that if $s$ is large enough then
we can apply Lemma \ref{cor-1strained} in order to smooth 
$d_{ p_{j+}}$
on $B(\star_M, 1)$, and thereby construct $\phi$ from $\phi_0$.
To see this, note that for any $\eps>0$,
by making $s$ sufficiently large,  we can
arrange that for any $m\in B(\star_M,\frac{3}{\de})$,
the comparison angle $\cangle_m(p_{j+},p_{j-})$ is as close
to $\pi$ as we wish, and hence by triangle comparison
the hypotheses of Corollary \ref{cor-1strained} will
hold with 
$Y= p_{j+}$,
$C=\overline{B(\star_M,\frac{2}{\de})}\subset
U=B(\star_M,\frac{3}{\de})$, and $\th=\th(\eps)$ as in the statement of 
Corollary \ref{cor-1strained}. 
\end{proof}

We now show that under certain conditions, the adapted coordinates
associated to an approximate splitting are essentially unique.

\begin{lemma}(Uniqueness of adapted coordinates)
\label{lem-uniquenessofadaptedcoords}
Given 
$1 \leq k\le n \in \Z^+$ and $\eps > 0$, there is an $\eps^\prime =
\eps^\prime(n,\eps) > 0$ with the following property.  Suppose
that
\begin{enumerate}
\item $(M, \star_M)$ is an $n$-dimensional complete pointed
Riemannian manifold with sectional curvature bounded below by
$- (\eps^\prime)^2$ on $B \left( \star_M, \frac{1}{\eps^\prime} \right)$.
\item $\alpha:(M,\star_M)\ra (\R^k\times Z,(0,\star_Z))$ is a 
$(k,\eps^\prime)$-splitting 
of $(M, \star_M)$.
\item $\phi_1:(B(\star_M,1),\star_M)\ra (\R^k,0)$ is an  $\al$-adapted coordinate
on $B(\star_M, 1)$ of quality $\eps^\prime$.
\item Either
 (a) $\phi_2: (B(\star_M,1),\star_M)\ra (\R^k,0)$
is  an $\al$-adapted coordinate
on $B(\star_M, 1)$ of quality $\eps^\prime$, 
or \\
(b) $\phi_2$ has $(1+\eps')$-Lipschitz components
and the following holds :\\  For every
$m\in B(\star_M,1)$ and every $j\in \{1,\ldots,k\}$, there is an
$m_j'\in B(\star_M,(\eps')^{-1})$ with $d(m_j',m)>1$ satisfying
(\ref{adapted})
(with $\phi \leadsto \phi_2$),
such that $(\pi_{\R^k}\circ\al)(m_j')$
lies in the $\eps'$-neighborhood of the
line $(\pi_{\R^k}\circ\al)(m)+\R\,e_j$, and
$(\pi_Z\circ \al)(m_j')$ lies in the $\eps'$-ball centered
at $(\pi_Z\circ\al)(m)$.
\end{enumerate}
Then $\parallel \phi_1 - \phi_2 \parallel_{C^1} 
\le \eps$ on
$B(\star_M, 1)$.
\end{lemma}
\begin{proof}
We first give the proof when $\phi_2$ is also an
$\al$-adapted coordinate on $B(\star_M,1)$ of quality
$\eps'$.

Let $\Phi:(M,\star_M)\ra \R^k$ be the composition
$(M,\star_M)\stackrel{\al}{\ra}\R^k\times Z\ra \R^k$.

Given $\epsilon_1 > 0$, if $\eps^\prime$ is sufficiently small,
 then
by choosing points $\{p_{i,\pm}\}_{i=1}^k$ in $M$ with
$d(\al(p_{i,\pm}), (\pm {\epsilon_1}^{-1} e_i, 
\star_{Z})) \le \eps_1$, we obtain a strainer of quality 
comparable to
$\epsilon_1$ at scale ${\epsilon_1}^{-1}$.
Given $m \in B(\star_M, 1)$, let $\ga_i$
be a unit speed minimizing geodesic from $m$ to $p_{i,+}$,
let $v_i \in T_mM$ be the initial velocity
of $\ga_i$, and let $m_i$ be the point on $\ga_i$
with $d(m_i,m)=2$.
Given
$\eps_2>0$, if $\eps_1$ is sufficiently small then using triangle
comparison, we get  
\begin{equation}
\label{eqn-frame}
|\langle v_i,v_j\rangle -\de_{ij}|< \eps_2\,,\quad
|(\Phi(m_i)-\Phi(m))-2e_i|<\eps_2.
\end{equation}
for all $1\leq i,j\leq k$.
Applying (\ref{adapted}) with 
$m^\prime = m_i$
 gives
\begin{equation}
\label{eqn-frame2}
\max\left(|D\phi_1(v_i) - e_i|\,,\,
|D\phi_2(v_i) - e_i|\right) < 
\eps_2.
\end{equation}
Finally, given $\eps_3>0$, since $\phi_1$, $\phi_2$ are
$(1 + \eps^\prime)$-Lipschitz, 
if $\eps^\prime$ and $\eps_2$ are small enough then
we can assume that if 
$v$ is a unit vector and $v \perp \spann(v_1, \ldots, v_k)$ then
$\max(|D\phi_1(v)|,|D\phi_2(v)|) < \eps_3$. 
So for any $\eps_4>0$, if $\eps_3$ is sufficiently small then the
operator norm of $D\phi_1 - D\phi_2$ is bounded above by
$\eps_4$ on $B(\star_M, 1)$.

Since $\phi_1(\star_M) - \phi_2(\star_M) = 0$, we can integrate
the inequality $\parallel D\phi_1 - D\phi_2 \parallel \le
\eps_4$ along minimizing curves in $B(\star_M, 1)$ to
conclude that 
$\parallel\phi_1-\phi_2\parallel_{C^0}
\;\le\; 2\eps_4$ on
$B(\star_M, 1)$. Since $\eps_4$ is arbitrary, the lemma follows
in this case.

Now suppose $\phi_2$ satisfies instead the second condition
in (4).  If $m\in B(\star_M,1)$, $j\in \{1,\ldots,k\}$,
$m_j'$ is as in (4), and $v_j$ is the initial velocity 
of a unit speed geodesic from $m$ to $m_j'$, then 
(\ref{adapted}) implies that  $(D(\phi_2)_j)(v_j)$
is close to $1$ when $\eps'$ is small.  Since the
$j^{th}$ component $(\phi_2)_j$ is $(1+\eps')$-Lipschitz, 
this implies $\nabla(\phi_2)_j$ is close to $v_j$ when
$\eps'$ is small.  Applying the reasoning of the above paragraphs
to $\phi_1$, we get that $\nabla (\phi_1)_j$ is also close
to $v_j$ when $\eps'$ is small.  Hence $|D\phi_1-D\phi_2|$ is
small when $\eps'$ is small, and integrating within $B(\star_M,1)$
as before, the lemma follows.
\end{proof}

Finally, we show that approximate compatibility of two 
approximate splittings leads to an approximate compatibility of
their associated adapted coordinates.

\begin{lemma} \label{adaptedclose}
Given 
$1 \le j \le k \le n \in \Z^+$ and $\eps > 0$, there is an $\eps^\prime =
\eps^\prime(n,\eps) > 0$ with the following property.  Suppose
that
\begin{enumerate}
\item $(M, \star_M)$ is an $n$-dimensional complete pointed
Riemannian manifold with sectional curvature bounded below by
$- (\eps^\prime)^2$ on $B \left( \star_M, \frac{1}{\eps^\prime} \right)$.
\item $\alpha_1$ is a $(j,\eps^\prime)$-splitting 
of $(M, \star_M)$ and
$\alpha_2$ is a $(k,\eps^\prime)$-splitting 
of $(M, \star_M)$.
\item $\alpha_1$ is $\eps^\prime$-compatible with
$\alpha_2$.
\item $\phi_1:(M,\star_M)\ra (\R^j,0)$ and 
$\phi_2:(M,\star_M)\ra (\R^k,0)$ are adapted coordinates 
on $B(\star_M, 1)$ of quality $\eps^\prime$, 
associated to $\alpha_1$ and $\alpha_2$, respectively.
\end{enumerate}
Then there exists a map $T : \R^k \ra \R^j$, which is 
a composition of 
an isometry with an orthogonal projection, such that
$\parallel \phi_1 - T \circ \phi_2 \parallel_{C^1} 
\le \eps$ on
$B(\star_M, 1)$.
\end{lemma}
\begin{proof}
Let $\al_{1} : M \ra \R^j \times Z_{1}$ and
$\al_{2} : M \ra \R^k \times Z_{2}$ be the  approximate 
splittings.   Let $\Phi_1$ be the composition 
$B(\star_M, (\eps^\prime)^{-1}) \stackrel{\al_1}{\ra} \R^j \times Z_1
\ra \R^j$ and  $\Phi_2$ be the composition 
$B(\star_M, (\eps^\prime)^{-1}) \stackrel{\al_2}{\ra} \R^k \times Z_2
\ra \R^k$.

By (3), there is a $j$-splitting
$\ga:\R^k \ra \R^j \times \R^{k-j}$ and  a pair
of $\eps^\prime$-Gromov-Hausdorff
approximations $\xi_1:\R^j\ra \R^j$, $\xi_2:\R^{n-j}\times Z_2\ra Z_1$
such that the map $\widehat{\al}_2$ given by the composition
\begin{equation} \label{comp}
M \stackrel{\al_2}{\ra} 
\R^k \times Z_2 \stackrel{(\ga,\Id_{Z_2})}{\lra} 
\R^j \times \R^{k-j} \times Z_2\stackrel{(\xi_1,\xi_2)}{\ra}
\R^j \times Z_1
\end{equation}
agrees with the map $\al_1$ on the ball
$B(\star_M,(\eps^\prime)^{-1})$ up to error at most
 $\eps^\prime$.  Since Gromov-Hausdorff approximations
$(\R^j,0)\ra (\R^j,0)$ are close to isometries, for all $\eps_1>0$,
there will be a map $T:\R^k\ra \R^j$ (a composition of an isometry
and a projection) which agrees with
the composition 
\begin{equation}
\R^k\stackrel{\ga}{\lra} 
\R^j \times \R^{k-j}\stackrel{\xi_1}{\ra} \R^j
\end{equation}
up to error at most $\eps_1$ on the  ball $B(\star_M,\eps_1^{-1})$,
provided $\eps^\prime$ is sufficiently small.
Thus for all $\eps_2>0$, the composition 
\begin{equation}
M \stackrel{\al_2}{\ra} 
\R^k \times Z_2\ra\R^k\stackrel{T}{\ra}\R^j\,
\end{equation}
agrees with $\Phi_1$ up to error at most $\eps_2$ on $B(\star_M,\eps_2^{-1})$,
provided $\eps_1$ and $\eps'$ are sufficiently small.
Using Definition \ref{def4} for the approximate splitting
$\alpha_2$ and applying $T$, it follows that for 
all $\eps_3>0$, if $\eps_2$ is sufficiently small then
we are ensured that
$T\circ\phi_2$ defines $\al_1$-adapted coordinates 
on $B(\star_M,1)$ of quality
$\eps_3$.  By Lemma \ref{lem-uniquenessofadaptedcoords}
(using the first criterion in part(4) of
Lemma \ref{lem-uniquenessofadaptedcoords}), it follows
that if $\eps_3$ is sufficiently small then 
$\|\phi_1 - T\circ\phi_2\|_{C^1}<\eps$. 
\end{proof}

\begin{remark} \label{addedremark}
In Definition \ref{def4} we defined adapted coordinates on the unit ball
$B(\star_M, 1)$. By a rescaling, we can
define adapted coordinates on a ball of any specified size, and
the results of this section will remain valid.
\end{remark}

\section{Standing assumptions}

We now start on the proof of Theorem \ref{thmmain} 
in the case of closed manifolds. The proof is by contradiction.

The next lemma states that if we can get a contradiction from a certain
``Standing Assumption''
then we have proven Theorem \ref{thmmain}.
We recall from the definition of the volume scale in Definition \ref{def1}
that if $w \le w'$ then $r_p(w) \ge r_p(w')$.

\begin{lemma} \label{contra}
If Theorem \ref{thmmain} is false then we can satisfy the
following Standing Assumption, for an appropriate choice of $A'$.

\begin{assumption}
\label{assumptions}
Let $K \ge 10$ be a fixed integer, and 
let $A': (0, \infty) \times (0, \infty) \ra (0,\infty)$ be a function. 

We assume that
$\{(M^\al,g^\al)\}_{\al = 1}^\infty$ 
is a sequence of connected closed Riemannian $3$-manifolds such that 
\begin{enumerate}
\item For all $p \in M^\al$, the ratio $\frac{R_p}{r_p(1/\alpha)}$
of the curvature scale at $p$ to the $\frac{1}{\al}$-volume scale at
$p$ is bounded below by $\al$.
\item 
For all $p \in M^\al$ and $w^\prime \in [\frac{1}{\al}, c_3)$, 
let $r_p(w^\prime)$ denote the $w^\prime$-volume scale at $p$.
Then for each integer $k \in [0, K]$ and each $C \in (0, \al)$, we have
$|\nabla^k\Rm| \le A'(C, w^\prime) \:  r_p(w^\prime)^{-(k+2)}$
on $B(p, C r_p(w^\prime))$.
\item Each $M^\al$ fails to be a graph manifold.
\end{enumerate}
\end{assumption}
\end{lemma}
\begin{proof}
If Theorem \ref{thmmain} is false then for every $\al \in \Z^+$,
there is a manifold $(M^\al,g^\al)$ which satisfies the
hypotheses of Theorem \ref{thmmain}
with the parameter $w_0$ of the theorem set to
$w_0^\al = \frac{1}{8\alpha^4}$, but $M^\al$ is not a
graph manifold. 

We claim first that for every $p^\al \in M^\al$, we have $r_{p^\al}(1/\al) < 
R_{p^\al}$.
If not then for some $p^\al \in M^\al$, we have
$r_{p^\al}(1/\alpha) \ge R_{p^\al}$. From the
definition of $r_{p^\al}(1/\alpha)$, it follows that
\begin{equation}
\vol(B({p^\al}, R_{p^\al})) \ge \frac{1}{\alpha} R_{p^\al}^3 > 
\frac{1}{8\alpha^4} R_{p^\al}^3,
\end{equation}
which contradicts our choice of $w_0^\al$.

Thus $r_{p^\al}(1/\al) < R_{p^\al}$. Then
\begin{equation}
\frac{1}{\al} (r_{p^\al}(1/\al))^3 = \vol(B({p^\al}, r_{p^\al}(1/\al))) \le 
\vol(B({p^\al}, R_{p^\al})) \le 
\frac{1}{8\alpha^4} R_{p^\al}^3,
\end{equation}
so $\frac{R_{p^\al}}{r_{p^\al}(1/\al)} \ge 2 \alpha$. 
This shows that $\{(M^\al,g^\al)\}_{\al = 1}^\infty$ satisfies 
condition (1) of Standing Assumption
\ref{assumptions}.

To see that condition (2) of Standing Assumption \ref{assumptions}
holds, for an appropriate choice of $A^\prime$, 
we first note that if suffices to just consider $C \in [1, \alpha)$,
since a derivative bound on a bigger ball implies a derivative
bound on a smaller ball.  For $\widetilde{w}^\prime \in 
\left[ \frac{1}{\alpha}, c_3 \right)$, we have
\begin{equation}
C r_{p^\al}(\widetilde{w}') \le \alpha r_{p^\al}(1/\alpha) \le 
R_p.
\end{equation}
Now
\begin{equation}
\vol(B({p^\al}, C r_{p^\al}(\widetilde{w}^\prime))) \ge 
\vol(B({p^\al}, r_{p^\al}(\widetilde{w}^\prime))) =
\widetilde{w}^\prime (r_{p^\al}(\widetilde{w}^\prime))^3 = 
C^{-3} \widetilde{w}^\prime (C r_{p^\al}(\widetilde{w}^\prime))^3.
\end{equation} 
Put $w^\prime = C^{-3} \widetilde{w}^\prime$. Then
\begin{equation}
w_0^\al = \frac{1}{8 \al^4} \le w^\prime < c_3.
\end{equation}
Hypothesis (2) of Theorem \ref{thmmain} implies that
\begin{equation}
|\nabla^k \Rm| \le A(w^\prime) \: (C r_{p^\al}(\widetilde{w}^\prime))^{-(k+2)}
\end{equation} 
on $B(p^\al, C r_{p^\al}(\widetilde{w}^\prime))$.
Hence condition (2) of Standing Assumption \ref{assumptions} 
will be satisfied, for $C \in [1, \alpha)$, if we take
\begin{equation}
A'(C,\widetilde{w}') = 
\max_{0 \le k \le K} A \left( C^{-3} \widetilde{w}^\prime \right)
\: C^{-(k+2)}.
\end{equation}
\end{proof}

Standing Assumption \ref{assumptions} will remain 
in force until Section \ref{sec-proof2}, where we consider manifolds
with boundary.
We will eventually get a contradiction to Standing Assumption
\ref{assumptions}.

For the sake of notational brevity,
we will usually suppress the  superscript $\al$;  thus $M$
will refer to $M^\al$.
By convention, each of the statements made in the proof is to be interpreted
as being valid provided $\al$ is sufficiently large, whether or not this
qualification appears explicitly. 

\begin{remark}
The condition $K \ge 10$ in Standing Assumption \ref{assumptions}
is clearly not optimal but it is
good enough for the application to the geometrization
conjecture.
\end{remark}

\begin{remark}
We note that for fixed $\widehat{w} \in (0, c_3)$, conditions (1) and
(2) of Standing Assumption \ref{assumptions} imply that for large $\al$,
the following holds for all $p \in M^\al$:
\begin{enumerate}
\item $\frac{R_p}{r_p(\widehat{w})} \ge \al$
and
\item 
For each integer $k \in [0, K]$ and each $C \in (0, \al)$, we have
$|\nabla^k\Rm| \le A'(C, \widehat{w}) \:  r_p(\widehat{w})^{-(k+2)}$
on $B(p, C r_p(\widehat{w}))$.
\end{enumerate}
Since in addition
$\vol(B(p, r_p(\widehat{w}))) = \widehat{w} (r_p(\widehat{w}))^3$,
we have all of the ingredients to extract convergent subsequences,
at the $\widehat{w}$-volume scale,
with smooth pointed limits that are nonnegatively curved.
This is how the hypotheses of Standing Assumption \ref{assumptions} will 
enter into finding a contradiction. In effect, $\widehat{w}$  will
eventually become a judiciously chosen constant.
\end{remark}

\section{The scale function $\r$}
\label{sec-scalefunction}

In this section we introduce a smooth scale function $\r:M\ra (0,\infty)$ 
which will
be used throughout the rest of the paper.  This function is like a volume 
scale in
the sense that one has lower bounds on volume at scale $\r$, which enables  
one to appeal to 
$C^K$-precompactness
arguments. The advantage of $\r$ over a volume scale is that
$\r$ can be arranged to have small Lipschitz constant, which is  
technically useful.  

We will use the following lemma to construct slowly varying 
functions subject to {\it a priori} 
upper and lower bounds.

\begin{lemma} 
\label{lemslowlyvarying}
Suppose $X$ is a  
metric space,  
$C\in(0,\infty)$, and $l,u:X\ra (0,\infty)$
are functions.  Then there is a $C$-Lipschitz function $r:X\ra (0,\infty)$
satisfying $l\leq r\leq u$ if and only if
\begin{equation}
\label{eqnenvelope}
l(p)-Cd(p,q)\leq u(q)
\end{equation}
for all $p,q\in X$. 
\end{lemma}
\begin{proof}
Clearly if such an $r$ exists then (\ref{eqnenvelope}) must hold.

Conversely, suppose that (\ref{eqnenvelope}) holds and
 define $r:X\ra(0,\infty)$ by
\begin{equation}
r(q) = \sup\{l(p)-Cd(p,q)\mid p\in X\}\,.
\end{equation}
Then $l\leq r\leq u$.
For $q,q^\prime \in X$, since $l(p) - Cd(p,q) \ge l(p) - Cd(p,q^\prime)
- C d(q,q^\prime)$, we obtain
$r(q) \ge r(q^\prime) - C d(q,q^\prime)$, from which it follows that
$r$ is $C$-Lipschitz.
\end{proof}

Recall that $c_3$ is the volume of the unit ball in $\R^3$.
Let $\La > 0$ and $\bar w \in (0, c_3)$ be new parameters,
and put
\begin{equation} \label{wprimeeqn}
w^\prime \: = \: \frac{\bar w}{2(1+2\Lambda^{-1})^3}.
\end{equation}

Recall the notion of the volume scale
$r_p(\bar w)$ from Definition \ref{def1}.

\begin{corollary} \label{fake}
There is a smooth
$\Lambda$-Lipschitz function $\r \: : \: M \rightarrow (0, \infty)$ such
that for every  $p \in M$, we have
\begin{equation} \label{fakebound}
\frac12 \: r_p(\bar{w})  \: \le \: \r(p) \: \le \: 2 r_p(w^\prime)\,.
\end{equation}
\end{corollary}
\begin{proof}
We let $l:M\ra (0,\infty)$
be the $\bar w$-volume scale, and $u:M\ra (0,\infty)$ be the $w'$-volume scale.
We first verify (\ref{eqnenvelope}) with parameter $C=\frac{\Lambda}{2}$. To argue by contradiction,
suppose that for some $p,q\in M$ we have
$l(p) - \frac12 \Lambda d(p,q) > u(q)$. In particular, $d(p,q) < 
\frac{2l(p)}{\Lambda}$ and $u(q) < l(p)$.
There are inclusions $B(p, l(p)) \subset B(q,(1+2\Lambda^{-1})l(p)) \subset
B(p,(1+4\Lambda^{-1})l(p))$.
Then 
\begin{equation}
\vol(B(q,(1+2\Lambda^{-1})l(p)))  \geq \vol(B(p,l(p))) =  \bar wl^3(p)= 
2 w' \left((1+2\Lambda^{-1})l(p) \right)^3.
\end{equation}
For any $c > 0$, if $\al$ is sufficiently large then the sectional curvature on 
$B(p,(1+4\Lambda^{-1})l(p))$, and hence on $B(q,(1+2\Lambda^{-1})l(p))$,
is bounded below by $- \: c^2 \: l(p)^{-2}$. As $u(q) < l(p) < (1+2\Lambda^{-1})l(p)$, 
the Bishop-Gromov inequality implies that
\begin{align}
\frac{w^\prime \: u(q)^3
}{
\int_0^{u(q)/l(p)} \sinh^{2}(cr) \: dr
}
& =
\frac{\vol(B(q, u(q)))
}{
\int_0^{u(q)/l(p)} \sinh^{2}(cr) \: dr
}
\ge
\frac{\vol(B(q, (1+2\Lambda^{-1})l(p)))
}{
\int_0^{!+2\Lambda^{-1}} \sinh^{2}(cr) \: dr
} \\
& \ge
\frac{2 w' \left((1+2\Lambda^{-1})l(p) \right)^3
}{
\int_0^{1+2\Lambda^{-1}} \sinh^{2}(cr) \: dr
}, \notag
\end{align}
or
\begin{equation} \label{xxx}
\frac{c^2 \left( \frac{u(q)}{l(p)} \right)^3
}{
\int_0^{u(q)/l(p)} \sinh^{2}(cr) \: dr
} \ge 
\frac{2 c^2 (1+2\Lambda^{-1})^3
}{
\int_0^{!+2\Lambda^{-1}} \sinh^{2}(cr) \: dr
}.
\end{equation}
As the function $x\mapsto \frac{c^2}{3}\frac{x^3}{\int_0^x\;\sinh^2(cr)\:dr}$
tends uniformly to $1$ as $c\ra 0$, for $x\in (0,3]$, 
taking $c$ small gives a contradiction. (Note the factor of $2$ on
the right-hand side of (\ref{xxx})).

By Lemma \ref{lemslowlyvarying}, there is a $\frac{\Lambda}{2}$-Lipschitz 
function $r$ on $M$
satisfying $l \le r \le u$. The corollary now follows from Corollary 
\ref{Lipsmoothing}.

\end{proof}

We will write $\r_p$ for $\r(p)$.  Recall our convention that
the index $\al$ in the sequence $\{M^{\al}\}_{\al = 1}^\infty$ 
has been suppressed,
and that all statements are to be interpreted as being valid
provided $\al$ is sufficiently large. The next lemma shows
$C^{K}$-precompactness
at scale $\r$.

\begin{lemma}
\label{lem-rpconsequences}

\mbox{}
\begin{enumerate}
\item There is a constant $\widehat{w}=\widehat{w}(w') > 0$ such that
$\vol(B(p,\r_p))\geq \widehat{w}(\r_p)^3$
for every $p\in M$.
\item
For every  $p \in M$, $C < \infty$ and $k \in [0,K]$, we have
\begin{equation}
|\nabla^k \Rm| \: \le \: 
2^{k+2} \: 
A'(C,w^\prime)
 \: \r_{p}^{-(k+2)}\quad\mbox{on the ball}\quad 
B \left( p,\frac{1}{2}C\r_p \right)\,.
\end{equation}

\item
Given $\eps>0$, for sufficiently large $\al$ and for every
$p\in M^\al$, the rescaled pointed manifold 
$(\frac{1}{\r_p}M^{\al},p)$ is $\eps$-close in the pointed
$C^{K}$-topology
to a complete nonnegatively curved
$C^{K}$-smooth
Riemannian $3$-manifold.   Moreover, this manifold belongs
to a family which is compact in the pointed 
$C^{K}$-topology.
\end{enumerate}
\end{lemma}
\begin{proof}
\mbox{}

(1).
As $\frac12 \: \r_{p^\al} \le r_{p^\al}(w^\prime)$, the 
Bishop-Gromov inequality gives
\begin{equation}
\frac{\vol(B(p^\al, \frac12 \: \r_{p^\al}))}
{\int_0^{\frac{\r_{p^\al}}{2 r_{p^\al}(w^\prime)}} \sinh^{2}(r) \: dr
} \ge
\frac{\vol(B(p^\al, r_{p^\al}(w^\prime)))
}{
\int_0^1 \sinh^{2}(r) \: dr
} = \frac{w^\prime (r_{p^\al}(w^\prime))^3}{\int_0^1 \sinh^{2}(r) \: dr},
\end{equation}
or
\begin{equation} \label{lowervolume}
\frac{\vol(B(p^\al, \frac12 \: \r_{p^\al}))}{(\frac12 \: \r_{p^\al})^3
} \: \ge \:
\frac{
w^\prime}{\int_0^1 \sinh^{2}(r) \: dr} \:
\frac{
 \int_0^{\frac{\r_{p^\al}}{2 r_{p^\al}(w^\prime)}} \sinh^{2}(r) \: 
dr
}{
\left( \frac{\r_{p^\al}}{2 r_{p^\al}(w^\prime) }\right)^3
}
\ge \frac{w^\prime}{3 \int_0^1 \sinh^2(r) \: dr}.
\end{equation}
Thus 
\begin{equation}
\vol(B(p^\al, \r_{p^\al})) \ge 
\vol \left( B \left( p^\al, \r_{p^\al}/2 \right) \right) \ge
\frac{w^\prime}{24 \int_0^1 \sinh^2(r) \: dr} (\r_{p^\al})^3,
\end{equation}
which gives (1). 

(2).  From hypothesis (2) of Standing Assumption \ref{assumptions},
for each 
$C < \al$
and $k \in [0, K]$ we have
\begin{equation}
|\nabla^k \Rm| \: \le \: 
A'(C,w^\prime)
 \: r_{p^\al}(w^\prime)^{-(k+2)} \: \le \:
2^{k+2} \: 
A'(C,w^\prime)
 \: \r_{p^\al}^{-(k+2)}
\end{equation}
on $B(p^\al, C r_{p^\al}(w^\prime)) \supset B \left( p^\al, \frac12 
C \r_{p^\al} \right)$.

(3).  If not, then for some $\eps>0$,
after passing to a subsequence, for every 
$\al$ we could find $p^\al\in M^\al$ such that
$(\frac{1}{\r_{p^\al}}M^\al,p^\al)$ has distance
at least $\eps$ in the 
$C^{K}$-topology 
from a complete nonnegatively curved
$3$-manifold.  

(1) and (2) imply that after passing to a subsequence,
the sequence $\{(\frac{1}{\r_{p^\al}}M^\al,p^\al)\}$
converges in the pointed 
$C^{K}$-topology
to a complete Riemannian $3$-manifold.  But since 
the ratio
$\frac{R_p}{\r_p}$ tends uniformly to infinity as $\al\ra \infty$, 
the limit manifold has nonnegative curvature, which is 
a contradiction.
\end{proof}

We now extend Lemma \ref{lem-rpconsequences} to provide 
$C^{K}$-splittings.

\begin{lemma} \label{smoothconv}
Given $\eps>0$
and  $0\leq j\leq 3$, provided $\de<\overline{\de}(\eps,w')$
the following holds.
If $p \in M$, and 
$\phi:\left( \frac{1}{\r_{p}} M, p \right)\ra
(\R^j\times X,(0,\star_X))$ is a $(j,\de)$-splitting, then 
$\phi$ is $\eps$-close to a $(j,\eps)$-splitting
$\widehat{\phi}:\left( \frac{1}{\r_{p}} M, p \right)\ra
(\R^j\times \widehat{X},(0,\star_{\widehat{X}}))$, where
$\widehat{X}$ is a complete nonnegatively curved 
$C^K$-smooth
Riemannian $(3-j)$-manifold, and $\widehat{\phi}$ is $\eps$-close
to an isometry on the ball $B(p,\eps^{-1})\subset \frac{1}{\r_p}M$, 
in the $C^{K+1}$-topology.
\end{lemma}
\begin{proof}
Suppose not.
Then for some $\epsilon > 0$, after passing
to a subsequence we can assume that there are points $p^\al \in M^\al$ so
that 
$\left( \frac{1}{\r_{p^{\al}}} M, p^{\al} \right)$ 
admits a $(j,\al^{-1})$-splitting $\phi_j$,
but the conclusion of the lemma fails.  

By Lemma \ref{lem-rpconsequences}, 
a subsequence of 
$\left\{ \left( \frac{1}{\r_{p^\al}} M^\al, p^\al \right) \right\}_{\al = 1}^\infty$
converges in the pointed 
$C^{K}$-topology 
to a complete pointed nonnegatively curved
$C^{K}$-smooth
$3$-dimensional 
Riemannian manifold $(M^\infty,p^\infty)$, such that  $\phi_j$ converges to 
a $(j,0)$-splitting $\phi_\infty$ of
$(M^\infty,p^\infty)$.
This is a contradiction.
\end{proof}

\begin{remark} \label{addedconv}
If we only assume condition (1) of Assumption \ref{assumptions} then the 
proof of Lemma \ref{smoothconv} yields 
the following 
weaker conclusion: $\widehat{X}$ is a (nonnegatively curved)
$(3-j)$-dimensional Alexandrov space, and $\widehat{\phi}$
is a homeomorphism on $B(p,\eps^{-1})$.  See Section \ref{sec-shioyayamaguchi}
for more discussion.
\end{remark}

Let $\si>0$ be a new parameter. In the next lemma, we show that
if the parameter $\bar w$ is small then the pointed $3$-manifold
$(\frac{1}{\r_p}M,p)$ is Gromov-Hausdorff close to something of
lower dimension.

\begin{lemma} \label{lemsicloseto2d}
Under the constraint
$\bar w < \overline{\bar w}(\si, \Lambda)$, 
the following holds.
For every $p\in M$, the pointed space
$(\frac{1}{\r_p}M,p)$ is $\si$-close in the pointed Gromov-Hausdorff
metric to a nonnegatively curved Alexandrov space of dimension at most $2$.
\end{lemma}
\begin{proof}
Suppose that the lemma is not true. Then for some $\si, \Lambda >0$, 
there is a sequence
$\bar w_i \rightarrow 0$ and for each $i$, a sequence
$\left\{ \left( M^{\al(i,j)}, p^{\al(i, j)} \right) \right\}_{j=1}^\infty$ 
so that for each $j$,
$\left( \frac{1}{\r_{p^{\al(i, j)}}} M^{\al(i,j)}, p^{\al(i, j)} \right)$ 
has 
pointed Gromov-Hausdorff distance at least $\si$ from any 
nonnegatively curved Alexandrov space of dimension at most $2$.

Given $i$, as $j \rightarrow \infty$ the curvature scale at
$p^{\al(i, j)}$ divided by $r_{p^{\al(i, j)}}(w^\prime)$ 
goes to infinity. Hence
the curvature scale at $p^{\al(i, j)}$ divided by 
$\r_{p^{\al(i, j)}}$ also goes to infinity.
Thus we can find some $j = j(i)$ so that the curvature scale 
at $p^{\al(i, j(i))}$ is at least 
$i \: \r_{p^{\al(i, j(i))}}$. We relabel $M^{\al(i, j(i))}$ as $M^i$ and
$p^{\al(i,j(i))}$ as $p^i$.
Thus we have a sequence 
$\left\{ \left( M^{i}, p^{i} \right) \right\}_{i=1}^\infty$ 
so that for each $i$,
$\left( \frac{1}{\r_{p^i}} M^i, p^i \right)$ has 
pointed Gromov-Hausdorff distance at least $\si$ from any 
nonnegatively curved
Alexandrov space of dimension at most $2$, and  the curvature scale at
$p^i$ is at least $i \: \r_{p^i}$. In particular, a subsequence of the
$\left( \frac{1}{\r_{p^i}} M^i, p^i \right)$'s converges in the pointed
Gromov-Hausdorff topology to a nonnegatively curved
 Alexandrov space $(X,x)$, necessarily
of dimension $3$. Hence there is a uniform positive lower bound on
$\frac{
\vol(B(p^{i}, 2 \r_{p^{i}}))
}{
(2 \r_{p^{i}})^3
}$.

As $r_{p^i}(\bar w_i) \le 2 \r_{p^i}$, the Bishop-Gromov inequality implies that
\begin{equation}
\frac{
\bar w_i \left( r_{p^i}(\bar w_i) \right)^3
}{
\int_0^{\frac{ r_{p^i}(\bar w_i) }{ 2 \r_{p^i} }} \sinh^{2}(r) \: dr 
} 
=
\frac{
\vol(B(p^{i}, r_{p^{i}}(\bar w_i)))
}{
\int_0^{\frac{r_{p^{i}}(\bar w_i)}{2 \r_{p^{i}}}} \sinh^{2}(r) \: dr 
} 
\ge
\frac{
\vol(B(p^{i}, 2 \r_{p^{i}}))
}{
\int_0^1 \sinh^{2}(r) \: dr 
}.
\end{equation}
That is,
\begin{equation}
\frac{
\vol(B(p^{i}, 2 \r_{p^{i}}))
}{
(2 \r_{p^{i}})^3
} \le
\bar w_i \: \left( \int_0^1 \sinh^{2}(r) \: dr \right) \:
\frac{
\left( \frac{r_{p^{i}}(\bar w_i)}{2 \r_{p^{i}}} \right)^3
}{
\int_0^{\frac{r_{p^{i}}(\bar w_i)}{2 \r_{p^{i}}}} \sinh^{2}(r) \: dr 
} 
\le
3 \: \bar w_i \: \int_0^1 \sinh^2(r) \: dr.
\end{equation}
Since $\bar w_i \ra 0$, we obtain a contradiction.
\end{proof}

As explained in Section \ref{sec-indexnotation}, we will assume henceforth that the 
constraint
\begin{equation}
\label{eqnwbar}
\bar w< \overline{\bar w}(\si, \Lambda)
\end{equation}
is satisfied.

\section{Stratification}
\label{sec-stratification}

In this section we define a 
rough stratification of a Riemannian
$3$-manifold,
based on the maximal dimension of a Euclidean
factor of an approximate splitting
at a given point.

\subsection{Motivation}
Recall that in a metric polyhedron $P$, there is a natural metrically defined
filtration $P_0\subset P_1\subset\ldots\subset P$, 
where $P_k\subset P$ is the 
set of  points $p\in P$ that do not have a neighborhood that
isometrically splits off a factor of $\R^{k+1}$. The associated strata
$\{P_k - P_{k-1}\}$ are manifolds of dimension $k$.  An approximate
version  of this kind of filtration/stratification will be 
used in the proof of
Theorem \ref{thmmain}.

For the proof of Theorem \ref{thmmain}, we use a stratification of $M$
so that if $p\in M$ lies in the $k$-stratum then there is a 
metrically defined fibration
of an approximate ball centered at $p$, over an open subset of $\R^k$.
We now give a rough description of the strata; a precise definition will
be given shortly.

{\em $2$-stratum.}  Here $\left( \frac{1}{\r_p} M, p \right)$ 
is close to 
splitting off an $\R^2$-factor and, due to the collapsing
assumption, it is Gromov-Hausdorff close to $\R^2$. One gets
a circle fibration over an open subset of $\R^2$.

{\em $1$-stratum.}  
Here $\left( \frac{1}{\r_p} M, p \right)$ is not 
close to splitting off an $\R^2$-factor, but
is close to splitting off an $\R$-factor.  
These points fall into two types: those where 
$\left( \frac{1}{\r_p} M, p \right)$ looks like a 
half-plane, and those where it look like the product
of $\R$ with a space with controlled diameter.
One  gets a fibration over an 
open subset
of $\R$, whose fiber is $D^2$, $S^2$, or $T^2$.

{\em $0$-stratum.}  Here $\left( \frac{1}{\r_p} M, p \right)$ is 
not close to splitting off
an $\R$-factor.  
We will show that for some
radius $r$ comparable to $\r_p$,
$\left( \frac{1}{r} M, p \right)$ 
is Gromov-Hausdorff close
to the Tits cone $\ctits N$ of some complete nonnegatively curved 
$3$-manifold $N$ with at most one end, and  the ball $B(p,r)\subset M$ is
diffeomorphic to $N$.
The possibilities for the topology of $N$ are:
$S^1\times B^2$, $B^3$, $\R P^3-D^3$, a twisted
interval bundle over the Klein bottle, $S^1\times S^2$,
$\R P^3 \# \R P^3$, a spherical space form $S^3/\Gamma$
and a isometric quotient $T^3/\Gamma$ of the 3-torus.

\subsection{The k-stratum points}
\label{subsec-k-stratum}

To define the stratification precisely,  we introduce
the additional parameters $0<\be_1<\be_2<\be_3$, and
put $\be_0=0$.  Recall that  the parameter $\sigma$ has already 
been introduced in Section \ref{sec-scalefunction}.

\begin{definition} \label{nstratumdef}
A point $p \in M$  belongs to the k-stratum,
$k\in \{0,1,2,3\}$, 
if $(\frac{1}{\r_p}M,p)$ admits a 
$(k,\beta_k)$-splitting,
but does not admit a $(j,\beta_j)$-splitting
for any $j>k$.  
\end{definition}
Note that every pointed space has a $(0,0)$-splitting, so
every $p\in M$ belongs to the $k$-stratum for some 
$k\in \{0,1,2,3\}$.

\begin{lemma} \label{lem3}
Under the constraints
$\beta_3 < \overline{\beta}_3$ and $\sigma < \overline{\sigma}$ 
there are no $3$-stratum points.
\end{lemma}
\begin{proof}
Let $c > 0$ be the minimal distance, in the pointed Gromov-Hausdorff metric, between
$(\R^3, 0)$ and a nonnegatively curved Alexandrov space of dimension at most $2$. Taking
$\overline{\beta}_3 = \overline{\sigma} = \frac{c}{4}$, the lemma follows from
Lemma \ref{lemsicloseto2d}.
\end{proof}

Let $\Delta \in (\beta_2^{-1},\infty)$   be a new parameter.

\begin{lemma}
\label{lem-deltabig2stratum}
Under the constraint
$\Delta>\overline{\De}(\be_2)$, if $p\in M$
has a $2$-strainer 
of size $\frac{\Delta}{100}\r_p$ and quality
$\frac{1}{\Delta}$ at $p$, then 
$\left( \frac{1}{\r_p} M, p \right)$ 
has a $(2, \frac12 \beta_2)$-splitting $\frac{1}{\r_p} M \rightarrow \R^2$. 
In particular $p$ is
in the $2$-stratum. 
\end{lemma}
\begin{proof}
This follows from Lemma \ref{strainerlemma}.
\end{proof}

\begin{definition} \label{1stratumtypes}
A  $1$-stratum point $p\in M$ 
is in the {\em slim $1$-stratum} if  there is
a $(1, \beta_1)$-splitting
$(\frac{1}{\r_p} M,p) \ra (\R \times X,(0,\star_X))$
where 
$\diam (X) \le 10^3 \Delta$.
\end{definition}

\section{The local geometry of the $2$-stratum }
\label{sec-loc2stratum}

In the next few sections, we examine the local geometry near points of
different type, introducing adapted coordinates and certain associated
cutoff functions.

In this section we consider the $2$-stratum points.  Along with
the adapted coordinates and cutoff functions, we discuss the
local topology and a selection process to get a ball covering of the
$2$-stratum points.

\subsection{Adapted coordinates, cutoff functions and local topology
near $2$-stratum points} \label{adapted2}

Let $p$ denote a point in the $2$-stratum and let
 $\phi_p:\left( \frac{1}{r_p} M, p \right)\ra \R^2\times 
(X,\star_X)$ be a $(2, \beta_2)$-splitting. 

\begin{lemma}
\label{lem-phipghapprox}
Under the constraints $\be_2<\overline{\be}_2$ and $\si<\overline{\si}$,
the factor $(X,\star_X)$ has diameter $<1$.
\end{lemma}
\begin{proof}
If not then we could find a subsequence $\{M^{\al_j}\}$
of the $M^{\al}$'s, and $p_j\in M^{\al_j}$, such
that with $\be_2=\si=\frac{1}{j}$, the map 
$\phi_{p_j}:(\frac{1}{\r_{p_j}}M^{\al_j},p_j)\ra (\R^2
\times X_j,(0,\star_{X_j}))$
violates the conclusion of the lemma. 
We then pass to a pointed Gromov-Hausdorff sublimit 
$(M_\infty, p_\infty)$, which will be a nonnegatively curved Alexandrov space of 
dimension at most $2$, and extract a limiting $2$-splitting
$\phi_\infty:(M_\infty,p_\infty)\ra \R^2\times X_\infty$. 
The only possibility is that $\dim(M_\infty) = 2$, $\phi_\infty$
is an isometry and $X_\infty$ is a point.
This contradicts the diameter assumption.
\end{proof}

Let $\varsigma_{\twostratum}>0$ be a new parameter.

\begin{lemma}
\label{lem-existsetap2stratum}
Under the constraint $\be_2<
\overline{\beta}_2(\varsigma_{\twostratum})$,
there is 
a $\phi_p$-adapted coordinate $\eta_p$
of quality $\varsigma_{\twostratum}$
on 
$B(p, 200) 
\subset \left( \frac{1}{\r_p} M, p \right)$
\end{lemma}
\begin{proof}
This follows from Lemma \ref{lem1} (see also Remark \ref{addedremark}).
\end{proof}

\begin{definition}
\label{def-2stratumetap}
Let $\zeta_p$ be the smooth function on $M$ which is the extension
by zero of 
$\Phi_{8,9} \circ |\eta_p|$.
(See Section \ref{sec-indexnotation} for the 
definition of $\Phi_{a,b}$).
\end{definition}

\begin{lemma} \label{2fibration}
Under the constraints $\beta_2<\overline{\beta}_2$, 
$\varsigma_{\twostratum}<\overline{\varsigma}_{\twostratum}$ and
$\si<\overline{\si}$, the
restriction of $\eta_p$ to $\eta_p^{-1}(B(0,100))$ is a fibration
with fiber $S^1$.  In particular, for all
$R \in (0,100)$, $|\eta_p|^{-1}[0,R]$ is diffeomorphic to
$S^1 \times \overline{B(0,R)}$.
\end{lemma}
\begin{proof}
For small $\beta_2$ and
$\varsigma_{\twostratum}$, the map
$\eta_p:\frac{1}{\r_p}M\supset B(p,200)\ra \R^2$ is a submersion;
this follows by applying (\ref{adapted}) to an appropriate
$2$-strainer at $x\in B(p,200)$
furnished by the $(2,\be_2)$-splitting $\phi_p$. 

By Lemma \ref{lem-phipghapprox}, if $\varsigma_{\twostratum}
<\overline{\varsigma}_{\twostratum}$ then
$\eta_p^{-1}(B(0,100))\subset B(p,102)\subset \frac{1}{\r_p}M$.
Therefore if $K\subset B(0,100)$ is  compact then
$\eta_p^{-1}(K)$ is a closed subset of 
$\overline{B(p,102)}\subset 
\frac{1}{\r_p}M$,
and hence compact.  It follows that 
$\eta_p\restr_{\eta_p^{-1}(B(0,100))}:\eta_p^{-1}(B(0,100))
\ra B(0,100)$
is a proper map.   Thus 
$\eta_p\restr_{\eta_p^{-1}(B(0,100))}$
is  a trivial fiber bundle with 
compact $1$-dimensional fibers.    

Since $\eta_p^{-1}(0)$
has diameter at most $2$ 
by Lemma \ref{lem-phipghapprox} (assuming $\varsigma_{\twostratum}
<\overline{\varsigma}_{\twostratum}$), it follows
that any two points in  the fiber
$\eta_p^{-1}(0)$ can be joined
by a path in $\eta_p^{-1}(B(0,100))$.   Now the triviality
of the bundle implies that the fibers are connected, i.e.
diffeomorphic to $S^1$.
\end{proof}

\subsection{Selection of $2$-stratum balls} \label{2selection}

Let $\mathcal{M}$ be a new parameter, which will become a bound on
intersection multiplicity of balls. The corresponding bound
$\overline{\mathcal{M}}$ will describe how big ${\mathcal M}$ has
to be taken in order for various assertions to be valid.

Let $\{p_i\}_{i\in I_{\twostratum}}$ be a maximal set of $2$-stratum points with the property
that the collection 
$\{B(p_i,\frac13\r_{p_i})\}_{i\in I_{\twostratum}}$
is disjoint.
We write $\zeta_i$ for $\zeta_{p_i}$.

\begin{lemma}
\label{mult2}
Under the constraints ${\mathcal M}>\overline{\mathcal M}$ and
$\La<\overline{\La}$,
\begin{enumerate}
\item 
$\bigcup_{i\in I_{\twostratum}}\;B(p_i,\r_{p_i})$ 
contains all $2$-stratum
points.
\item The intersection multiplicity of the collection $\{\supp(\zeta_i)\}_{i\in I_{\twostratum}}$
is bounded by $\mathcal{M}$.
\end{enumerate}

\end{lemma}
\begin{proof}
(1). Assume $1+\frac23 \La<1.01$.
If $p$ is a $2$-stratum point, there is an $i\in I_{\twostratum}$
such that  
$B(p,\frac13 \r_p)\cap B(p_i,\frac13 \r_{p_i})\neq
\emptyset$ for some $i\in I_{\twostratum}$.  Then
$\frac{\r_p}{\r_{p_i}}\in [.9,1.1]$, and $p\in B(p_i,\r_{p_i})$.

(2). From the definition of $\zeta_i$, 
if $\varsigma_{\twostratum}$ is sufficiently
small then we are ensured that $\supp(\zeta_i)\subset B(p_i,10\r_{p_i})$.

Suppose that for some $p \in M$, we have
$p \in \bigcap_{j=1}^N B(p_{i_j}, 10 \r_{p_{i_j}})$ for distinct
$i_j$'s. We relabel so that
$B(p_{i_1}, \r_{p_{i_1}})$ has the smallest volume among the
$B(p_{i_j}, \r_{p_{i_j}})$'s.

If $10\Lambda $ is sufficiently small then we can assume that
for all $j$, $\frac12 \le \frac{\r_{p_{i_j}}}{\r_{p_{i_1}}} \le 2$.
Hence the $N$ disjoint balls $\{ B(p_{i_j}, \frac13\r_{p_{i_j}}) \}_{j=1}^N$
lie in $B(p_{i_1}, 100 \r_{p_{i_1}})$ and by Bishop-Gromov
volume comparison
\begin{equation}
N \le \frac{\vol(B(p_{i_1}, 100 \r_{p_{i_1}}))}{\vol(B(p_{i_1}, 
\frac13\r_{p_{i_1}}))} \le \frac{\int_0^{100} \sinh^2(r) \: 
dr}{\int_0^{\frac13} \sinh^2(r) \: dr}.
\end{equation}
This proves the lemma.
\end{proof}

\section{Edge points and associated structure}
\label{sec-edgepoints}

In this section we study points $p\in M$  where the pair 
$(M,p)$ looks like a half-plane with a basepoint lying on
the edge. Such points define an edge set $E$. 
We show that 
any $1$-stratum point, which is not a slim $1$-stratum point,
is not far from  $E$.

As a technical tool, we also introduce an approximate edge set $E'$ 
consisting of points where the edge structure is of slightly
lower quality than that of $E$. The set $E'$ will fill in the
boundary edges of the approximate half-plane regions around points
in $E$. We construct a smoothed distance function from $E'$, along
with an associated cutoff function. 

We describe the local topology
around points in $E$ and 
choose a
ball covering of $E$.

\subsection{Edge points}
\label{subsec-edgepoints}

We begin with a general lemma about $1$-stratum points.

\begin{lemma}
\label{lem-1splitalexandrov}
Given $\eps>0$, if 
$\be_1<\overline{\beta}_1(\eps)$ and
$\si<\overline{\si}(\eps)$
then the following holds.
If $(\frac{1}{\r_p}M,p)$ has a $(1,\be_1)$-splitting
then 
there
is a $(1,\eps)$-splitting
$(\frac{1}{\r_p}M,p)\ra (\R\times Y,(0,\star_Y))$, where
$Y$ is an Alexandrov space with $\dim(Y) \le 1$.
\end{lemma}
\begin{proof}
This is similar to the proof of Lemma \ref{lem-phipghapprox}.
If the lemma were false then we could find a sequence $\al_j\ra\infty$
so that taking 
$\be_1=j^{-1}$ and $\si=j^{-1}$,
for every $j$ there would
be $p_j\in M^{\al_j}$ and a $(1,j^{-1})$-splitting
of   $(\frac{1}{\r_{p_j}}M^{\al_j},{p_j})$, but no $(1,\eps)$-splitting
as asserted.  Passing to a subsequence,
we obtain a pointed Gromov-Hausdorff limit
$(M_\infty,p_\infty)$, and the $(1,j^{-1})$-splittings
converge to a $1$-splitting $\phi_\infty:M_\infty
\ra \R\times Y$.  It follows from
Lemma \ref{lemsicloseto2d} that $\dim M_\infty\leq 2$, 
and hence $\dim Y\leq 1$. 
This implies
that for large $j$, we can find arbitrarily good splittings
$(\frac{1}{\r_{p_j}}M^{\al_j},{p_j})\ra (\R\times Y_j,(0,\star_{Y_j}))$
where 
$Y_j$ is an Alexandrov space with $\dim(Y_j) \le 1$.
This is a contradiction.
\end{proof}

Let $0<\beta_{E}<\beta_{E'}$ and $0 < \sigma_E < \sigma_{E^\prime}$
be new parameters.

\begin{definition} \label{edgepoints}
A point $p \in M$ is an {\em $(s,t)$-edge point} if there is
a $(1,s)$-splitting
\begin{equation} 
F_p  :  \left( \frac{1}{\r_p} M, p \right)  \ra 
\left( \R \times Y, (0,\star_Y) \right)
\end{equation}
and a $t$-pointed-Gromov-Hausdorff approximation
\begin{equation}
G_p : (Y, \star_Y) \rightarrow
([0, C], 0),
\end{equation}
with $C \ge 200\Delta$.
Given $F_p$ and $G_p$, we put
\begin{equation} \label{edgeeqn}
Q_p = (\Id \times G_p) \circ F_p :  \left( \frac{1}{\r_p} M, p \right)  \ra 
(\R \times [0,C], (0,0)).
\end{equation}

We let $E$ denote the
set of  $(\beta_{E},\sigma_E)$-edge
points, and $E'$ denote the set of $(\beta_{E'},\sigma_{E^\prime})$-edge 
points.
Note that $E\subset E'$.
We will often refer to elements of $E$ as {\em edge points}.
\end{definition}

We emphasize that in the definition above, $Q_p$
maps the basepoint $p \in M$ to $(0,0) \in \R \times [0,C]$. 

\begin{lemma}
\label{lem-edgepointnot2stratum}
Under the constraints 
$\beta_{E'}<\overline{\beta}_{E'}$,
$\sigma_{E'}<\overline{\sigma}_{E'}$ and
$\beta_2 < \overline{\beta}_2$, 
no element $p\in E'$  can be a $2$-stratum point.
\end{lemma}
\begin{proof}
By Lemma \ref{lem-phipghapprox}, if $p$ is a $2$-stratum point and 
$\phi_p:(\frac{1}{\r_p}M,p)\ra (\R^2\times X,(0,\star_X)$ 
is a $(2,\be_2)$-splitting then $\diam X<1$.  Thus if
$\be_{E^\prime}$, $\sigma_{E^\prime}$ and $\be_2$ are all small
then a large pointed ball in
$(\R^2,0)$ has pointed Gromov-Hausdorff distance less than two from a large
pointed ball in $(\R\times [0,C], (0,0))$, which is a contradiction.
\end{proof}

We now show that in a neighborhood of $p \in E$, the set $E'$ looks
like the border of a half-plane.

\begin{lemma} \label{nearedge}
Given $\eps > 0$, if
$\beta_{E'} < \overline{\beta}_{E'}(\eps, \Delta)$,
$\si_{E'} < \overline{\si}_{E'}(\eps, \Delta)$,
$\be_E<\overline{\be}_E(\be_{E'}, \si_{E'})$,
$\si_E<\overline{\si}_E(\be_{E'}, \si_{E'})$
and $\La< \overline{\Lambda}(\eps,
\Delta)$ then
the following holds.

For $p\in E$, if  $Q_p$ is as in Definition \ref{edgepoints} and  
$\widehat{Q}_p:\left( \R \times [0, C], (0,0) \right)\ra 
\left( \frac{1}{\r_p} M, p \right)$ is a quasi-inverse for $Q_p$
(see Subsection
\ref{def3}),   then 
$\widehat{Q}_p([-100\Delta,100\Delta]\times \{0\})$ is 
$\frac{\eps}{2}$-Hausdorff
close to $E'\cap Q_p^{-1}([-100\Delta,100\Delta]
\times [0,100\Delta])$.
\end{lemma}

\begin{proof}
Suppose  the lemma were false. Then for some
$\eps > 0$, there would be
sequences $\al_i \ra \infty$, $s_i \ra 0$ and $\Lambda_i \ra 0$
so that for each $i$,
\begin{enumerate}
\item The
scale function $\r$ of $M^{\al_i}$ has  Lipschitz constant
bounded above by $\Lambda_i$, and
\item There is an 
$(s_i^2, s_i^2)$-edge point
$p_i \in M^{\al_i}$
such that $\widehat{Q}_{p_i}([-100\Delta,100\Delta]\times \{0\})$ is not
$\frac{\eps}{2}$-Hausdorff
close to $E^\prime_i 
\cap Q_{p_i}^{-1}([-100\Delta,100\Delta]\times [0,100\Delta])$,
where  $E^\prime_i$ denotes the set of 
$(s_i, s_i)$-edge points
in $M^{\al_i}$. 
\end{enumerate}

After passing to a subsequence, we can assume that
$\lim_{i \ra \infty} \left( \frac{1}{\r_{p_i}} M^{\al_i}, p_i \right)
= \left( X^\infty, p_\infty \right)$ for some
pointed nonnegatively curved Alexandrov space $ \left( X^\infty, p_\infty \right)$.
We can also assume that
$\lim_{i \ra \infty} 
\widehat{Q}_{p_i} \big|_{[-200\Delta,200\Delta]\times [0,200\Delta]}$ exists
and is an isometric embedding 
$\widehat{Q}_\infty : [-200\Delta,200\Delta]\times [0,200\Delta] \ra
X^\infty$, with $\widehat{Q}_\infty(0,0) = p_\infty$. 
Then 
\begin{equation}
\lim_{i \ra \infty} \widehat{Q}_{p_i}([-100\Delta,100\Delta]\times \{0\})
= \widehat{Q}_\infty ([-100\Delta,100\Delta]\times \{0\}).
\end{equation}
However, since $s_i \ra 0$ and $\Lambda_i \ra 0$, it follows that
\begin{equation}
\lim_{i \ra \infty} E_i^\prime
\cap Q_{p_i}^{-1}([-100\Delta,100\Delta]\times [0,100\Delta]) =
\widehat{Q}_\infty ([-100\Delta,100\Delta]\times \{0\}).
\end{equation}
Hence for large $i$, 
$\widehat{Q}_{p_i}([-100\Delta,100\Delta]\times \{0\})$ is
$\frac{\eps}{2}$-Hausdorff
close to $E^\prime_i 
\cap Q_{p_i}^{-1}([-100\Delta,100\Delta]\times [0,100\Delta])$.
This is a contradiction.
\end{proof}

The first part of the next lemma says that $1$-stratum points are
either slim $1$-stratum points or lie not too far from an edge point.
The second part says that $E$ is coarsely dense in $E'$.

\begin{lemma} \label{hitsall}
Under the constraints
$\beta_{E'} < \overline{\beta}_{E'}(\Delta)$,
$\si_{E'} < \overline{\si}_{E'}(\Delta)$,
$\beta_{E} < \overline{\beta}_E({\beta}_{E'},\si_{E'})$,
$\si_{E} < \overline{\si}_E({\beta}_{E'},\si_{E'})$,
$\beta_1 < \overline{\beta}_1(\Delta, \beta_{E})$,
$\si<\overline{\si}(\De, \si_E)$ 
and
$\La<\overline{\Lambda}(\Delta)$, 
the
following holds.
\begin{enumerate}
\item For
every $1$-stratum point $p$ which is not in the slim $1$-stratum,
there is some $q \in E$ with
$p \in B(q, \Delta\r_q)$. 
\item For every $1$-stratum point $p$ which is not in the slim $1$-stratum
and for every $p^\prime \in E^\prime \cap B(p, 10 \Delta \r_p)$,
there is some $q \in E$ with $p^\prime \in B(q, \r_q)$.  See Figure 2 below.
\end{enumerate}
\end{lemma}

\bigskip\bigskip

\mbox{}

\begin{figure}[h]
\label{fig-9.7}
\begin{center}
 
\input{9.7.pspdftex}
 
\caption{}
\end{center}
\end{figure}

\begin{proof}
Let $\eps > 0$ be a constant that will be adjusted
during the proof. Let $p$ be a $1$-stratum point which is not in
the slim $1$-stratum.

By Lemma \ref{lem-1splitalexandrov}, if 
$\be_1<\overline{\beta}_1(\De,\eps)$ and
$\si<\overline{\si}(\De,\eps)$
then there is a $(1,\eps)$-splitting
$F:(\frac{1}{\r_p}M,p)\ra (\R\times Y,(0,\star_Y))$,
where $Y$ is a nonnegatively curved Alexandrov space of dimension at most one.

\begin{sublemma}
$\diam(Y)\geq 500\De$. 
\end{sublemma}
\begin{proof}
Suppose that $\diam(Y)< 500\De$. Let
$\phi:(\frac{1}{\r_p}M,p)\ra (\R\times X,(0,\star_X))$
be a $(1,\be_1)$-splitting.

Since $p$ belongs to the $1$-stratum, $(\frac{1}{\r_p}M,p)$
does not admit a $(2,\be_2)$-splitting.  By Lemma 
\ref{alexcompatible}, it follows that if 
$\eps<\overline{\eps}(\De)$ and
$\be_1<\overline{\be}_1(\De)$ then there is a
$\frac{1}{10^3\De}$-Gromov-Hausdorff approximation
$(X,\star_X)\ra (Y,\star_Y)$.  Since 
$Y\subset B(\star_Y,500\De)$, we conclude that the metric
annulus $A(\star_X,600\De,900\De)\subset X$
is empty. But then the image of the ball 
$B(p,\be_1^{-1})\subset \frac{1}{\r_p}M$ under the
composition $B(p,\be_1^{-1})\stackrel{\phi}{\ra}
\R\times X\stackrel{\pi_X}{\ra}X$ 
will be contained in $B(\star_X,600\De)\subset X$
(because the inverse image of $B(\star_X,600\De)$
under $\pi_X\circ\phi$ is open and closed in 
the connected set $B(p,\be_1^{-1})$).  Thus 
$\phi: 
(\frac{1}{\r_p}M,p)\ra (\R\times B(\star_X,600\De),(0,\star_X))$
is a $(1,\be_1)$-splitting, and $p$ is in the slim $1$-stratum.
This is a contradiction.
\end{proof}

\noindent
{\em Proof of Lemma \ref{hitsall} continued.}
Suppose first that $Y$ is a circle.  If $\eps$ is sufficiently small 
then there is a $2$-strainer of size
$\frac{\Delta}{100} \r_p$ and quality $\frac{1}{\Delta}$ at $p$. By the choice of $\Delta$
(see Section \ref{sec-stratification}),
$p$ is a $2$-stratum point, which is a contradiction.

Hence up to isometry, $Y$ is an interval
$[0,C]$ with $C > 500\De$.
If $\star_Y \in \left( \frac{\Delta}{10}, C - \frac{\Delta}{10} \right)$ then the same
argument as in the preceding paragraph shows that $\star_Y$ is a $2$-stratum point provided $\eps$ is sufficiently small. Hence
$\star_Y \in \left[ 0, \frac{\Delta}{10} \right]$ or
$\star_Y \in \left[ C-\frac{\Delta}{10}, C \right]$. In the second case, we can
flip $[0, C]$ around its midpoint to reduce to the first case. So we can
assume that $\star_Y \in \left[ 0, \frac{\Delta}{10} \right]$.
Let $\widehat{F}$ be a quasi-inverse of $F$ and put
$q = \widehat{F}(0,0)$. If $\Lambda\Delta$ is sufficiently small then
we can ensure that $.9 \le \frac{\r_p}{\r_q} \le 1.1$. 
From Lemma \ref{faralignment1}, if 
$\beta_1$ is sufficiently small, relative to $\beta_E$, then
$q$ has a $(1, \beta_E)$-splitting. If in addition
$\eps$ is sufficiently small, relative to $\si_E$, then
$q$ is guaranteed to be in $E$.
Then $d(p,q) \le  \frac12\Delta \r_p < \Delta\r_q$.

To prove the second part of the lemma, Lemma \ref{nearedge} implies that if
$p^\prime \in E^\prime \cap B(p, 10 \Delta \r_p)$ then 
we can assume that $p^\prime$
lies within distance $\frac12 \r_p$
from $\widehat{F}([-100\Delta, 100\Delta] \times \{0\})$.
(This is not a constraint on the present parameter $\eps$.)
Choose $q = \widehat{F}(x,0)$ for some $x \in [-100\Delta, 100\Delta]$
so that 
$d(p',q) \le \frac12 \r_p$.
From Lemma \ref{faralignment1}, if 
$\beta_1$ is sufficiently small, relative to $\beta_E$, then
$q$ has a $(1, \beta_E)$-splitting. If in addition
$\eps$ is sufficiently small, relative to $\si_E$, then
$q$ is guaranteed to be in $E$.
If $\Lambda \Delta$ is sufficiently small then
$d(p',q) \le \r_q$.
This proves the lemma.
\end{proof}

\subsection{Regularization of the distance function  $d_{E'}$}
\label{subsec-regularizationofrhoe'}
Let $d_{E'}$ be the distance function from $E'$.
We will apply the smoothing results from Section 
\ref{subsec-smoothinglipschitz}
to  $d_{E'}$.  We will 
see that the resulting smoothing of the distance function from $E'$
defines part of a good coordinate
in a  collar region near $E$.

Let $\varsigma_{E'} > 0$ be a new parameter.

\begin{lemma}
\label{lem-rhoerhoe'}
Under the constraints
$\beta_{E'}<\overline{\beta}_{E'}(\De,\varsigma_{E'})$ and
$\si_{E'}<\overline{\si}_{E'}(\De,\varsigma_{E'})$
there is a function
$\rho_{E'}:M\ra [0,\infty)
$
such that if 
$\eta_{E'}=\frac{\rho_{E'}}{\r}$ then:
\begin{enumerate}
\item We have
\begin{equation}
  \left|\frac{\rho_{E'}}{\r}-\frac{d_{E'}}{\r}\right|\leq \varsigma_{E'}\,.
\end{equation}
\item  In the set
$
\eta_{E'}^{-1} \left[ \frac{\De}{10},10\De \right]
\cap (\frac{d_{E}}{\r})^{-1}[0,50\De]\,,
$
the function
$\rho_{E'}$ is smooth and its gradient lies in the $\varsigma_{E'}$-neighborhood
of the generalized gradient of $d_{E'}$. 

\item $\rho_{E'}-d_{E'}$ is $\varsigma_{E'}$-Lipschitz.

\end{enumerate}
\end{lemma}
\begin{proof}

Let $\eps_1\in (0,\infty)$ and $\th\in (0,\pi)$ be constants, to be determined
during the proof.

Put
\begin{equation}
C=\left(\frac{d_{E'}}{\r}\right)^{-1}
\left[ \frac{\De}{20},20\De \right]
\cap \left(\frac{d_E}{\r}\right)^{-1}[0,50\De]\,.
\end{equation}
If $x\in C$ and $\La<\overline{\La}(\De)$ then there exists a 
$p\in E$ such that $x\in B(p,60\De)\subset\frac{1}{\r_p}M$.
By Lemma \ref{nearedge}, provided that
$\be_{E'}<\overline{\be}_{E'}(\eps_1,\De)$,
$\si_{E'}<\overline{\si}_{E'}(\eps_1,\De)$,
$\be_E<\overline{\be}_E(\be_{E'},\si_{E'})$,
$\si_E<\overline{\si}_E(\be_{E'},\si_{E'})$,
and $\La<\overline{\La}(\eps_1,\De)$,
there is a $y\in M$ such that in $\frac{1}{\r_p}M$,
\begin{equation}
|d(y,x)-d_{E'}(x)|<\eps_1,\quad |d(y,E')-2d_{E'}(x)|<\eps_1\,.
\end{equation}

\begin{figure}[h]
\label{fig-9.12}
\begin{center}
 
\input{9.12.pspdftex}
 
\caption{}
\end{center}
\end{figure}

By triangle comparison,
involving triangles whose vertices are at $x$, $y$ and points
in $E'$ whose distance to $x$ is almost infimal,
it follows that if $\eps_1<\overline{\eps}_1(\th,\De)$
then $\diam(V_x)<\th$, where $V_x$ is the set  of initial velocities
of minimizing geodesic segments from $x$ to $E'$.  See Figure
3.

Applying Corollary \ref{cor-1strained}, 
if $\th<\overline{\th}(\varsigma_{E'})$ then
we obtain a function
$\rho_{E'}:M\ra [0,\infty)$ such that 
\begin{enumerate}
\item $\rho_{E'}$ is smooth in a neighborhood of $C$.
\item $\left\|\frac{\rho_{E'}}{\r}-\frac{d_{E'}}{\r}\right\|<\varsigma_{E'}$.
\item For every $x\in C$, the gradient of $\rho_{E'}$
lies in the $\varsigma_{E'}$-neighborhood of the generalized gradient 
of $d_{E'}$.
\item $\rho_{E'}-d_{E'}$ is $\varsigma_{E'}$-Lipschitz.
\end{enumerate}
If $\varsigma_{E'}<\frac{\De}{20}$ then 
\begin{equation}
\eta_{E'}^{-1} \left[ \frac{\De}{10},10\De \right] \cap 
(\frac{d_{E}}{\r})^{-1}[0,50\De]
\subset C\,,
\end{equation}
so the lemma follows.
\end{proof}

\subsection{Adapted coordinates tangent to the edge}
\label{subsec-edgetangent}

In this subsection, $p\in E$ will denote an edge
point and $Q_p$ will denote a map as in (\ref{edgeeqn}). 

Let $\varsigma_{\edge} > 0$ be a new parameter.
Applying Lemma \ref{lem1}, we get:

\begin{lemma}
\label{lem-existsetapedge}
Under the constraint 
$\beta_{E}<\overline{\beta}_E(\De,\varsigma_{\edge})$,
there is a
$Q_p$-adapted
coordinate 
\begin{equation}
\eta_p:\left( \frac{1}{\r_p} M, p \right)\supset B(p, 100\Delta)\ra \R
\end{equation}
of quality $\varsigma_{\edge}$.  
\end{lemma}
We define a global  function $\zeta_p:M\ra [0,1]$
by extending
\begin{equation}
\label{eqn-edgezetap}
\left(\Phi_{-9\De,-8\De,8\De,9\De}\circ\eta_p\right)
\cdot\left(\Phi_{8\De,9\De}\circ\eta_{E'}\right):
B(p,100\De)\ra [0,1]
\end{equation}
by zero.

\begin{lemma} \label{localprod}
The following holds:
\begin{enumerate}
\item $\zeta_p$ is smooth.
\item Under the constraints
$\beta_2 < \overline{\beta}_2(\varsigma_{\twostratum})$,
$\La < \overline{\La}(\varsigma_{\twostratum},\De)$,
$\be_{E'}<\overline{\be}_{E'}(\varsigma_{\twostratum},\De)$,
$\si_{E'}<\overline{\si}_{E'}(\varsigma_{\twostratum},\De)$,
$\be_E<\overline{\be}_E(\beta_2, \be_{E'}, \si_{E'}, \varsigma_{\twostratum})$,
$\si_E<\overline{\si}_E(\beta_2, \be_{E'}, \si_{E'}, \varsigma_{\twostratum})$,
$\varsigma_{E'}<\overline{\varsigma}_{E'}(\varsigma_{\twostratum})$ and
$\varsigma_{\edge}<\overline{\varsigma}_{\edge}(\varsigma_{\twostratum})$,
if
$x\in 
(\eta_p,\eta_{E'})^{-1}([-10\De,10\De]\times \left[ \frac{1}{10} \Delta,10\De \right])$
then $x$ is a $2$-stratum point, and there is
a $(2,\be_2)$-splitting
$\phi:(\frac{1}{\r_x}M,x)\ra\R^2$ such that
$(\eta_i,\eta_{E'}):(\frac{1}{\r_x}M,x)
\ra (\R^2,\phi(x))$  defines $\phi$-adapted coordinates of quality 
$\varsigma_{\twostratum}$ on the ball $B(x,100)\subset
\frac{1}{\r_x}M$.
\end{enumerate}
\end{lemma}

\begin{proof}
(1). This follows from Lemma \ref{lem-rhoerhoe'}.

(2).  Let $\eps_1,\ldots,\eps_5>0$ be constants, to 
be chosen at the end of the proof.
For $i\in \{1,2\}$ choose points $x_i^{\pm}\in M$ such that
$Q_p(x_i^\pm)\in B(Q_p(x)\pm \frac{\De}{20}e_i,
\si_E) \subset \R\times [0,C]$.
Provided that 
$\be_E<\overline{\be}_E(\beta_2,\De,\eps_1)$,
$\si_E<\overline{\si}_E(\beta_2,\De,\eps_1)$
and $\La<\overline{\La}(\De)$,
the tuple $\{x_i^\pm\}_{i=1}^2$ will
be a $2$-strainer at $x$ of quality $\eps_1$, and scale at least
$\frac{\De}{30}$ in $\frac{1}{\r_x}M$. Therefore,
if $\eps_1<\overline{\eps}_1(\be_2)$ then  $x$
will be a $2$-stratum point, with a $(2,\be_2)$-splitting 
$\phi:(\frac{1}{\r_x}M,x)\ra (\R^2,0)$ given by strainer coordinates 
as in Lemma \ref{strainerlemma}.

\begin{figure}[h]
\label{fig-9.17}
\begin{center}
 
\input{9.17.pspdftex}
 
\caption{}
\end{center}
\end{figure}

Suppose that $y$ is a point in $\frac{1}{\r_x}M$
with $d(y,x)<\eps_2\cdot\frac{\De}{20}$,
and $z\in E'$ is a point with $d(y,z) \le d(y,E') + \eps_2 \Delta$;
see Figure 4. 
Then by Lemma \ref{nearedge}, 
if $\be_{E'}<\overline{\be}_{E'}(\eps_3,\De)$,
$\si_{E'}<\overline{\si}_{E'}(\eps_3,\De)$,
$\be_E<\overline{\be}_E(\eps_3,\be_{E'}, \si_{E'})$,
$\si_E<\overline{\si}_E(\eps_3,\be_{E'}, \si_{E'})$,
$\La<\overline{\La}(\eps_3,\De)$ and
$\eps_2<\overline{\eps}_2(\eps_3)$ then
the comparison angles
$\cangle_y(x_1^\pm,z)$, $\cangle_y(x_2^+,z)$ will
satisfy
$|\cangle_y(x_1^\pm,z)-\frac{\pi}{2}|<\eps_3$ and
$|\cangle_y(x_2^+,z)-\pi|<\eps_3$.
If $\ga_i^\pm$ is a minimizing segment
from $y$ to $x_i^\pm$, and $\ga_z$ is a minimizing segment 
from $y$ to $z$, it follows that 
$|\angle_y(\ga_1^\pm,\ga_z)-\frac{\pi}{2}|<\eps_4$ and 
$|\angle_y(\ga_2^+,\ga_z)
-\pi|<\eps_4$, provided that $\eps_i<\overline{\eps}_i(\eps_4)$
for $i\leq 3$.   Therefore $|D\eta_{E'}((\ga_1^\pm)'(0))|<\eps_5$ and
$|D\eta_{E'}((\ga_2^+)'(0))-1|<\eps_5$, provided that
$\eps_4<\overline{\eps}_4(\eps_5)$ and 
$\varsigma_{E'}<\overline{\varsigma}_{E'}(\eps_5)$.  Likewise, 
$|D\eta_p(\ga_1^\pm)'(0))-1|<\eps_5$ and $|D\eta_p((\ga_2^\pm)'(0))|<\eps_5$,
provided that $\be_E<\overline{\be}_E(\eps_5)$ and 
$\varsigma_{\edge}<\overline{\varsigma}_{\edge}(\eps_5)$.
It follows from Lemma \ref{lem-uniquenessofadaptedcoords} that $(\eta_p,\eta_{E'})$ defines
$\phi$-adapted coordinates of quality $\varsigma_{\twostratum}$ on 
$B(x,100)\subset \frac{1}{\r_x}M$, provided that
$\eps_5<\overline{\eps}_5(\varsigma_{\twostratum})$ and $\beta_2 < \overline{\beta}_2(\varsigma_{\twostratum})$.  

We may fix the 
constants in the order $\eps_5,\ldots,\eps_1$.
The lemma follows.
\end{proof}

\subsection{The topology of the edge region}

In this subsection we determine the topology of a suitable neighborhood
of an edge point 
$p \in E$.

\begin{lemma}
\label{lem-edgetopology}
Under the constraints 
$\beta_{E'} < \overline{\beta}_{E'}(\De)$,
$\si_{E'} < \overline{\si}_{E'}(\De)$,
$\be_{E}<\overline{\be}_E(\be_{E'},\si_{E'},w')$,
$\si_{E}<\overline{\si}_E(\be_{E'},\si_{E'})$,
$\varsigma_{\edge} < \overline{\varsigma}_{\edge}(\De)$,
$\varsigma_{E'} < \overline{\varsigma}_{E'}(\De)$, 
$\La<\overline{\La}(\De)$ and 
$\sigma < \overline{\sigma}(\De)$, 
the map $\eta_p$ restricted
to $(\eta_p, \eta_{E^\prime})^{-1}((-4\De,4\De)\times (-\infty,4\De]) $
is a fibration with fiber 
diffeomorphic to the closed $2$-disk $D^2$.
\end{lemma}
\begin{proof}
Let $\eps > 0$ be a constant which will be internal to this proof.

By Lemma \ref{smoothconv}, if $\beta_{E}<\overline{\be}_E(\eps,\De,w')$
then 
the map $F_p$ of Definition \ref{edgepoints} 
is $\eps$-close to a $(1,\eps)$-splitting 
$\phi:\left( \frac{1}{\r_p} M, p \right) \ra (\R \times Z,(0,\star_Z))$,
where $Z$ is a complete pointed
nonnegatively curved 
$C^K$-smooth
surface, and $\phi$ is $\eps$-close in the 
$C^{K+1}$-topology
to an isometry between the ball $B(p,1000\De)\subset 
\left( \frac{1}{\r_p} M, p \right)$ and its image
in $(\R \times Z, (0,\star_Z))$.
If in addition $\sigma_E < \overline{\si}_E(\Delta)$ then
the pointed ball
$(B(\star_Z,10\De),\star_Z)\subset (Z,\star_Z)$
will have pointed-Gromov-Hausdorff distance at most $\delta$
from the pointed interval $([0,10],0)$, where $\delta$ is the parameter
of Lemma \ref{surfacelemma}.
By Lemma \ref{surfacelemma}, 
we conclude that $\overline{B(\star_Z,\De)}$ is homeomorphic
to the closed $2$-disk.

Put $Y=\R\times \{\star_Z\}\subset \R\times Z$
and let $d_Y : \R \times Z \rightarrow \R$ be the distance to $Y$.
By Lemma
\ref{surfacelemma}, if $\eps$ is sufficiently small then for every 
$x\in \R\times Z$ with $d_Y(x)\in [\De,9\De]$, 
the set $V_x$ of initial velocities of minimizing segments
from $x$ to $Y$ has small diameter; moreover $V_x$ is orthogonal
to the $\R$-factor of $\R\times Z$.  Thus we may apply 
Lemma \ref{cor-1strained} to find a smoothing $\rho_Y$ of 
$d_Y$, where $\|\rho_Y-d_Y\|_\infty$ is small, and in 
$d_Y^{-1}(\De,9\De)$ the gradient of $\rho_Y$
is close to the generalized gradient of $d_Y$.

Note that by Lemma \ref{nearedge}, we may assume that
$\phi^{-1}(\R\times \{\star_Z\})\cap B(p,50\De)$  
is Hausdorff close  to $E'\cap B(p,50\De)$.  Since $\phi$ is 
$C^{K+1}$-close
to an isometry,  the 
generalized gradient of $d_Y\circ\phi$ will be close to the generalized
gradient of $d_{E'}$ in the region 
$(\eta_p,\eta_{E'})^{-1}(-5\De,5\De)\times (2\De,5\De)$,
where the gradients are taken with respect to the metric on  
$\frac{1}{\r_p}M$.  (One may see this by applying a compactness
argument to conclude that
minimizing geodesics to $E'$  in this region are mapped by $\phi$
to be $C^1$-close to minimizing geodesics to $Y$.)
Hence if $\varsigma_{E'}$ is small then  $\eta_{E'}$ 
and $\rho_Y\circ\phi$  will be $C^1$-close
in the region 
$(\eta_p,\eta_{E'})^{-1}(-5\De,5\De)\times (2\De,5\De)$.
Similarly, if $\be_E$ and  $\varsigma_{\edge}$ are small
then $\eta_p$ will be $C^1$-close to
$\pi_Z\circ \phi$ in the  region
$(\eta_p,\eta_{E'})^{-1}(-5\De,5\De)\times (-\infty,5\De)$.

For $t\in [0,1]$, define a map
$
f^t:(\eta_p,\eta_{E'})^{-1}\left((-5\De,5\De)\times 
(-\infty,5\De)\right)\ra \R^2
$
by 
\begin{equation}
f^t=(t\,\eta_p+(1-t)\pi_Z\circ\phi\,,\,t\,
\eta_{E'}+(1-t)\rho_Y\circ\phi).
\end{equation}
Let $F:(\eta_p,\eta_{E'})^{-1}\left((-5\De,5\De)\times 
(-\infty,5\De)\right)\times[0,1]\ra \R^2$ be the map
with slices $\{f^t\}$.  
In view of the $C^1$-closeness discussed above,
we may now apply Lemma \ref{lem-isotopic1}  
to conclude that
$(\eta_p,\eta_{E'})^{-1}(\{0\}\times(-\infty,4\De])$ is diffeomorphic
to $(\pi_Z,\rho_Y)^{-1}(\{0\}\times(-\infty,4\De])$, which
is  a closed $2$-disk.

Finally, we claim that  the restriction of
$\eta_p$ to 
$(\eta_p,\eta_{E'})^{-1}(-4\De,4\De)\times (-\infty,4\De]$
yields a proper submersion to 
$(-4\De,4\De)$, and is therefore a fibration.  The properness
follows from the fact that 
$(\eta_p,\eta_{E'})^{-1}((-4\De,4\De)\times (-\infty,4\De])$
is contained in a compact subset of the domain of $(\eta_p,\eta_{E'})$.
The fact that it is a submersion follows from the nonvanishing
of $D\eta_p$, and the linear independence of $\{D\eta_p,D\eta_{E'}\}$
at points with $(\eta_p,\eta_{E'})\in (-4\De,4\De)\times \{4\De\}$.
\end{proof}

\begin{remark} Given $w^\prime$, we take $\beta_E$ very small in
the proof of Lemma \ref{lem-edgetopology} in order to get a very
good $1$-splitting.  On the other hand, we just have to take $\sigma_E$,
and hence $\sigma$, small
enough to apply Lemma \ref{surfacelemma}; the parameter $\delta$ of
Lemma \ref{surfacelemma} is independent of $w^\prime$.
This will be important for the order in which we choose the parameters.
\end{remark}

\subsection{Selection of edge balls}

 Let $\{p_i\}_{i\in I_{\edge}}$ be a maximal set of edge points with the property
that the collection 
$\{B(p_i,\frac13\De\r_{p_i})\}_{i\in I_{\edge}}$
is disjoint.
We write $\zeta_i$ for $\zeta_{p_i}$.

\begin{lemma}
\label{lem-edgeballcover}
Under the constraints ${\mathcal M}>\overline{\mathcal M}$ and
$\La<\overline{\La}(\De)$,
\begin{itemize}
\item 
$\bigcup_{i\in I_{\edge}}\;B(p_i,\De\r_{p_i})$
 contains $E$.
\item The intersection multiplicity of the collection 
$\{\supp(\zeta_i)\}_{i\in I_{\edge}}$
is bounded by $\mathcal{M}$.
\end{itemize}

\end{lemma}
\begin{proof}
We omit the proof, as it is similar to the proof
of Lemma \ref{mult2}.
\end{proof}

We now give a useful covering of the $1$-stratum points.

\begin{lemma}
\label{lem-1stratumslimoredge}
Under the constraint $\La<\overline{\Lambda}(\Delta) $,
any $1$-stratum
point lies in the slim $1$-stratum or lies in  
$\bigcup_{i \in I_{\edge}} B(p_i,
3 \Delta \r_{p_i})$.
\end{lemma} 
\begin{proof}
From Lemma \ref{hitsall}, if a $1$-stratum point does not lie in the slim
$1$-stratum then it lies in
$B(p, \Delta \r_{p})$ for some $p \in E$.
By Lemma \ref{lem-edgeballcover} we have $p \in B(p_i,\De\r_{p_i})$ for
some $i\in I_{\edge}$.
If $\Lambda\Delta$ is sufficiently small then we can assume that
$.9 \le \frac{\r_p}{\r_{p_i}} \le 1.1$. The lemma follows.
\end{proof}

The next lemma will be used later for the interface between the
slim stratum and the edge stratum.

\begin{lemma}
\label{lem-smalleretai}
Under the constraints 
$\be_E<\overline{\be}_E(\De,\be_2)$,  
$\varsigma_{\edge}<\overline{\varsigma}_{\edge}(\De,\be_2)$
and 
$\La<\overline{\La}(\De)$, the following holds.
Suppose for some $i\in I_{\edge}$ and $q\in B(p_i,10\De\r_{p_i})$
we have 
\begin{equation}
\eta_{E'}(q)<5\De\,,\quad |\eta_{p_i}(q)|< 5\De\,.
\end{equation}
Then either $p_i$ belongs to the slim $1$-stratum, or
there is a $j\in I_{\edge}$ such that 
$q\in B(p_j,10\De\r_{p_j})$ and $|\eta_{p_j}(q)|<2\De$.

\end{lemma}
\begin{proof}
We may assume that $p_i$ does not belong to the slim $1$-stratum.

From the definition of $\eta_i$ and
Lemma \ref{nearedge},  provided that
$\be_{E'}<\overline{\be}_{E'}(\De)$,
$\be_E<\overline{\be}_E(\be_{E'})$,
$\La<\overline{\La}(\De)$ and 
$\varsigma_{\edge}<\overline{\varsigma}_{\edge}(\De)$,
there will be a $q'\in E'\cap B(p_i,10\De\r_{p_i})$ such
that $|\eta_{p_i}(q')-\eta_{p_i}(q)|<\frac{1}{10}\De$.  
Since $p_i$ is not 
in the slim $1$-stratum, by Lemma \ref{hitsall}(2)
there is a $p\in E$ such that $q'\in B(p,\r_p)$,
and by Lemma \ref{lem-edgeballcover} we have $p\in B(p_j,\De\r_{p_j})$
for some $j\in I_{\edge}$.  If $\La\De$ is small then we will
have $|\eta_{p_j}(q')|<1.5\De$.  Lemma \ref{adaptedclose} now implies that if
$\be_{\edge}<\overline{\be}_{\edge}(\De,\be_2)$  and 
$\varsigma_{\edge}<\overline{\varsigma}_{\edge}(\De,\be_2)$ then 
$|\eta_{p_j}(q')-\eta_{p_j}(q)|<\frac12 \De$.  
Thus $|\eta_{p_j}(q)|<2\De$.
\end{proof}

\subsection{Additional cutoff functions}
\label{subsec-additionalcutoff}

We define two additional cutoff functions for later use:
\begin{equation} \label{additional1}
\zeta_{\edge}=1-\Phi_{\frac12,1}\circ\left(
\sum_{i\in I_{\edge}}\,\zeta_i\,\right)
\end{equation}
and
\begin{equation} \label{additional2}
\zeta_{E'}=\left(\Phi_{\frac{2}{10}\De,\frac{3}{10}\De,8\De,9\De}\circ
\eta_{E'}\right)\cdot
\zeta_{\edge}.
\end{equation}

\section{The local geometry of the slim $1$-stratum}
\label{sec-locslim}

In this section we consider the slim $1$-stratum points.  Along with
the adapted coordinates and cutoff functions, we discuss the
local topology and a selection process to get a ball covering of the
slim $1$-stratum points.

\subsection{Adapted coodinates, cutoff functions and local topology
near slim $1$-stratum points}
\label{adaptedslim}

In this subsection, we let $p$ denote a point in the 
slim $1$-stratum, and
$\phi_p : \left( 
\frac{1}{\r_p} M, p \right) \rightarrow (\R \times X,
(0, \star_X))$
be
a $(1, \beta_1)$-splitting, with $\diam(X) \le 10^3 \Delta$. 
Let $\varsigma_{\slim} > 0$ be a new parameter.

\begin{lemma} \label{1adapted}

Under the constraint $\be_1<\overline{\be}_1(\De,\varsigma_{\slim})$:

\begin{itemize}
\item 
There is a $\phi_p$-adapted
coordinate $\eta_p$  of quality $\varsigma_{\slim}$ on 
$B(p, 10^6  \Delta) 
\subset \left( \frac{1}{\r_p} M, p \right)$.
\item The 
cutoff function
\begin{equation}
\left(\Phi_{-9 \cdot 10^5 \Delta,-8 \cdot 10^5 \Delta,
8 \cdot 10^5 \Delta, 9 \cdot 10^5 \Delta} \right)
\,\circ\, \eta_p
\end{equation}
extends by zero to a smooth 
function $\zeta_p$ on $M$.
\end{itemize}
\end{lemma}
\begin{proof}
This follows from
Lemma \ref{lem1} (see also Remark \ref{addedremark}).
\end{proof}

Let $\eta_p$ and $\zeta_p$ be as in  Lemma \ref{1adapted}.

\begin{lemma}
\label{lem-slimtopology}
Under the constraints
$\be_1<\overline{\be}_1(\varsigma_{\slim},\De,w')$, 
$\varsigma_{\slim}<\overline{\varsigma}_{\slim}(\De)$, then  $\eta_p^{-1}\{0\}$ is diffeomorphic to
$S^2$ or $T^2$.
\end{lemma}
\begin{proof}
From Lemma \ref{smoothconv}, if $\beta_1<\overline{\be}_1(\De,w')$
then close to $\phi_p$, there is a 
$(1,\beta_1)$-splitting
$\phi:\left( \frac{1}{\r_p} M, p \right) \ra (\R \times Z,(0,\star_Z))$ 
for some
complete pointed nonnegatively curved 
$C^K$-smooth
surface $(Z,\star_Z)$, with the map being 
$C^{K+1}$-close
to an isometry on $B(p, 10^6  \Delta)$. From Definition
\ref{1stratumtypes}, the diameters of the $Z$-fibers
are at most $10^4 \De$. In particular, since $Z$ is compact and $M$ is
orientable, $Z$ must be diffeomorphic to $S^2$ or $T^2$.  Furthermore,
we may assume that for any pair of 
points $m,m'\in M$ with  $m\in B(p,10^6\De)\subset\frac{1}{\r_p}M$, 
$d(m,m')\in [2,10]$, and $\pi_Z(\phi(m))=\pi_Z(\phi(m'))$, the initial
velocity $v$ of a minimizing segment $\ga$ from $m$ to $m'$ maps 
under $\phi_*$ to a vector almost tangent to the $\R$-factor of
$\R\times Z$.

As $\phi$ is close to $\phi_p$, we may assume that  
$\eta_p$ is a 
$\phi$-adapted coordinate of quality
$2\varsigma_{\slim}$.  
If $\varsigma_{\slim}<\overline{\varsigma}_{\slim}(\De)$, then
we may apply the estimate from the preceding paragraph, and
the definition of adapted coordinates  
(specifically (\ref{adapted})), to conclude that $\eta_p$ is 
$C^1$-close  to the composition 
$\frac{1}{\r_p}M\stackrel{\phi}{\ra}\R \times Z \ra \R$
on the ball $B(p, 10^6 \Delta)$. The lemma now
follows from Lemma
\ref{lem-isotopic2}.
\end{proof}

\subsection{Selection of slim $1$-stratum balls}

Let $\{p_i\}_{i\in I_{\slim}}$ be a maximal set of slim stratum points with the property
that the collection $\{B(p_i,\frac13\De\r_{p_i})\}_{i\in I_{\slim}}$ is disjoint. We write $\zeta_i$ for $\zeta_{p_i}$.

\begin{lemma}
\label{lem-slimballselection}
Under the constraints ${\mathcal M}>\overline{\mathcal M}$ and
$\La<\overline{\La}(\De)$,
\begin{itemize}
\item $\bigcup_{i\in I_{\slim}}\;B(p_i,\De\r_{p_i})$ contains all 
slim stratum points.
\item The intersection multiplicity of the collection 
$\{\supp(\zeta_i)\}_{i\in I_{\slim}}$
is bounded by $\mathcal{M}$.
\end{itemize}

\end{lemma}
We omit the proof, as it is similar to the proof of
Lemma \ref{mult2}.

\section{The local geometry of the $0$-stratum}
\label{sec-loc0stratum}

Thus far, points
in the $0$-stratum have been defined by a process of elimination (they are 
points that are neither $2$-stratum points nor $1$-stratum points) rather 
than by the
presence of some particular geometric structure.  We now discuss
their geometry.   We show in  
Lemma \ref{lemgoodannulus} that $M$ has conical structure near
every point -- not just the $0$-stratum
points -- provided one looks at an appropriate scale larger than $\r$. 
We then use this  to define   radial 
and cutoff functions near $0$-stratum points.

Let  $\de_0 > 0$  and 
$\Upsilon_0, \Upsilon_0',\tau_0 > 1$ 
be new parameters.

\subsection{The Good Annulus Lemma} \label{subsecann}

We now show that for every point $p$ in $M$, there is a scale
at which a neighborhood of $p$
is well approximated by a model geometry in two different
ways: by a nonnegatively curved $3$-manifold in the 
pointed 
$C^K$-topology,
and by the Tits cone of this manifold
in the pointed Gromov-Hausdorff topology.

\begin{lemma}
\label{lemgoodannulus}
Under the constraint 
$\Upsilon'_0>
\overline{\Upsilon}'_0(\de_0,\Upsilon_0,w')$,
if $p\in M$ 
then there exists
$r^0_p \in [\Upsilon_0 \r_p,\Upsilon'_0 \r_p]$ and a 
complete $3$-dimensional nonnegatively curved 
$C^K$-smooth
Riemannian manifold $N_p$ such that:
\begin{enumerate}
\item
$(\frac{1}{r^0_p} \, M,p)$ is $\de_0$-close in the pointed Gromov-Hausdorff topology to
the Tits cone $C_TN_p$ of $N_p$.

\item
The  ball 
$B(p,r^0_p) \subset M$ is diffeomorphic
 to $N_p$.
\item The distance function from $p$ has no  critical points in the annulus
 $A(p, \frac{r^0_p}{100}, r^0_p)$.
\end{enumerate}
\end{lemma}

\begin{proof}
Suppose that conclusion (1) does not hold. 
Then for each $j$, if we take  $\Upsilon^\prime_0 = j \Upsilon_0$, it
is not true that
conclusion (1) holds for sufficiently large $\alpha$. 
Hence we
can find a sequence $\alpha_j \rightarrow \infty$ so that for each $j$,
$(M^{\alpha_j},
p_{\alpha_j})$ provides a counterexample with 
$\Upsilon^\prime_0 = j \Upsilon_0$.

For convenience of notation, we relabel
$(M^{\alpha_j}, p_{\alpha_j})$ as $(M_j,p_j)$ and write
$\r_j$ for $\r_{p_{\alpha_j}}$.
Then by assumption, for each $r^0_j \in [\Upsilon_0 \r_j, j\Upsilon_0 \r_j]$ there 
is no $3$-dimensional
nonnegatively curved Riemannian manifold 
$N_j$ such that  conclusion (1) holds.

Assumption \ref{assumptions}
implies that a subsequence of $\{(\frac{1}{\r_j}M_j,p_j)\}_{j=1}^\infty$, which we relabel as
$\{(\frac{1}{\r_j}M_j,p_j)\}_{j=1}^\infty$,
converges in the pointed 
$C^{K}$-topology
to a pointed $3$-dimensional 
nonnegatively curved 
$C^{K}$-smooth
Riemannian manifold $(N,p_\infty)$.  
Now $N$ is asymptotically
conical.  That is, there is some $R > 0$ so that if $R^\prime > R$ then
$(\frac{1}{R^\prime} N, p_\infty)$ is $\frac{\de_0}{2}$-close 
in the pointed Gromov-Hausdorff topology
to the Tits cone $C_TN$.

By critical point theory, large 
open balls in $N$ are diffeomorphic to 
$N$ itself.
Hence we can find 
$R^\prime > 10^3 \max \left(\Upsilon_0 , R \right)$ 
so that for any $R^{\prime \prime} \in 
\left( \frac12 R^\prime, 2 R^\prime \right)$, 
there are no critical points of the distance function from
$p_\infty$ in 
$A(p_\infty, \frac{R^{\prime \prime}}{10^3}, 10 R^{\prime \prime}) \subset 
N$, and
the ball
$B(p_\infty, R^{\prime \prime})$ is diffeomorphic to $N$.
In view of the convergence $(\frac{1}{\r_j} M_j,p_j)\ra (N,p_\infty)$ in
the pointed 
$C^{K}$-topology, it follows that
for large $j$ 
there are no critical points of the distance function in
$A(p_j, \frac{R^{\prime \prime} \r_j}{100)},  
R^{\prime \prime} \r_j) \subset N_j$, and
$B(p_j, R^{\prime \prime} \r_j) \subset M_j$ is diffeomorphic to
$B(p_\infty, R^{\prime \prime}) \subset N$. Taking 
$r_j^0 = R^\prime \r_j$ gives
a contradiction.
\end{proof}

\begin{remark}
If we take the parameter $\sigma$ of Lemma \ref{lemsicloseto2d} 
to be small
then we can additionally conclude that $C_TN_p$ is pointed
Gromov-Hausdorff close to a conical nonnegatively curved
Alexandrov space of dimension
at most two.
\end{remark}

\subsection{The radial function near a $0$-stratum point}
\label{subsec-0stratumscale}

For every $p\in M$, we apply 
Lemma \ref{lemgoodannulus} to get 
a scale
$r^0_p \in \left[ \Upsilon_0 \r_p, \Upsilon_0^\prime \r_p \right]$ 
for which the conclusion of
Lemma \ref{lemgoodannulus} holds.
In particular, 
$\left( \frac{1}{r^0_p} M, p \right)$ is $\de_0$-close in the
pointed Gromov-Hausdorff topology
to the Tits cone $C_T N_p$ of a nonnegatively curved $3$-manifold $N_p$. 

Let $d_{p}$ be the distance
function from $p$ in $\left( \frac{1}{r^0_p} M, p \right)$. 
Let $\varsigma_{\zeroball} >0$ be a new parameter.
  
\begin{lemma}
\label{lem-existsetap0stratum}
Under the constraint $\delta_0 <
\overline{\delta}_0(\varsigma_{\zeroball})$,
there is a  function $\eta_p:\frac{1}{\r^0_p}M\ra [0,\infty)$ 
such that:
\begin{enumerate}
\item $\eta_p$ is smooth 
on
$A \left( p, \frac{1}{10}, 10 \right) \subset  
\frac{1}{r^0_p} M$.
\item 
$\parallel \eta_p - d_p \parallel_\infty < \varsigma_{\zeroball}$.
\item 
$\eta_p - d_p:\frac{1}{\r^0_p}M\ra [0,\infty)$ 
is $\varsigma_{\zeroball}$-Lipschitz.
\item $\eta_p$ is smooth and has no critical points in 
$\eta_p^{-1}([\frac{2}{10},2])$, and for every
$\rho\in [\frac{2}{10},2]$, the sublevel set $\eta_p^{-1}([0,\rho])$
is diffeomorphic to 
either the closed disk bundle in the normal bundle 
$\nu S$ of the
soul $S\subset N_p$, if $N_p$ is noncompact, or to $N_p$ itself
when $N_p$ is compact.
\item The 
 composition
$
\Phi_{ \frac{2}{10},\frac{3}{10},\frac{8}{10},\frac{9}{10}} \circ 
\eta_p$
extends by zero to  a smooth cutoff function $\zeta_p:M\ra [0,1]$.
\end{enumerate}
\end{lemma}
\begin{proof}
We apply Lemma \ref{cor-1strained} with $Y = \{p\}$, $U =
A \left( p, \frac{1}{20}, 20 \right)$ and
$C = \overline{A \left( p, \frac{1}{10}, 10 \right)}$.
To verify the hypotheses of Lemma \ref{cor-1strained},
suppose that $q \in U$. From Lemma \ref{lemgoodannulus},
for any $\mu >0$, there is an $\overline{\delta}_0 =
\overline{\delta}_0(\mu)$ so that if $\delta_0 < 
\overline{\delta}_0$ then we can find some $q^\prime \in M$
with $d(p, q^\prime) = 2 d(p,q)$ and $d(q, q^\prime) \ge
(1-\mu) d(p,q)$. Fix a minimizing geodesic $\gamma_1$
from $q$ to $q^\prime$. By triangle comparison, for any 
$\theta > 0$, if $\mu$ is sufficiently small then we can
ensure that for any minimizing geodesic $\gamma$ from $q$ to
$p$, the angle between $\gamma^\prime(0)$ and $\gamma_1^\prime(0)$
is at least $\pi - \frac{\theta}{2}$. Parts (1), (2) and (3)
of the lemma now follow
from Lemma \ref{cor-1strained}. 

(4).  Using the same proof as the ``Morse lemma'' for distance functions, 
one gets a smooth vector field $\xi$ in $A(p,\frac{1}{10},10)$, such
that $\xi d_p$ and $\xi\eta_p$ are both close to $1$.  Using the 
flow of $\xi$, if $\varsigma_{\zeroball}$
is sufficiently small then for every $\rho\in [\frac{2}{10},2]$, 
the sublevel sets $d_p^{-1}([0,\rho])$ and $\eta_p^{-1}([0,\rho])$
are homeomorphic.  Let $\bar N$ be the closed disk bundle in $\nu S$.
Then $\Int(\bar N)\stackrel{homeo}{\simeq} N_p
\stackrel{homeo}{\simeq}\Int(d_p^{-1}([0,\rho]))\stackrel{homeo}{\simeq}
\Int(\eta_p^{-1}([0,\rho]))$.  Since two compact 
orientable $3$-manifolds with boundary
are homeomorphic provided that their interiors are homeomorphic, we
have
$\bar N\stackrel{homeo}{\simeq}\eta_p^{-1}([0,\rho])$.
(This may be readily deduced from the fact that if $S$ is a closed
orientable surface then any  smooth embedding 
$S\ra S\times \R$, which is also a homotopy equivalence, is isotopic 
to the fiber $S\times \{0\}$, as follows from the Schoenflies
theorem when $S=S^2$ and from \cite{stallings} when $\genus(S)>0$.)

(5) follows from the
fact that the composition 
$
\Phi_{ \frac{2}{10},\frac{3}{10},\frac{8}{10},\frac{9}{10}} \circ 
\eta_p$
is compactly
supported in the annulus 
$A(p,\frac{1}{10},10)$.
\end{proof}

\begin{remark}
One may avoid the Schoenflies and Stallings theorems in the proof
of Lemma \ref{lem-existsetap0stratum}(4).  If $M$ is a 
complete noncompact nonnegatively
curved manifold, and $p\in M$, then the distance function $d_p$
has no critical points outside $B(p,r_0)$ for some $r_0\in (0,\infty)$.  
In fact, for every $r>r_0$, the closed ball $\ol{B(p,r)}$ 
is isotopic, by an isotopy with arbitrarily small tracks, to
a compact domain with smooth boundary  $D$;
moreover, the 
smooth 
isotopy class $[D]$
 is canonical and independent of $r\in (r_0,\infty)$.
(These assertions are true in general for noncritical sublevel
sets of proper distance functions.  They are proved by showing that one may
smooth $d_p$ near $S(p,r)$ without introducing 
critical points.)  The proof of the soul theorem actually shows
that the isotopy class $[D]$ is the same as that of 
a closed smooth tubular neighborhood of the soul, which is
diffeomorphic to the unit normal bundle of the soul.
\end{remark}

\subsection{Selecting the $0$-stratum balls}
\label{annular}

The next lemma has a statement about an adapted coordinate for the
radial splitting in an annular region of a $0$-stratum ball.
We use the parameter $\varsigma_{\slim}$ for the quality of this
splitting, even though there is no {\it a priori} relationship
to slim $1$-stratum points.  Our use of this parameter will
simplify the later parameter ordering.

\begin{lemma} 
\label{0-ballseparation}

Under the constraints
$\de_0< \bar\de_0(\beta_1,\varsigma_{\slim})$, 
$\Upsilon_0>\overline{\Upsilon}_0(\be_1)$, 
$\be_1<\overline{\be}_1(\varsigma_{\slim})$ and
$\varsigma_{\zeroball}<\overline{\varsigma}_{\zeroball}(\varsigma_{\slim})$,
there is a finite
collection $\{p_i\}_{i \in I_{\zeroball}}$ of points in $M$ so that 
\begin{enumerate}
\item 
\label{item-disjoint}
The balls $\{B(p_i, r^0_{p_i})\}_{i \in I_{\zeroball}}$ are disjoint. 
\item
\label{item-radiusratio}
If $q\in B(p_i,10r^0_{p_i})$, for some $i\in I_{\zeroball}$, then
$r^0_q\leq 20r^0_{p_i}$ and
$\frac{r^0_{p_i}}{\r_q}\geq \frac{1}{20} \Upsilon_0$.
\item 
\label{item-1or2}
For each $i$, every $q\in  A(p_i,\frac{1}{10}r^0_{p_i},10r^0_{p_i})$
belongs to the $1$-stratum or $2$-stratum, and 
there is a $(1,\be_1)$-splitting
of $\left( \frac{1}{\r_q}M, q \right)$, for which 
$\frac{r^0_{p_i}}{\r_q}\;\eta_{p_i}$ is an adapted coordinate of quality 
$\varsigma_{\slim}$.
\item 
\label{item-catchall0stratum}
$\bigcup_{i\in I_{\zeroball}}\;B(p_i,\frac{1}{10}r^0_{p_i})$
contains all the $0$-stratum points.
\item 
\label{item-atmostoneend}
For each $i\in I_{\zeroball}$, 
the manifold $N_{p_i}$ has at
most one end.

\end{enumerate}
\end{lemma}
\begin{proof}

(\ref{item-disjoint}).
Let $V_0\subset M$ be the set of points $p\in M$ such that the 
ball $B(p,r^0_p)$ contains a $0$-stratum point.  We partially order 
$V_0$ by declaring that $p_1\prec p_2$ if and only if
($2r^0_{p_1}<r^0_{p_2}$ and 
$B(p_1,r^0_{p_1})\subset B(p_2,r^0_{p_2})$).  Note that every chain
in the poset $(V_0,\prec)$ has an upper bound, since $r^0_p<\Upsilon_0'\r_p$
is bounded above.  
Let $V\subset V_0$ be the subset of elements which
are maximal with respect to $\prec$,
and apply Lemma \ref{ballcovers} with ${\mathcal R}_p = r^0_p$ to get
the finite disjoint collection of balls  
$\{B(p_i, r^0_{p_i})\}_{i \in I_{\zeroball}}$.  Thus (\ref{item-disjoint})
holds.

(\ref{item-radiusratio}).  If $q\in B(p_i,10r^0_{p_i})$
then $r^0_q\leq 20r^0_{p_i}$, for otherwise we would have 
$q\in V_0$ and $p_i\prec q$, contradicting the maximality of $p_i$.
Thus $\frac{r^0_{p_i}}{\r_q}\geq\frac{r^0_q}{20\r_q}\geq \frac{1}{20} \Upsilon_0$.

(\ref{item-1or2}). Suppose that $i\in I_{\zeroball}$ and   
$q\in A(p_i,\frac{1}{10}r^0_{p_i},10r^0_{p_i})$.  
Recall that
$\left( \frac{1}{r^0_{p_i}} M, p_i \right)$ is $\delta_0$-close in the pointed
Gromov-Hausdorff topology to the Tits cone $C_T N_{p_i}$. 
If the Tits cone $C_T N_{p_i}$
were a single 
point then $\diam(M)$ would be bounded above by $\delta_0 r^0_{p_i}$;
taking $\delta_0 < \frac{1}{10}$ we get $q \in B(p_i, \frac{1}{10} r^0_{p_i})$,
which is a contradiction.  Therefore 
$C_T N_{p_i}$ is not a point. 
It follows that there is a $1$-strainer at $q$ of scale comparable 
to $r^0_{p_i}$ and quality comparable to $\delta_0$, where one of the strainer
points is $p_i$.
By (\ref{item-radiusratio}),
if $\Upsilon_0>\overline{\Upsilon}_0(\beta_1)$ and 
$\de_0<\overline{\de}(\beta_1)$ then Lemma \ref{strainerlemma}
implies there is a $(1,\be_1)$-splitting
$\al:\left( \frac{1}{\r_{q}} M, q \right)\ra (\R\times X,(0,\star_X))$,
where the first component is given by $d_{p_i}-d_{p_i}(q)$.
In particular 
$q$ is a $1$-stratum point or a $2$-stratum point.  By
Lemma \ref{lem-existsetap0stratum}, the smooth radial
function $\eta_{p_i}$ is $\varsigma_{\zeroball}$-Lipschitz close
to $d_{p_i}$. Lemma \ref{lem-uniquenessofadaptedcoords} implies that  
if
$\varsigma_{\zeroball}<\overline{\varsigma}_{\zeroball}(\varsigma_{\slim})$
then we are ensured that
$\eta_{p_i}$ is an $\al$-adapted coordinate
of quality  $\varsigma_{\slim}$.  Hence (\ref{item-1or2})
holds.

(\ref{item-catchall0stratum}).  If $q$ is in the $0$-stratum then  
$q\in V_0$,
so $q\prec \bar q$ for some $\bar q\in V$.   By  Lemma \ref{ballcovers},
for some $i\in I_{\zeroball}$ we have
$B(\bar q,r^0_{\bar q})\cap B(p_i,r^0_{p_i})\neq\emptyset$ and
$r^0_{\bar q}\leq 2r^0_{p_i}$.  Therefore $q\in B(p_i,5r^0_{p_i})$ and
by (\ref{item-1or2}), we have $q\in B(p_i,\frac{1}{10}r^0_{p_i})$.

(\ref{item-atmostoneend}).  Let $\eps > 0$
be a new constant.  Suppose that $N_{p_i}$ has more than one
end.  Then $\ctits N_{p_i}\simeq \R$.  If
$\de_0<\overline{\de}_0(\eps)$ then every point 
$q\in B(p_i,1)\subset \frac{1}{r^0_{p_i}}M$ will have a strainer
of quality $\eps$ and scale $\eps^{-1}$.  By (\ref{item-radiusratio})
and Lemma \ref{strainerlemma}, 
if $\eps<\overline{\eps}(\be_1)$ and 
$\Upsilon_0>\overline{\Upsilon}_0(\be_1)$ then there is 
a $(1,\be_1)$ splitting of $(\frac{1}{\r_q}M,q)$.  Thus every
point in $B(p_i,r^0_{p_i})$ is in the $1$-stratum or $2$-stratum.
This contradicts the definition of $V$, and hence $N_{p_i}$ has
at most one end.
\end{proof}

\section{Mapping into Euclidean space}
\label{sec-mapping}

\subsection{The definition of the map $\e^0:M\ra H$} \label{subsecmapping}

We will now use the ball collections defined in Sections
\ref{sec-loc2stratum}-\ref{sec-loc0stratum}, and the 
geometrically defined functions
discussed in earlier sections, to construct a smooth map
$\e^0:M\ra H=\bigoplus_{i\in I}\;H_i$, 
where 
\begin{itemize}
\item $I=I_{\r}\cup I_{E'}\cup
I_{\zeroball}\cup I_{\slim}\cup I_{\edge}\cup I_{\twostratum}$,
where the two index sets $I_{\r}$ and $I_{E'}$ are singletons
$I_{\r}=\{\r\}$ and $I_{E'}=\{E'\}$ respectively,
\item $H_i$ is a copy of $\R$
when 
$i=\r$,

\item $H_i$ is a copy of $\R \oplus \R$ when 
$i\in I_{E'}
\cup
I_{\zeroball}\cup I_{\slim}\cup I_{\edge}$, 
and
\item $H_i$ is a copy of $\R^2 \oplus \R$ when $i\in I_{\twostratum}$.

\end{itemize}

We also put 
\begin{itemize}
\item $H_{\zeroball}=\bigoplus_{i\in I_{\zeroball}}\,H_i$,
\item $H_{\slim}=\bigoplus_{i\in I_{\slim}}\,H_i$,
\item $H_{\edge}=\bigoplus_{i\in I_{\edge}}\,H_i$, 
\item $H_{\twostratum}=\bigoplus_{i\in I_{\twostratum}}\,H_i$,
\item $Q_1=H$,
\item $Q_2=H_{\zeroball}\bigoplus H_{\slim}
\bigoplus H_{\edge}$,
\item $Q_3=H_{\zeroball}\bigoplus H_{\slim}$,
\item $Q_4=H_{\zeroball}$, and
\item $\pi_{i,j}:Q_i\ra Q_j$, $\pi_i=\pi_{1,i}:H\ra Q_i$,
$\pi_i^\perp:H\ra Q_i^\perp$
are the orthogonal projections, for $1\leq i\leq j\leq 4$.
\end{itemize}
If
$x\in Q_j$,
we denote the projection to a summand $H_i$ by $\pi_{H_i}(x)=x_i$.
When $i\neq \r$, we write $H_i = H_i' \oplus H_i'' \cong \R^{k_i} \oplus \R$,
where $k_i \in \{1,2\}$,
and we denote the 
decomposition of $x_i \in H_i$ into its components by $x_i=(x_i',x_i'')\in 
H_i' \oplus H_i''$. We denote orthogonal projection onto
$H_i'$ and $H_i''$ by $\pi_{H_i'}$ and $\pi_{H_i''}$, respectively.

In Sections 
\ref{sec-loc2stratum}-\ref{sec-loc0stratum}, 
we defined adapted coordinates $\eta_p$, and cutoff functions $\zeta_p$
corresponding to points $p\in M$ of different types.
If $\{p_i\}$ is a collection of points used to define a
ball cover, as in Sections
\ref{sec-loc2stratum}-\ref{sec-loc0stratum}, then we write
$\eta_i$ for $\eta_{p_i}$ and $\zeta_i$ for $\zeta_{p_i}$.
  Recall that we also defined $\eta_{E'}$
and $\zeta_{E'}$ in Sections \ref{subsec-regularizationofrhoe'}
and \ref{subsec-additionalcutoff}, respectively.
For $i\in I\setminus \{\r\}$,
we will also define a new
scale parameter $R_i$, as follows:
\begin{itemize}
\item If $i\in I_{\zeroball}$
we put
$R_i = r^0_{p_i}$,  where $r^0_{p_i}$ is as in Lemma  \ref{0-ballseparation};
\item If $i\in I_{\slim}\cup I_{\edge}\cup I_{\twostratum}$,
then $R_i=\r_{p_i}$;
\item If $i=E'$, then $R_i=\r$; note
that unlike in the other cases, $R_i$ is not a constant.
\end{itemize}

The  component $\e^0_i:M\ra H_i$ of the map $\e^0:M\ra H$ 
is defined to be $\r$ when $i=\r$,
and
\begin{equation}
( R_i \eta_i \zeta_i, R_i \zeta_i)
\end{equation}
otherwise.

In the remainder of this section we prepare for the adjustment procedure in
Section \ref{sec-adjusting} by examining the behavior of $\e^0$ near 
the different strata.

\subsection{The image of $\e^0$} \label{subsecimagee0}

Before proceeding, we make some observations about the
image of $\e^0$, to facilitate the choice of cutoff functions.   
Let $x=\e^0(p)\in H$.  Then the components  of 
$x$
satisfy the following inequalities:
\begin{equation}
x_{\r}>0
\end{equation}
and for 
every $i\in I_{\zeroball}\cup I_{\slim}\cup I_{\edge}\cup I_{\twostratum} $,
\begin{equation}
x_i''\in [0,R_i]\quad\mbox{and}\quad |x_i'|\leq c_i\,x_i''\,,
\end{equation}
where 
\begin{equation}
c_i = 
\begin{cases}
9 \Delta &  \text{ when
$i \in I_{E^\prime}$,} \\
\frac{9}{10} & \text{ when
$i \in I_{\zeroball}$,} \\
10^5 \Delta & \text{ when
$i \in I_{\slim}$,} \\
9 \Delta & \text{ when
$i \in I_{\edge}$,} \\
9 & \text{ when $i\in I_{\twostratum}$.}
\end{cases}
\end{equation}

\begin{lemma} \label{uniform}
Under the constraint $\La \le \overline{\La}({\mathcal M})$,
there is a number $\Omega_0 = \Omega_0({\mathcal M})$ so that
for all $p \in M$, $|D\e^0_p| \le \Omega_0$.
\end{lemma}
\begin{proof}
This follows from the definition of $\e^0$.
\end{proof}

\subsection{Structure of $\e^0$ near the $2$-stratum}
\label{subsec-e0near2stratum}
Put
\begin{equation}
\tilde A_1=\bigcup_{i\in I_{\twostratum}}\;\{|\eta_i|\leq 8\}
\,,\quad
A_1=\bigcup_{i\in I_{\twostratum}}\;\{|\eta_i|\leq 7\}\,.
\end{equation}
We refer to Definition \ref{def-cloudykmanifold}
for the definition of a cloudy manifold.
We will see that on a scale which is sufficiently small compared
with $\r$, the pair $(\tilde S_1,S_1)
=(\e^0(\tilde A_1),\e^0(A_1))\subset H$
is a cloudy $2$-manifold.
In brief, this is because, on a scale small compared with
$\r$,  near any point in $A_1$
the map $\e^0$ is well approximated  in the
$C^1$ topology by an affine function of
$\eta_i$, for some $i\in I_{\twostratum}$.

Let $\Sigma_1, \Gamma_1 > 0$ be new parameters.
Define $r_1: \tilde S_1\ra (0,\infty)$  by putting 
$r_1(x)=\Si_1\; \r_{p}$ for some $p \in (\e^0)^{-1}( x)\cap
\tilde A_1$.

\begin{lemma}
\label{lem-2stratumestimates}
There is a constant $\Om_1=\Om_1(\mathcal{M})$ so that
under the constraints $\Si_1< \overline{\Si}_1(\Ga_1,\mathcal{M})$,
$\be_2<\overline{\be}_2(\Ga_1,\Si_1,\mathcal{M})$,  
$\varsigma_{\twostratum}<\overline{\varsigma}_{\twostratum}(\Ga_1,\Si_1,\mathcal{M})$,
$\be_E<\overline{\be}_E(\Ga_1,\Si_1,\Delta,\mathcal{M})$,
$\si_E<\overline{\si}_E(\Ga_1,\Si_1,\Delta,\mathcal{M})$,
$\varsigma_{\edge}<\overline{\varsigma}_{\edge}(\Ga_1,\Si_1,
\Delta,\mathcal{M})$,
$\varsigma_{E'}<\overline{\varsigma}_{E}(\Ga_1,\Si_1,\Delta,\mathcal{M})$,
$\be_1<\overline{\be}_1(\Ga_1,\Si_1,\Delta,\mathcal{M})$, 
$\varsigma_{\slim}<\overline{\varsigma}_{\slim}(\Ga_1,\Si_1,
\Delta,\mathcal{M})$,
$\varsigma_{\zeroball}<\overline{\varsigma}_{\zeroball}(\Ga_1,\Si_1,
\Delta,\mathcal{M})$, $\Upsilon_0 \ge 
\overline{\Upsilon}_0(\Ga_1, \Si_1, \Delta, \mathcal {M})$ and
$\La<\overline{\La}(\Ga_1,\Si_1,\Delta,\mathcal{M})$,
the following holds.

\begin{enumerate}
\item The triple $(\tilde S_1,S_1,r_1)$ is a $(2,\Ga_1)$ cloudy $2$-manifold.
\item  The affine subspaces $\{A_x\}_{x \in S_1}$ inherent in the definition
of the cloudy $2$-manifold can be chosen to have the following property.
Pick $p \in A_1$ and put $x=\e^0(p)\in S_1$.  
Let
$A^0_x\subset H$ be the linear subspace parallel to $A_x$   (i.e.
$A_x=A^0_x+x$) and let $\pi_{A^0_x}:H\ra A^0_x$ denote orthogonal
projection onto $A^0_x$.
Then
\begin{equation} \label{inherent1}
\|D\e^0_p-\pi_{A^0_x}\circ D\e^0_p\|<\Ga_1\,,
\end{equation}
and 
\begin{equation} \label{inherent2}
\Om_1^{-1}\|v\|\leq \|\pi_{A^0_x}\circ D\e^0(v)\|\leq 
\Om_1\|v\|\,
\end{equation}
for every $v\in T_{p}M$ which is orthogonal to 
$\ker(\pi_{A^0_x}\circ D\e^0_p)$.
\item 
Given $i \in I_{\twostratum}$, there is a smooth map
$\widehat{{\mathcal E}}^0_i \: : \: 
(\overline{B(0, 8)} \subset \R^2) \rightarrow
(H_i')^\perp$ such that 
\begin{equation} \label{slope1}
\| D\widehat{{\mathcal E}}^0_i \| \le \Omega_1 R_i
\end{equation}
and on the subset 
$\{|\eta_i| \le 8\}\subset \frac{1}{R_i}M$,
we have
\begin{equation}
\label{eqn-f0jagain}
\left\| \frac{1}{R_i} \e^0- \left( \eta_i,
\frac{1}{R_i} \widehat{{\mathcal E}}^0_i \circ \eta_i \right) 
\right\|_{C^1}<\Gamma_1.
\end{equation}

Furthermore, if $x \in S_1$ then there are some 
$i \in I_{\twostratum}$ and $p \in \{|\eta_i| \le 7\}$ such that
$x = \e^0(p)$ and
$A^0_x = \Image \left( I, \frac{1}{R_i} 
(D\widehat{{\mathcal E}}^0_i)_{\eta_i(p)} \right)$.
\end{enumerate} 
\end{lemma}

The parameters $\eps_1,\,\eps_2 > 0$ will be
internal to this subsection, which is devoted to the proof of  Lemma
\ref{lem-2stratumestimates}.
Until further notice, the index $i$ will denote a fixed element
of  $ I_{\twostratum}$.

Put
$J=\{j\in I_{E'} \cup 
I_{\zeroball}\cup I_{\slim}\cup I_{\edge}\cup I_{\twostratum} 
\mid \supp\zeta_j\cap  B(p_i,10R_i)\neq\emptyset\}$.

\begin{sublemma} \label{2adjusting}
Under the constraints $\beta_2 <\overline{\beta}_2(\eps_1)$, $\varsigma_{\twostratum}
<\overline{\varsigma}_{\twostratum}(\eps_1)$,
$\beta_{E}< \overline{\beta}_{E}(\eps_1,\Delta)$,
$\varsigma_{\edge}<\overline{\varsigma}_{\edge}(\eps_1,\Delta)$,
$\varsigma_{E'}<\overline{\varsigma}_{E}(\eps_1,\Delta)$,
$\beta_1<\overline{\beta}_1(\eps_1,\Delta)$,  
$\varsigma_{\slim}<\overline{\varsigma}_{\slim}(\eps_1,\Delta)$,
$\varsigma_{\zeroball}<\overline{\varsigma}_{\zeroball}(\eps_1,\Delta)$
and 
$\Lambda<\overline{\Lambda}(\eps_1,\Delta)$,
the following holds.

For each $j\in J$,
there is a map $T_{ij} : \R^2 \rightarrow \R^{k_j}$ which
is a composition of an isometry and an orthogonal projection, such
that on the ball $B(p_i,10)\subset \frac{1}{R_i}M$, 
the map $\eta_j$ is 
defined and satisfies
\begin{equation}
\label{eqn-etajtij}
\left\|\frac{R_j}{R_i} \eta_j-
\left(T_{ij}\circ\eta_i\right)\right\|_{C^1}<
\eps_1\,.
\end{equation}
\end{sublemma}
\begin{proof}
As we are assuming the hypotheses of Lemma \ref{lem3}, there are no
$3$-stratum points.  

Suppose first that $j \in I_{\twostratum}$. Then $d(p_j, p_i) \le
10(R_i + R_j)$. If ${\Lambda}$ is sufficiently small then
we can assume that $\frac{R_j}{R_i}$ is arbitrarily close to $1$,
so in particular $d(p_j, p_i) \le 40 R_j$.
By Lemma \ref{faralignment1}, if ${\beta}_2$ is sufficiently
small then the $(2,\beta_2)$-splitting of $\left( \frac{1}{R_j} M, p_j \right)$
gives an arbitrarily good $2$-splitting of $\left( \frac{1}{R_j} M, p_i \right)$.
By Lemma \ref{alexcompatible}, if ${\beta}_2$ is sufficiently small
then this splitting of $\left( \frac{1}{R_j} M, p_i \right)$ is compatible,
to an arbitrarily degree of closeness, with the $(2,\beta_2)$-splitting of 
$\left( \frac{1}{R_i} M, p_i \right)$ coming from the fact that
$p_i$ is a $2$-stratum point.  
Hence in this case, if $\beta_2$ and ${\varsigma}_{\twostratum}$
are sufficiently small (as functions of $\eps_1$) 
then the sublemma follows from
Lemma \ref{adaptedclose}, along with Remark \ref{addedremark}. 

If $j \in I_{\edge} \cup I_{\slim}$ then 
$d(p_i, p_j)\le 10R_i + 10^5 \Delta R_j$.  We now have an approximate $1$-splitting
at $p_j$, which gives an approximate $1$-splitting at $p_i$. As before,
if ${\beta}_2$, ${\varsigma}_{\twostratum}$,
${\beta}_1$,   
${\varsigma}_{\edge}$,
${\varsigma}_{\slim}$,
and 
${\Lambda}$ are sufficiently small (as functions of $\eps_1$ and $\Delta$)
then we can apply
Lemmas \ref{alexcompatible} and \ref{adaptedclose} to deduce the conclusion of
the sublemma. Note that in this case, we have to allow ${\La}$ to
depend on $\Delta$. 

If $j \in I_{E^\prime}$ then since
$\supp\zeta_{E'}\cap  B(p_i,10R_i)\neq\emptyset$,
we know that $\eta_{E^\prime}(q) \in [\frac{2}{10}\Delta, 9 \Delta]$ for some
$q \in B(p_i,10R_i)$.
As $\Delta >> 10$, it follows from Lemma \ref{lem-rhoerhoe'} 
that if $\La$ is sufficiently small then
$B(p_i, 10R_i) \subset
\eta_{E^\prime}^{-1} \left( \frac{1}{10} \Delta, 10 \Delta
\right)$.  
From the definition of $E'$, it follows that if
${\beta}_{E'}$ and ${\La}$ are sufficiently small then 
there is a $1$-splitting at $p_i$ of arbitrarily good
quality, coming from the $[0,C]$-factor in Definition 
\ref{edgepoints}. 
As before,
if ${\beta}_{E'}$, ${\La}$, ${\beta}_2$, ${\varsigma}_{\twostratum}$,
${\beta}_{E'}$ and   
${\varsigma}_{E}$ are sufficiently small (as functions of
$\eps_1$ and $\Delta$) then
we can apply
Lemmas \ref{alexcompatible} and \ref{adaptedclose} to deduce the conclusion of
the sublemma.

If $j \in I_{\zeroball}$ then
since $\supp\zeta_{j} \cap  B(p_i,10R_i)\neq\emptyset$,
we know that $\eta_j(q) \in [\frac{2}{10}, \frac{9}{10}]$ for some
$q \in B(p_i, 10 R_i)$. From Lemma \ref{0-ballseparation}, 
$\frac{r^0_{p_j}}{R_i} \ge
\frac{1}{20} \Upsilon_0$. Hence we may assume that
$B(p_i, 10 R_i) \subset A(p_j, \frac{1}{10} r^0_{p_j}, r^0_{p_j})$.
Lemma \ref{0-ballseparation} also gives a 
$(1,\beta_1)$-splitting 
of $\left( \frac{R_i}{R_j} M, p_i \right)$. If ${\beta}_1$ and
${\beta}_2$ are sufficiently small then
by Lemma \ref{alexcompatible}, this 1-splitting is
compatible with the $(2,\beta_2)$-splitting of 
$\left( \frac{1}{R_i} M, p_i \right)$ to an arbitrary degree of closeness. 
As before, if $\beta_1$, $\beta_2$, ${\varsigma}_{\twostratum}$
and
$\overline{\varsigma}_{\zeroball}$ are sufficiently small
(as functions of $\eps_1$ and $\Delta$) then the
sublemma follows from Lemma \ref{adaptedclose}.
\end{proof}

We retain
the hypotheses of Sublemma \ref{2adjusting}.

For $j\in J$, the cutoff function $\zeta_j$
is a function of the $\eta_{j'}$'s for $j'\in J$, i.e.  there is a 
smooth function $\Phi_j \in C^\infty_c(\R^J)$
such that 
$
\zeta_j(\cdot) = 
\Phi_j \left( \{\eta_{j^\prime}(\cdot)\}_{j^\prime \in J} \right)$.
(Note from (\ref{additional2}) 
that $\zeta_{E^\prime}$ depends on $\eta_{E^\prime}$ and
$\{\zeta_k\}_{k \in I_{\edge}}$.)
The $H_j$-component of ${\mathcal E}^0$, after dividing by
$R_i$, can be written as 
\begin{equation}
\label{eqn-recalle0j}
\frac{1}{R_i}\e^0_j =
\left(\frac{R_j}{R_i}\eta_j\zeta_j,\frac{R_j}{R_i}\zeta_j\right)
=\left(\frac{R_j}{R_i}\eta_j \cdot 
(\Phi_j\circ \{\eta_{j^\prime}\}_{j^\prime \in J}),
\frac{R_j}{R_i}\Phi_j\circ \{\eta_{j^\prime}\}_{j^\prime \in J}\right)\,.
\end{equation}

Let $\f^0:\R^2\ra H$ be the map so that the $H_j$-component
of $\f^0 \circ \eta_i$,
for $j\in J$, is obtained from the preceding formula by replacing
each occurrence of $\eta_j$ with the approximation 
$\frac{R_i}{R_j}(T_{ij}\circ\eta_i)$, i.e.
\begin{equation} \label{feqn}
\frac{1}{R_i }\f^0_j(u)=\left( 
T_{ij}(u) \cdot \left( \Phi_j \left( \left\{ 
\frac{R_i}{R_{j^\prime}} T_{ij^\prime}
\left( u \right) \right\}_{j^\prime \in J} \right) \right),
\frac{R_j}{R_i} \Phi_j \left( \left\{ 
\frac{R_i}{R_{j^\prime}} T_{ij^\prime}
\left( u \right) \right\}_{j^\prime \in J} \right)
\right)\,,
\end{equation}
whose $H_{\r}$-component is the constant function $R_i$,
and whose other components vanish.
That is,
\begin{equation}
\frac{1}{R_i }\f^0_j \circ \eta_i=\left( 
(T_{ij} \circ \eta_i) \cdot \left( \Phi_j \left( \left\{ 
\frac{R_i}{R_{j^\prime}} T_{ij^\prime} \circ \eta_i
 \right\}_{j^\prime \in J} \right) \right),
\frac{R_j}{R_i} \Phi_j \left( \left\{ 
\frac{R_i}{R_{j^\prime}} T_{ij^\prime}
\circ \eta_i \right\}_{j^\prime \in J} \right)
\right).
\end{equation}

\begin{sublemma}  \label{eps2}
Under the constraints $\eps_1 \leq 
\overline{\eps}_1(\eps_2,{\mathcal M})$,
$\Upsilon_0 \ge \overline{\Upsilon}_0(\eps_2,{\mathcal M})$ 
and 
$\La \leq \overline{\La}(\eps_2,{\mathcal M})$,
\begin{equation}
\label{eqn-f0j}
\left\| \frac{1}{R_i} \e^0-
\frac{1}{R_i} \f^0\circ \eta_i \right\|_{C^1}<\eps_2
\end{equation}
on $B(p_i, 10) \subset \frac{1}{R_i} M$.
\end{sublemma}
\begin{proof}
First note that $\e_{\r}(p_i) = \f^0_{\r}(p_i) = R_i$ and
the $\e_{\r}$-component of $\e$ has Lipschitz constant $\La$,
so it suffices to control the remaining components.
For $j \in J$ and 
$j \in I_{E'} \cup I_{\slim}\cup I_{\edge}\cup I_{\twostratum}$, 
if ${\La}$ is sufficiently
small then we can assume that $\frac{R_i}{R_j}$ is arbitrarily close to
one.  Then the $H_j$-component of $\frac{1}{R_i} \e^0-
\frac{1}{R_i} \f^0\circ \eta_i$ can be estimated in $C^1$-norm
by using (\ref{eqn-etajtij}) to estimate $\eta_{j'}$, plugging this into
(\ref{eqn-recalle0j}) and applying the chain rule.
In applying the chain rule, we use the fact that the  functions
$\Phi_j$ have explicit bounds on their derivatives of order 
up to $2$.

If $j \in J \cap I_{\zeroball}$ then the only relevant argument of 
$\Phi_j$ is when $j' = j$.  
Hence in this case we can write
\begin{equation}
\frac{1}{R_i}\e^0_j =
\left(\frac{R_j}{R_i}\eta_j \cdot 
\Phi_j(\eta_{j}),
\frac{R_j}{R_i}\Phi_j(\eta_{j}) \right)
\end{equation}
and
\begin{equation}
\frac{1}{R_i }\f^0_j \circ \eta_i=\left( 
(T_{ij} \circ \eta_i) \cdot \left( \Phi_j \left( 
\frac{R_i}{R_{j}} T_{ij} \circ \eta_i
 \right) \right),
\frac{R_j}{R_i} \Phi_j \left( 
\frac{R_i}{R_{j}} T_{ij}
\circ \eta_i \right)
\right).
\end{equation}
From part (2) of Lemma \ref{0-ballseparation},
$\frac{R_i}{R_{j}} \le \frac{20}{\Upsilon_0}$.
Then the $H_j$-component of $\frac{1}{R_i} \e^0-
\frac{1}{R_i} \f^0\circ \eta_i$ can be estimated in $C^1$-norm
by using (\ref{eqn-etajtij}). Note when we use the chain rule to estimate
the second component of
$\frac{1}{R_i} \e^0-
\frac{1}{R_i} \f^0\circ \eta_i$, namely
$\frac{R_j}{R_i} \left( \Phi_j(\eta_{j}) -
\Phi_j \left( 
\frac{R_i}{R_{j}} T_{ij}
\circ \eta_i \right) \right)$, we differentiate
$\Phi_j$ and this brings down a factor of $\frac{R_i}{R_{j}}$
when estimating norms.
\end{proof}

\begin{sublemma} \label{etaclose}
Given $\Si \in (0,\frac{1}{10})$, 
suppose that
$|\eta_i(p)| < 8$ for some $p \in M$.
Put $x = {\mathcal E}^0(p)$.
For any $q \in M$, if
${\mathcal E}^0(q) \in B(x, \Si R_i)$ then
$|\eta_i(p) - \eta_i(q)| < 20\Si$.
\end{sublemma}
\begin{proof}
We know that $\zeta_i(p) = 1$. 
By hypothesis, $|{\mathcal E}^0(p) - {\mathcal E}^0(q)| <
\Si R_i$. In particular,
$|\zeta_i(p) - \zeta_i(q)| < \Si$ and
$|\zeta_i(p) \eta_i(p)- \zeta_i(q) \eta_i(q)| 
< \Si$.
Then 
\begin{align}
|\eta_i(p) - \eta_i(q)|
& 
= \frac{1}{\zeta_i(q)} |\zeta_i(q) \eta_i(p) - \zeta_i(q)
 \eta_i(q)| \\
& \le 
\frac{1}{\zeta_i(q)} \left[ |\zeta_i(p) \eta_i(p) - \zeta_i(q)
 \eta_i(q)| + 
|\zeta_i(p) - \zeta_i(q)| |\eta_i(p)| \right] \notag \\
& \le 
\frac{10\Si}{1 - \Si} \le 20\Si. \notag
\end{align}
This proves the sublemma.
\end{proof}

We now prove Lemma \ref{lem-2stratumestimates}.
We no longer fix $i \in I_{\twostratum}$.
Given $x \in {S}_1$, choose ${p} \in A_1$ and $i \in I_{\twostratum}$
so that $\e^0(p) = x$ and $|\eta_i(p)| \le 8$.
Put $A^0_x = \Image(d{\mathcal F}^0_{\eta_i(p)})$, a $2$-plane in $H$,
and let $A_x = x + A^0_x$ be the corresponding affine subspace through $x$. 
We first show that
under the constraints
$\Sigma_1 \le \overline{\Sigma}_1(\Gamma_1,{\mathcal M})$,
$\eps_2 \le \overline{\eps}_2(\Gamma_1,{\mathcal M})$ and 
$\La \le \overline{\Lambda}(\Gamma_1,{\mathcal M})$,
the triple $(\tilde S_1,S_1,r_1)$ is a $(2,\Ga_1)$ cloudy $2$-manifold.

We verify condition (1) of Definition \ref{approxlemma}.
Pick $x,y\in \widetilde{S}_1$, and choose 
$p \in (\e^0)^{-1}(x)\cap 
\bigcup_{i \in I_{\twostratum}} |\eta_i|^{-1}[0, 8)$
(respectively
$q\in (\e^0)^{-1}(y)
\cap \bigcup_{i \in I_{\twostratum}} |\eta_i|^{-1}[0, 8)$) 
satisfying $r_1(x)=\Si_1\r_{p}$
(respectively $r_1(y)=\Si_1\r_{q}$).

We can assume that $\La < \frac{1}{100}$.
Suppose first that $d(p,q)\leq \frac{\r_{p}}{\Lambda}$.
Then
since $\r$ is $\Lambda$-Lipschitz, we get
$|\r_{p}-\r_{q}|\leq \r_{p}$, so in this
case 
\begin{equation}
|r_1(x)-r_1(y)|=\Si_1|\r_{p}-\r_{q}|
\leq \Si_1\r_{p}=r_1(x)\,.
\end{equation}

Now suppose that $d(p,q)\geq 20\r_{p}$.  We claim
that if $\La$ 
is sufficiently small then this implies that $d(p, q)
\geq 19\r_{q}$ as well.
Suppose not. Then 
$20\r_{p} \le d(p, q)
\leq 19\r_{q}$, so
$\frac{\r_{p}}{\r_{q}} \le \frac{19}{20}$.
On the other hand, since $|{\r_{q}} - {\r_{p}} | \le
\Lambda d(p, q)$, we also know that
${\r_{q}} - {\r_{p}} \le
\Lambda d(p, q) \le 19 \Lambda \r_{q}$, so
$\frac{\r_{p}}{\r_{q}} \ge 1 - 19 \Lambda$.
If $\La$ is sufficiently small then this is a contradiction.

Thus there are $i,j\in I_{\twostratum}$ such that
$p\in |\eta_i|^{-1}[0,8)$, $q \in |\eta_j|^{-1}[0,8)$, 
$\zeta_i(p)=1=\zeta_j(q)$ and
$\zeta_i(q)=0=\zeta_j(p)$.  Then 
\begin{align}
|x-y|&=|\e^0(p)-\e^0(q)|\geq 
\max(\;\r_{p_i}|\zeta_i(p)-\zeta_i(q)|\,,\,
\r_{p_j}|\zeta_j(p)-\zeta_j(q)|\;) \\
&=\max(\r_{p_i}\,,\,\r_{p_j})\geq \frac12\max(\r_{p}\,,\,\r_{q})
=\frac{\max(r_1(x)\,,\,r_1(y))}{2\,\Si_1}\,. \notag
\end{align}
So $|r_1(x)-r_1(y)|\leq |x-y|$ provided $\Si_1\leq \frac14$.
Thus  condition (1) of Definition \ref{approxlemma} will be satisfied.

We now verify condition (2) of Definition \ref{approxlemma}.
Given $x \in S_1$, let $i \in I_{\twostratum}$ and
$p \in M$ be such that $\e^0(p)=x$ and $|\eta_i(p)| \le 7$.
Taking $\Sigma = \frac{1}{100}$ in Sublemma \ref{etaclose},
we have 
$\Image(\e^0) \cap B \left( x, \frac{R_i}{100} \right) \subset
\Image \left( \e^0 \Big|_{|\eta_i|^{-1}[0,7.2)} \right)$.
Thus we can restrict attention to the action of $\e^0$ on
$|\eta_i|^{-1}[0,7.2)$. Now $\Image \left( \f^0 \Big|_{B(0, 7.2)}
\right)$ is the restriction to $B(0, 7.2)$ of the graph of a function 
$G^0_i : H_i' \rightarrow (H_i')^\perp$, 
since $T_{ii} = \Id$ and
$\zeta_i \Big|_{B(0, 7.2)} = 1$. Furthermore, in view of the
universality of the functions $\{\Phi_j\}_{j \in J}$ and the
bound on the cardinality of $J$, there are uniform $C^1$-estimates on $G^0_i$. 
Hence we can find ${\Sigma}_1$ (as a function of $\Gamma_1$ and
${\mathcal M}$) to ensure that
$\left( \frac{1}{r_1(x)} \Image \left( \f^0 \Big|_{B(0, 7.2)}
\right), x \right)$ is $\frac{\Gamma_1}{2}$-close in the 
pointed Hausdorff topology to $x + \Image(d\f^0_p)$.
Finally, if the parameter $\eps_2$ of Sublemma \ref{eps2} is sufficiently
small then we can ensure that 
$\left( \frac{1}{r_1(x)} \Image \left( \e^0 
\right), x \right)$ is ${\Gamma_1}$-close in the 
pointed Hausdorff topology to $x + \Image(D\f^0_p)$. Thus
condition (2) of Definition \ref{approxlemma} will be satisfied.

To finish the proof of Lemma \ref{lem-2stratumestimates},
equation (\ref{inherent1}) is clearly satisfied if
the parameter $\epsilon_2$ of Sublemma \ref{eps2}
is sufficiently small. 
Equation (\ref{inherent2})
is equivalent to upper and lower bounds on
the eigenvalues of the matrix 
$(\pi_{A^0_x}\circ D\e^0_p) (\pi_{A^0_x}\circ D\e^0_p)^*$,
which acts on the two-dimensional space $A^0_x$.
In view of Sublemma \ref{eps2} and the definition of $A_x$, 
it is sufficient to 
show upper and lower bounds on the eigenvalues of
$D\f^0_{\eta_i(p)} (D\f^0_{\eta_i(p)})^*$ acting on $A^0_x$.
In terms of the function $G^0_i$, these are the same as the
eigenvalues of $I_2 + ((DG^0_i)_{\eta_i(p)})^* (DG^0_i)_{\eta_i(p)}$,
acting on $\R^2$.
The eigenvalues are clearly bounded below by one.
In view of the $C^1$-bounds on $G^0_i$, there is an upper
bound on the eigenvalues in terms of $\dim(H)$, which in
turn is bounded above in terms of ${\mathcal M}$.
This shows equation (\ref{inherent2}).

Finally, given $i \in I_{\twostratum}$, we can write
$\frac{1}{R_i} {\mathcal F}^0$ on $\overline{B(0,8)} \subset \R^2$ in the form
$\frac{1}{R_i} {\mathcal F}^0 =
\left( I, \frac{1}{R_i} \widehat{\e}_i^0 \right)$ with respect to the
orthogonal decomposition $H = H_i' \oplus (H_i')^\perp$. (Recall that
${\mathcal F}^0$ is defined in reference to the given value of $i$.)
We use this to define $\widehat{\e}_i^0$. 
Equation (\ref{eqn-f0jagain}) is a consequence of
Sublemma \ref{eps2}. The last statement of Lemma 
\ref{lem-2stratumestimates} follows from
the definition of $A^0_x$.

\subsection{Structure of $\e^0$ near the edge stratum}
\label{subsec-e0nearedgestratum}

Recall that $Q_2=H_{\zeroball} \oplus
H_{\slim} \oplus H_{\edge}$, and 
$\pi_2:H\ra Q_2$ is the orthogonal projection.

Put
\begin{equation}
\tilde A_2=\bigcup_{i\in I_{\edge}}\;\{|\eta_i|\leq 8\De,\,
\eta_{E'}\leq 8\De\}, \quad
A_2=\bigcup_{i\in I_{\edge}}\;\{|\eta_i|\leq 7\De,\,\eta_{E'}\leq 7\De\}
\end{equation}
and
\begin{equation}
\tilde S_2=(\pi_2\circ\e^0)(\tilde A_2), \quad S_2=(\pi_2\circ\e^0)(A_2).
\end{equation}

Let $\Si_2, \Gamma_2 > 0$ be new parameters.
Define $r_2: \tilde S_2\ra (0,\infty)$  by 
putting 
$r_2(x)=\Si_2\; \r_{p}$ for some 
$p \in (\pi_2\circ\e^0)^{-1}( x)\cap
\tilde A_2$.

The analog of Lemma \ref{lem-2stratumestimates}  for the region 
near edge points is:

\begin{lemma}
\label{lem-edgestratumestimates}
There is a constant $\Om_2=\Om_2(\mathcal{M})$
so that under the constraints
$\Si_2<\overline{\Si}_2(\Ga_2,\mathcal{M})$, 
$\be_E<\overline{\be}_E(\Ga_2,\Si_2,\be_2,\Delta,\mathcal{M})$,
$\si_E<\overline{\si}_E(\Ga_2,\Si_2,\be_2,\Delta,\mathcal{M})$,
$\varsigma_{\edge}<
\overline{\varsigma}_{\edge}(\Ga_2,\Si_2,\be_2,\Delta,\mathcal{M})$,
$\be_1<\overline{\be}_1(\Ga_2,\Si_2,\be_2,\Delta,\mathcal{M})$,
$\varsigma_{\slim}<
\overline{\varsigma}_{\slim}(\Ga_2,\Si_2,\be_2,\Delta,\mathcal{M})$
$\varsigma_{\zeroball}<
\overline{\varsigma}_{\zeroball}(\Ga_2,\Si_2,\be_2,\Delta,\mathcal{M})$ and
$\La<\overline{\La}(\Ga_2,\Si_2,\be_2,\Delta,\mathcal{M})$, 
the following holds.

\begin{enumerate}
\item The triple $(\tilde S_2,S_2,r_2)$ is a 
$(2,\Ga_2)$ cloudy $1$-manifold.
\item 
The affine subspaces $\{A_x\}_{x \in S_2}$ inherent in the definition of
the cloudy $1$-manifold can be chosen to have the following property.
Pick $p \in A_2$ and put $x=(\pi_2\circ\e^0)(p)\in S_2$.
Let $A^0_x\subset Q_2$ be the linear subspace parallel to $A_x$
(i.e. $A_x = A^0_x + x$) and let
$\pi_{A^0_x} : H \rightarrow A^0_x$ denote orthogonal projection onto
$A^0_x$. Then
\begin{equation}
\|D(\pi_2\circ\e^0)_p
-\pi_{A^0_x}\circ D(\pi_2\circ\e^0)_p\|<\Ga_2
\end{equation}
and 
\begin{equation}
\Om_2^{-1}\|v\|\leq \|(\pi_{A^0_x}\circ D(\pi_2\circ\e^0))(v)\|
\leq \Om_2\|v\|
\end{equation}
for every $v\in T_pM$ which is orthogonal to 
$\ker(\pi_{A^0_x}\circ D(\pi_2\circ\e^0)_p)$.
\item 
Given $i \in I_{\edge}$, there is a smooth map
$\widehat{{\mathcal E}}^0_i \: : \: 
(\overline{B(0, 8\Delta)} \subset \R) \rightarrow
(H_i')^\perp \cap Q_2$ such that 
\begin{equation} \label{slope2}
\| D\widehat{{\mathcal E}}^0_i \| \le \Omega_2 R_i
\end{equation}
and on the subset 
$\{|\eta_i| \le 8 \Delta, \eta_{E'} \le 8 \Delta \}$,
we have
\begin{equation}
\label{eqn-f1jagain}
\left\| \frac{1}{R_i} \pi_2 \circ \e^0- \left( \eta_i,
\frac{1}{R_i} \widehat{{\mathcal E}}^0_i \circ \eta_i \right) 
\right\|_{C^1}<\Gamma_2.
\end{equation}

Furthermore, if $x \in S_2$ then there are some $i \in I_{\edge}$ and
$p \in \{ |\eta_i| \le 7 \Delta, \: \eta_{E'} \le 7 \Delta \}$ such that
$x = (\pi_2 \circ \e^0)(p)$ and
$A^0_x = \Image \left( I, \frac{1}{R_i} 
(D\widehat{{\mathcal E}}^0_i)_{\eta_i(p)} \right)$.
\end{enumerate}
\end{lemma}

We omit the proof as it is similar to the proof of Lemma
\ref{lem-2stratumestimates}.

\subsection{Structure of $\e^0$ near the slim $1$-stratum}
\label{subsec-e0nearslimstratum}

Recall that $Q_3=H_{\zeroball} \oplus
H_{\slim}$, and 
$\pi_3:H\ra Q_3$ is the orthogonal projection.

Put
\begin{equation}
\tilde A_3=\bigcup_{i\in I_{\edge}}\;\{|\eta_i|\leq 8\cdot 10^5 \De\}, \quad
A_3=\bigcup_{i\in I_{\edge}}\;\{|\eta_i|\leq 7 \cdot 10^5\De\}
\end{equation}
and
\begin{equation}
\tilde S_3=(\pi_3\circ\e^0)(\tilde A_3), \quad S_3=(\pi_3\circ\e^0)(A_3).
\end{equation}

Let $\Si_3, \Gamma_3 > 0$ be new parameters.
Define $r_3: \tilde S_3\ra (0,\infty)$  by 
putting 
$r_3(x)=\Si_3\; \r_{p}$ for some 
$p \in (\pi_3\circ\e^0)^{-1}( x)\cap
\tilde A_3$.

The analog of Lemma \ref{lem-2stratumestimates}  for the slim
$1$-stratum points is:

\begin{lemma}
\label{lem-slimstratumestimates}
There is a constant $\Om_3=\Om_3(\mathcal{M})$
so that under the constraints
$\Si_3<\overline{\Si}_3(\Ga_3,\mathcal{M})$, 
$\be_E<\overline{\be}_E(\Ga_3,\Si_3,\be_2,\Delta,\mathcal{M})$,
$\si_E<\overline{\si}_E(\Ga_3,\Si_3,\be_2,\Delta,\mathcal{M})$,
$\varsigma_{\edge}<
\overline{\varsigma}_{\edge}(\Ga_3,\Si_3,\be_2,\Delta,\mathcal{M})$,
$\be_1<\overline{\be}_1(\Ga_3,\Si_3,\be_2,\Delta,\mathcal{M})$,
$\varsigma_{\slim}<
\overline{\varsigma}_{\slim}(\Ga_3,\Si_3,\be_2,\Delta,\mathcal{M})$
$\varsigma_{\zeroball}<
\overline{\varsigma}_{\zeroball}(\Ga_3,\Si_3,\be_2,\Delta,\mathcal{M})$ and
$\La<\overline{\La}(\Ga_3,\Si_3,\be_2,\Delta,\mathcal{M})$, 
the following holds.

\begin{enumerate}
\item The triple $(\tilde S_3,S_3,r_3)$ is a 
$(2,\Ga_3)$ cloudy $1$-manifold.
\item 
The affine subspaces $\{A_x\}_{x \in S_3}$ inherent in the definition of
the cloudy $1$-manifold can be chosen to have the following property.
Pick $p \in A_3$ and put $x=(\pi_3\circ\e^0)(p)\in S_3$.
Let $A^0_x\subset Q_3$ be the linear subspace parallel to $A_x$
(i.e. $A_x = A^0_x + x$) and let
$\pi_{A^0_x} : H \rightarrow A^0_x$ denote orthogonal projection onto
$A^0_x$. Then
\begin{equation}
\|D(\pi_3\circ\e^0)_p
-\pi_{A^0_x}\circ D(\pi_3\circ\e^0)_p\|<\Ga_3
\end{equation}
and 
\begin{equation}
\Om_3^{-1}\|v\|\leq \|(\pi_{A^0_x}\circ D(\pi_3\circ\e^0)(v)\|
\leq \Om_3\|v\|
\end{equation}
for every $v\in T_pM$ which is orthogonal to 
$\ker(\pi_{A^0_x}\circ D(\pi_3\circ\e^0)_p)$.
\item
Given $i \in I_{\slim}$, there is a smooth map
$\widehat{{\mathcal E}}^0_i \: : \: 
(\overline{B(0, 8 \cdot 10^5 \Delta)} \subset \R) \rightarrow
(H_i')^\perp \cap Q_3$ such that 
\begin{equation} \label{slope3}
\| D\widehat{{\mathcal E}}^0_i \| \le \Omega_3 R_i
\end{equation}
and on the subset 
$\{|\eta_i|\leq 8\cdot 10^5 \De\}$,
we have
\begin{equation}
\label{eqn-f2jagain}
\left\| \frac{1}{R_i} \pi_3 \circ \e^0- \left( \eta_i,
\frac{1}{R_i} \widehat{{\mathcal E}}^0_i \circ \eta_i \right) 
\right\|_{C^1}<\Gamma_3.
\end{equation}

Furthermore, if $x \in S_3$ then there are some
$i \in I_{\slim}$ and $p \in \{ |\eta_i| \le 7 \cdot 10^5 \Delta \}$
such that $x = (\pi_3 \circ \e^0)(p)$ and
$A^0_x = \Image \left( I, \frac{1}{R_i} 
(D\widehat{{\mathcal E}}^0_i)_{\eta_i(p)} \right)$.
\end{enumerate}
\end{lemma}

We omit the proof as it is similar to the proof of Lemma
\ref{lem-2stratumestimates}.

\subsection{Structure of $\e^0$ near the $0$-stratum}
\label{subsec-e0nearzerostratum}

The only information we will need near the $0$-stratum is:

\begin{lemma}
\label{lem-zerosstratumestimates}
For $i \in I_{\zeroball}$,
the only nonzero component of the map 
$\pi_4 \circ \e^0:M\ra Q_4=H_{\zeroball}$
in the region 
$\{\eta_i\in [\frac{3}{10},\frac{8}{10}]\}$
is  $\e^0_i$, where it coincides with $(R_i\eta_i,R_i)$. 
\end{lemma}

\section{Adjusting the map to Euclidean space}
\label{sec-adjusting}

The main result of this section is the following
proposition, which asserts that it is possible to adjust $\e^0$
slightly, to get a new map $\e$ which is a submersion
in different parts of $M$.  In Section \ref{sec-extracting} 
this structure will yield compatible fibrations
of different parts of $M$.  

Let $c_{\adjust}>0$ be a  parameter.

\begin{proposition}
\label{prop-adjustment}
Under the constraints imposed in this and prior sections,
there is a smooth map $\e:M\ra H$ with the following properties:
\begin{enumerate}
\item
For every $p\in M$, 
\begin{equation}
\|\e(p)-\e^0(p)\|<c_{\adjust}\,\r(p)\quad \mbox{and}\quad
\|D\e_p-D\e^0_p\|<c_{\adjust}\,.
\end{equation}
\item For $j\in \{1,2,3\}$
the restriction of $\pi_j\circ\e:M\ra Q_j$ to
the  region $U_j\subset M$  is a submersion to a 
$k_j$-manifold $W_j\subset Q_j$,
where
\begin{align}
U_1 & =\bigcup_{i\in I_{\twostratum}}
\{|\eta_i|<5\}\,, \\
U_2 & =\bigcup_{i\in I_{\edge}}
\{|\eta_i|<5\De\,,\;\eta_{E'}<5\De\}\,, \notag \\
U_3 & =\bigcup_{i\in I_{\slim}}
\{|\eta_i|<5\cdot 10^5\De\} \notag
\end{align}
and $k_1 = 2$, $k_2 = k_3 = 1$.
\end{enumerate}
\end{proposition}

We will use the following additional parameters in this 
section:  $c_{\twostratum}, c_{\edge}, c_{\slim} > 0$ and 
$\Xi_i > 0$ for $i\in \{1,2,3\}$.

\subsection{Overview of the proof of Proposition
\ref{prop-adjustment}}

 In certain regions
of $M$,  the map $\e^0$ defined in the previous section, as well as
its composition with projection onto certain summands of $H$, 
behaves like a ``rough fibration''. 
As indicated in the
overview in Section \ref{sec-overview}, the next step is to 
modify the map $\e^0$ so as to 
 promote these rough fibrations  to honest fibrations, in such a
way that they are compatible on their overlap.
We will do this by producing a sequence of maps $\e^j:M\ra H$, for ${j\in \{1,2,3\}}$,
which are successive adjustments of the
map $\e^0$.   

To construct
the map $\e^j$ from $\e^{j-1}$,  $j\in \{1,2,3\}$, we will use the 
following procedure.   We  consider the orthogonal splitting $H= Q_j\oplus Q_j^\perp
$ of $H$,  and let $\pi_j=\pi_{1,j}:H\ra Q_j$, $\pi_j^\perp:H\ra Q_j^\perp$
 be the orthogonal projections. In Section \ref{sec-mapping}
we  considered a pair of subsets $(\widetilde{A_j},A_j)$ in $M$ whose
image $(\widetilde S_j,S_j)$ under the composition $\pi_j\circ \e^{j-1}$ is
 a cloudy $k_j$-manifold in $Q_j$, in the sense of Definition
 \ref{def-cloudykmanifold} of Appendix
\ref{sec-cloudy}. 
 We think of the restriction
of $\e^{j-1}$ to $A_j$ as defining a ``rough submersion''
over the cloudy $k_j$-manifold $(\widetilde S_j,S_j)$.
By Lemma \ref{approxlemma},
there is a  $k_j$-dimensional manifold
$W_j \subset Q_{j}$ near $(\widetilde S_j,S_j)$ and a projection map $P_j$ onto $W_j$,
defined in a neighborhood $\widehat{W}_j$ of $W_j$.  Hence we
have a well-defined map 
\begin{equation}
H\supset\widehat{W}_j\times Q_j^\perp
\stackrel{(\;P_j\circ \pi_j\;,\;\pi_j^\perp\;)}{\lra} Q_j\oplus Q_j^\perp = H\,.
\end{equation}
 Then using a partition
of unity, we blend the composition
$(P_j\circ \pi_j,\pi_j^\perp)\circ \e^{j-1}$ with $\e^{j-1}:M\ra H$
to obtain  $\e^j:M\ra H$. 
In fact, $\e^j$ will be the postcomposition of $\e^{j-1}$ with a map
from $H$ to itself.

We draw attention to two key features  of the construction.  
First,  in passing from $\e^{j-1}$ to $\e^j$, we
do not change it much. More precisely, at a point $p \in M$, we have
$|\e^{j-1}(p)) -  \e^j(p) | < \const \r_p$ and
$| D\e^{j-1}_p -  D \e^j_p | < \const$ for some
small constants. 
Second, the passage from   $\e^j$ to $\e^{j-1}$
respects the submersions defined by $\e^{j-1}$.

\subsection{Adjusting the map near the $2$-stratum} 
\label{subsecadjust2}

Our first adjustment step involves the $2$-stratum.

We take $Q_1=H$, $Q_1^\perp=\{0\}$, and we let 
$\widetilde{A}_1$, $A_1$, $\widetilde{S}_1$, $S_1$
and $r_1:\widetilde{S}_1 \rightarrow (0,\infty)$ be
as in Section 
\ref{subsec-e0near2stratum}.

Thus $(\tilde S_1,S_1,r_1)$ is a $(2,\Ga_1)$ cloudy $2$-manifold
by Lemma \ref{lem-2stratumestimates}.
By Lemma \ref{approxlemma}, there is  a $2$-manifold
$W_1^0 \subset H$
so that the
conclusion of Lemma \ref{approxlemma} holds, where the parameter
$\eps$ in the lemma is given by $\Xi_1=\Xi_1(\Ga_1)$.
(We remark that $W_1^0$ will not be the same as
the $W_1$ of
Proposition \ref{prop-adjustment}, due to subsequent adjustments.)
In particular, there is 
a well-defined nearest point projection
\begin{equation}
P_1:N_{r_1}(S_1)=\widehat{W}_1\ra 
W_1^0
\,,
\end{equation}
where we are using the notation for variable thickness 
neighborhoods from Section \ref{preliminaries}.

We now define a certain cutoff function.

\begin{lemma} \label{psi_1}
There is a smooth function $\psi_1:H\ra [0,1]$ with the following 
properties:
\begin{enumerate}
\item
\begin{align}
\psi_1\circ \e^0 & \equiv 1 \text{ in }
\bigcup_{i\in I_{\twostratum}}\;\{|\eta_i|<6\} \text{ and } \\
\psi_1\circ \e^0& \equiv 0 \text{ outside }
\bigcup_{i\in I_{\twostratum}}\;\{|\eta_i|<7\}\,. \notag
\end{align}
\item  $\supp(\psi_1)\cap\im(\e^0)\subset\widehat{W}_1$. 
\item There is a constant $\Om_1'=\Om_1'(\mathcal{M})$ such that
\begin{equation} \label{estimate}
|(d\psi_1)_x|< \Om_1'\,x_{\r}^{-1}
\end{equation}
for all $x\in \im(\e^0)$.
\end{enumerate}
\end{lemma}
\begin{proof}
Let $\psi_1:H\ra [0,1]$ be given by 
\begin{equation}
\label{eqn-psi1}
\psi_1(x)=1-\Phi_{\frac12,1}\left(
\sum_{\{i\in I_{\twostratum}\;\mid\; x_i''>0\}}\Phi_{6,6.5}\left(\frac{|x_i'|}{x_i''}\right)\cdot 
\left( 1-\Phi_{\frac12,1}\left(\frac{x_i''}{R_i}  \right) \right)
\right)\,.
\end{equation}
For each $i \in I_{\twostratum}$, the function
$x\mapsto \Phi_{6,6.5}\left(\frac{|x_i'|}{x_i''}\right)\cdot 
\left( 1-\Phi_{\frac12,1}\left(\frac{x_i''}{R_i}  \right) \right)$ is well-defined
and smooth in the set $\{x_i''>0\}$, with support contained
in the set $\{x_i''\geq \frac12 R_i\}$;
so extending it by zero
defines a  smooth function on $H$.  
Hence $\psi_1$ is smooth.

To prove part (1), suppose that $i\in I_{\twostratum}$ and $|\eta_i(p)|<6$.  
Then $\zeta_i(p)=1$ and 
$|\eta_i|<6$. Putting $x = \e^0(p)$, we have
\begin{equation} \label{components}
x_i=(x_i',x_i'')=\e^0_i(p)=(R_i\zeta_i(p)\eta_i(p),R_i\zeta_i(p))\,,
\end{equation}
so 
$x_i''=R_i$ and $\frac{|x_i'|}{x_i''}\in [0,6)$.
Hence 
\begin{equation}
\Phi_{6,6.5}\left(\frac{|x_i'|}{x_i''}\right)\cdot 
\left( 1-\Phi_{\frac12,1}\left(\frac{x_i''}{R_i} \right) \right) \: = \: 1,
\end{equation}
so $\psi_1(x)=1$.

Suppose now that $|\eta_i(p)|\geq 7$ for every $i \in I_{\twostratum}$. 
Putting $x = \e^0(p)$, for each 
$i \in I_{\twostratum}$ we claim that
\begin{equation}
\Phi_{6,6.5}\left(\frac{|x_i'|}{x_i''}\right)\cdot 
\left( 1-\Phi_{\frac12,1}\left(\frac{x_i''}{R_i} \right) \right) \: = \: 0;
\end{equation}
otherwise we would have $|x_i'|<6.5\,x_i''$ and $x_i''\geq R_i/2$,
which contradicts our assumption on $p$.
It follows that $\psi_1(x) = 0$. This proves part (1). 

To prove part (2),
suppose $x=\e^0(p)$ and $\psi_1(x)>0$.  Then from 
part (1), $|\eta_i(p)| < 7$ for some $i\in I_{\twostratum}$.
Therefore, $p\in A_1$ and
$x\in \e^0(A_1)=S_1\subset \widehat{W}_1$, so part (2) follows.

To prove part (3),
suppose that $x=\e^0(p)$.  If $x_i''>0$ then $\zeta_i(p)>0$,
so the number of such indices $i\in I_{\twostratum}$
is bounded by the multiplicity of the $2$-stratum cover;  for the
remaining indices $j\in I_{\twostratum}$, the quantity
$1-\Phi_{\frac12,1}\left(\frac{x_j''}{R_j}\right)$ vanishes
near $x$.  Thus by the chain rule, it suffices to bound the 
differential of
\begin{equation}
\Phi_{6,6.5}\left(\frac{|x_i'|}{x_i''}\right)\cdot 
\left( 1-\Phi_{\frac12,1}\left(\frac{x_i''}{R_i}  \right) \right)
\end{equation}
for each $i\in I_{\twostratum}$ for which $x_i''>0$.  But the 
differential is nonzero only when $\frac{|x_i|}{x_i''}\leq 6.5$
and $\frac{x_i''}{R_i}\geq \frac12$. In this case,
$R_i$ will be comparable
to $x_{\r}$ and the estimate (\ref{estimate}) follows easily.
\end{proof}

Define  $\Psi_1:H\ra H$ by
$\Psi_1(x)=x$ if $x\not\in\widehat{W}_1$ and 
\begin{equation}
\Psi_1(x)=\psi_1(x)P_1(x)+(1-\psi_1(x))x
\end{equation}
otherwise.   Put $\e^1=\Psi_1\circ \e^0$.

\begin{lemma}
\label{lem-propertiesofPsi_1}
Under  the constraints $\Si_1<\overline{\Si}_1(\Om_1,c_{\twostratum})$,
$\Ga_1<\overline{\Ga}_1(\Om_1,c_{\twostratum})$ and 
$\Xi_1<\overline{\Xi}_1(c_{\twostratum})$,
we have: 
\begin{enumerate}
\item $\e^1$ is smooth. 
\item  For all $p\in M$, 
\begin{equation} \label{est01}
\|\e^1(p)-\e^0(p)\|<c_{\twostratum}\,\r(p)\quad\mbox{and}
\quad \| D\e^1_p-D\e^0_p\|<c_{\twostratum}\,.
\end{equation}
\item
The restriction of $\e^1$ to  $\bigcup_{i\in I_{\twostratum}}
\{|\eta_i|<6\}$
is a submersion to 
$W_1^0$.
\end{enumerate}
\end{lemma}
\begin{proof}
That $\e^1$ is smooth follows from part (2) of Lemma 
\ref{psi_1}.

Given $p \in M$, put $x = \e^0(p)$. We have
\begin{equation}
\e^1(p) - \e^0(p) = \psi_1(x) \left( P_1(x) - x  \right).
\end{equation}
Now $|\psi_1(x)| \le 1$.
From Lemma \ref{approxlemma}(1),
$|P_1(x) - x| \le \Xi_1 r_1(x)$. 
From Sublemma \ref{etaclose}, we can assume that
$r_1(x) \le 10 \r_p$. This gives the first equation in
(\ref{est01}). 

Next,
\begin{align}
D\e^1_p - D\e^0_p = & (D\psi_1)_{x} \: \left( P_1(x) - x
\right) \: + \: \psi_1(x) \:
((DP_1)_x \circ  
D\e^0_p - D\e^0_p) \\
= & (D\psi_1)_{x} \: \left( P_1(x) - x
\right) \: + \: \psi_1(x) \:
((DP_1)_x  - \pi_{A^0_x}) \circ  
D\e^0_p + \notag \\
& \psi_1(x) \:
(\pi_{A^0_x} \circ  
D\e^0_p - D\e^0_p). \notag
\end{align}
Equation (\ref{estimate}) gives a bound on
$|(D\psi_1)_{x}|$. 
Lemma \ref{approxlemma}(1)
gives a bound on 
$|P_1(x) - x|$. 
Lemma \ref{approxlemma}(7)
gives a bound on 
$|(DP_1)_x  - \pi_{A^0_x}|$. 
Lemma \ref{uniform} gives a bound on $|D\e^0_p|$.
Equation (\ref{inherent1}) gives a bound on
$|\pi_{A^0_x} \circ  
D\e^0_p - D\e^0_p|$. 
The second equation in (\ref{est01}) follows from
these estimates.

Finally, the restriction of $\e^1$ to  $\bigcup_{i\in I_{\twostratum}}
\{|\eta_i|<6\}$ equals $P_1 \circ \e^0$. For $p \in 
\bigcup_{i\in I_{\twostratum}}
\{|\eta_i|<6\}$, put $x = \e^0(p)$. Then
\begin{equation}
D(P_1 \circ \e^0)_p = \pi_{A^0_x} \circ d\e^0_p + 
\left( (DP_1)_x - \pi_{A^0_x}\right) \circ D\e^0_p.
\end{equation}
Using (\ref{inherent2}) and 
Lemma \ref{approxlemma}(7),
if 
$\Xi_1$ is sufficiently small then
$D(P_1 \circ \e^0)_p$ maps onto $(TW^0_1)_{P_1(x)}$. 
This proves the lemma. 
\end{proof}

\subsection{Adjusting the map near the edge points}
\label{subsecadjustingedge}

Our second adjustment step involves the region near the edge points.

Recall that $Q_2=H_{\zeroball} \oplus H_{\slim} \oplus H_{\edge}$
and $\pi_2 : H \rightarrow Q_2$ is orthogonal projection.
We let 
$\widetilde{A}_2$, $A_2$, $\widetilde{S}_2$, $S_2$
and $r_2:\widetilde{S}_2 \rightarrow (0,\infty)$ be
as in Section 
\ref{subsec-e0nearedgestratum}.

Thus $(\tilde S_2,S_2,r_2)$ is a $(2,\Ga_2)$ cloudy $1$-manifold
by Lemma \ref{lem-edgestratumestimates}.
By Lemma \ref{approxlemma}, there is  a $1$-manifold
$W_2^0 \subset Q_2$
so that the
conclusion of Lemma \ref{approxlemma} holds, where the parameter
$\eps$ in the lemma is given by $\Xi_2=\Xi_2(\Ga_2)$.
(We remark that $W_2^0$ will not be the same as
the $W_2$ of
Proposition \ref{prop-adjustment}, due to subsequent adjustments.)
In particular, there is 
a well-defined nearest point projection
\begin{equation}
P_2:N_{r_2}(S_2)=\widehat{W}_2\ra 
W_2^0
\,,
\end{equation}
where we are using the notation for variable thickness 
neighborhoods from Section \ref{preliminaries}.
 
\begin{lemma} \label{psi_2}
Under the constraint
$c_{\twostratum} < \overline{c}_{\twostratum}$,
there is a smooth function 
$\psi_2: \{x_\r > 0 \} \ra [0,1]$ with the following 
properties:
\begin{enumerate}
\item
\begin{align}
\psi_2\circ \e^1 & \equiv 1 \text{ in }
\bigcup_{i\in I_{\edge}}\;\{|\eta_i|<6\Delta, \eta_{E'} < 6 \Delta\} 
\text{ and } \\
\psi_2\circ \e^1& \equiv 0 \text{ outside }
\bigcup_{i\in I_{\edge}}\;\{|\eta_i|<7\Delta, \eta_{E'} < 7 \Delta\}\,. \notag
\end{align}
\item $\supp(\psi_2)\cap\im(\e^1)\subset\widehat{W}_2
\times Q_2^\perp$. 
\item
There is a constant $\Om_2'=\Om_2'(\mathcal{M})$ such that
\begin{equation} \label{estimate2}
|(D\psi_2)_x|< \Om_2'\,x_{\r}^{-1}
\end{equation}
for all $x\in \im(\e^1)$.
\end{enumerate}
\end{lemma}
\begin{proof}
If the parameter $c_{\twostratum}$ is sufficiently small then
$\e^1(p) \in 
\bigcup_{i\in I_{\edge}}\;\{|\eta_i|<6\Delta, \eta_{E'} < 6 \Delta\}$
implies that
$\e^0(p) \in 
\bigcup_{i\in I_{\edge}}\;\{|\eta_i|<6.1\Delta, \eta_{E'} < 6.1 \Delta\}$,
and $\e^1(p) \notin 
\bigcup_{i\in I_{\edge}}\;\{|\eta_i|<7\Delta, \eta_{E'} < 7 \Delta\}$
implies that
$\e^0(p) \notin 
\bigcup_{i\in I_{\edge}}\;\{|\eta_i|<6.9\Delta, \eta_{E'} < 6.9 \Delta\}$.

In analogy to (\ref{additional1}), put
\begin{equation}
z_{\edge} = 1 - \Phi_{\frac12, 1} \left( \sum_{i \in I_{\edge}} 
\frac{x_i''}{R_i}
\right).
\end{equation}
Define $\psi_2 : \{x_\r > 0\} \ra [0,1]$ by
\begin{align}
\label{eqn-psi2}
\psi_2(x)=& 1-\Phi_{\frac12,1}\left( 
\sum_{\{i\in I_{\edge}\;\mid\; x_i''>0\}}
\Phi_{6.1\Delta,6.5\Delta}
\left(\frac{|x_i'|}{x_i''}\right)\cdot 
\left( 1-\Phi_{\frac12,1}\left(\frac{x_i''}{R_i}  \right) \right) \cdot
\right. \\
& \left. \left[
\left( 1-\Phi_{\frac14,\frac12}\left( \frac{x_{E'}''}{x_{\r}} \right) \right)
\Phi_{6.1\Delta,6.5\Delta}
\left(\frac{|x_{E'}'|}{x_{E'}''}\right) +
10 \left( \frac{x_i''}{R_i} z_{\edge} - \frac{x_{E'}''}{x_{\r}} \right) \right]
\right)\,. \notag
\end{align}
It is easy to see that $\psi_2$ is smooth.

To prove part (1), it is enough to show that
\begin{align}
\psi_2\circ \e^0 & \equiv 1 \text{ in }
\bigcup_{i\in I_{\edge}}\;\{|\eta_i|<6.1\Delta, \eta_{E'} < 6.1 \Delta\} 
\text{ and } \\
\psi_2\circ \e^0& \equiv 0 \text{ outside }
\bigcup_{i\in I_{\edge}}\;\{|\eta_i|<6.9\Delta, \eta_{E'} < 6.9 
\Delta\}\,. \notag
\end{align}

Suppose that $i \in I_{\edge}$, 
$|\eta_i(p)|<6.1 \Delta$ and $\eta_{E'}(p) < 6.1 \Delta$. 
Put $x = \e^0(p)$. Recall that
$x_i'' = R_i \zeta_{i}(p)$, where
$\zeta_{i}$ is given in (\ref{eqn-edgezetap}) with $p \leadsto p_i$,
and $x_{E'}'' = \r_p \zeta_{E'}(p)$, where $\zeta_{E'}$ is the expresssion
in (\ref{additional2}). Hence 
\begin{align}
& \frac{x_i''}{R_i} = \zeta_i(p) = 1, \\
& 1-\Phi_{\frac12,1}\left( \frac{x_i''}{R_i}  \right) = 1, \notag \\
& \Phi_{6.1\Delta,6.5\Delta}
\left( \frac{|x_i'|}{x_i''} \right) = 
\Phi_{6.1\Delta,6.5\Delta}
\left( |\eta_i(p)| \right) = 1. \notag
\end{align} 
If $\frac{x_{E'}''}{x_\r} = \zeta_{E'}(p) \ge \frac12$ then 
\begin{align} 
& 1-\Phi_{\frac14,\frac12}\left( \frac{x_{E'}''}{x_\r} \right) = 1, \\
& \Phi_{6.1\Delta,6.5\Delta}
\left( \frac{|x_{E'}'|}{x_{E'}''} \right) = 
\Phi_{6.1\Delta,6.5\Delta}
\left( |\eta_{E'}|(p) \right) = 1, \notag \\
& \frac{x_i''}{R_i} z_{\edge} - \frac{x_{E'}''}{x_{\r}} =
\zeta_i(p) \zeta_{\edge}(p) - \zeta_{E'}(p) 
= \zeta_{\edge}(p) - \zeta_{E'}(p) \ge 0. \notag
\end{align}
If $\frac{x_{E'}''}{x_\r} = \zeta_{E'}(p) < \frac12$ then  
$\left( 1-\Phi_{\frac14,\frac12}\left( \frac{x_{E'}''}{x_{\r}} \right) \right)
\Phi_{6.1\Delta,6.5\Delta}
\left(\frac{|x_{E'}'|}{x_{E'}''}\right) \ge 0$
and
\begin{equation}
\frac{x_i''}{R_i} z_{\edge} - \frac{x_{E'}''}{x_{\r}} =
\zeta_i(p) \zeta_{\edge}(p) - \zeta_{E'}(p) = 1 - \zeta_{E'}(p) \ge \frac12.
\end{equation}
In either case, the argument of $\Phi_{\frac12,1}$ in (\ref{eqn-psi2}) 
is bounded below by one and so
$\psi_2(x) = 1$. 

Now suppose that for all $i \in I_{\edge}$, either
$\zeta_i(p) =0$, or $\zeta_i(p) > 0$ and
$|\eta_i(p)| \ge 6.9\Delta$, or $\zeta_i(p) > 0$ and
$|\eta_i(p)| < 6.9\Delta$ and $\eta_{E'}(p) \ge 6.9 \Delta$.
If $\zeta_i(p) =0$, or $\zeta_i(p) > 0$ and
$|\eta_i(p)| \ge 6.9\Delta$, then 
\begin{equation}
\Phi_{6.1\Delta,6.5\Delta}
\left(\frac{|x_i'|}{x_i''}\right)\cdot 
\left( 1-\Phi_{\frac12,1}\left(\frac{x_i''}{R_i}  \right) \right) =
\Phi_{6.1\Delta,6.5\Delta}
\left( |\eta_i|(p) \right) \cdot 
\left( 1-\Phi_{\frac12,1}\left(\zeta_i(p)  \right) \right) = 0.
\end{equation}
If $|\eta_i(p)| < 6.9\Delta$ and $\eta_{E'}(p) \ge 6.9 \Delta$ then
\begin{equation}
\left( 1-\Phi_{\frac14,\frac12}\left( \frac{x_{E'}''}{x_{\r}} \right) \right)
\Phi_{6.1\Delta,6.5\Delta}
\left(\frac{|x_{E'}'|}{x_{E'}''}\right) =
\left( 1-\Phi_{\frac14,\frac12}\left( \zeta_{E'}(p) \right) \right)
\cdot \Phi_{6.1\Delta,6.5\Delta}
\left( |\eta_{E'}|(p) \right) = 0.
\end{equation}
and
\begin{align}
\frac{x_i''}{R_i} z_{\edge} - \frac{x_{E'}''}{x_{\r}} & =
\zeta_i(p) \zeta_{\edge}(p) - \zeta_{E'}(p) \\
& = \Phi_{8\De, 9\De}(\eta_{E'}(p)) \cdot \zeta_{\edge}(p) -
\Phi_{\frac{2}{10}\De, \frac{3}{10}\De, 8\De, 9\De}(\eta_{E'}(p)) \cdot
\zeta_{\edge}(p) = 0. \notag
\end{align}
Hence $\psi_2(x) = 0$.

This proves part (1) of the lemma.

The proof of the rest of the lemma is similar to that of Lemma \ref{psi_1}.
\end{proof}

We can assume that $\widehat{W}_2 \subset \{x_\r > 0\}$.
Define $\Psi_2: \{x_\r > 0\} \ra \{x_\r > 0\}$
by $\Psi_2(x)=x$ if 
$\pi_2(x)\not\in\widehat{W}_2$
and 
\begin{equation}
\Psi_2(x)= (\psi_2(x)P_2(\pi_2(x))+(1-\psi_2(x)) \pi_2(x),
\pi_2^\perp(x))
\end{equation}
otherwise.   Put $\e^2=\Psi_2 \circ \e^1$.

\begin{lemma}
\label{lem-propertiesofPsi_2}
Under  the constraints $\Si_2<\overline{\Si}_2(\Om_2,c_{\edge})$,
$\Ga_2<\overline{\Ga}_2(\Om_2,c_{\edge})$,
$\Xi_2<\overline{\Xi}_2(c_{\edge})$ and $c_{\twostratum} <
\overline{c}_{\twostratum}(c_{\edge})$,
we have: 
\begin{enumerate}
\item
 $\e^2$ is smooth.
\item For all $p\in M$, 
\begin{equation} \label{est12}
\|\e^2(p)-\e^0(p)\|<c_{\edge}\,\r(p)\quad\mbox{and}
\quad \| D\e^2_p-D\e^0_p\|<c_{\edge}\,.
\end{equation}
\item
 The restriction of $\pi_2 \circ \e^2$ to  
$\bigcup_{i\in I_{\edge}}
\{|\eta_i|<6\Delta, \eta_{E'} < 6 \Delta \}$
is a submersion to 
$W_2^0$.
\end{enumerate}
\end{lemma}
\begin{proof}
As in the proof of Lemma \ref{lem-propertiesofPsi_1},
$\e^2$ is smooth and
we can ensure that
\begin{equation}
\|\e^2(p)-\e^1(p)\|< \frac12 c_{\edge}\,\r(p)\quad\mbox{and}
\quad \| D\e^2_p-D\e^1_p\|< \frac12 c_{\edge}\,.
\end{equation}   
Along with (\ref{est01}), part (2) of the lemma follows.

The proof of part (3) is similar to that of Lemma 
\ref{lem-propertiesofPsi_1}(3).
We omit the details.
\end{proof}

\subsection{Adjusting the map near the slim $1$-stratum}
\label{subsecadjustslim}

Our third adjustment step involves the slim stratum.

Recall that $Q_3=H_{\zeroball} \oplus H_{\slim}$
and $\pi_3 : H \rightarrow Q_3$ is orthogonal projection.
We let 
$\widetilde{A}_3$, $A_3$, $\widetilde{S}_3$, $S_3$
and $r_3:\widetilde{S}_3 \rightarrow (0,\infty)$ be
as in Section 
\ref{subsec-e0nearslimstratum}.

Thus $(\tilde S_3,S_3,r_3)$ is a $(2,\Ga_3)$ cloudy $1$-manifold
by Lemma \ref{lem-slimstratumestimates}.
By Lemma \ref{approxlemma}, there is  a $1$-manifold
$W^0_3 \subset Q_3$
so that the
conclusion of Lemma \ref{approxlemma} holds, where the parameter
$\eps$ in the lemma is given by $\Xi_3=\Xi_3(\Ga_3)$.
In particular, there is 
a well-defined nearest point projection
\begin{equation}
P_3:N_{r_3}(S_3)=\widehat{W}_3\ra 
W^0_3
\,,
\end{equation}
where we are using the notation for variable thickness 
neighborhoods from Section \ref{preliminaries}.
 
\begin{lemma} \label{psi_3}
Under the constraint $c_{\edge} < \overline{c}_{\edge}$,
there is a smooth function $\psi_3:H\ra [0,1]$ with the following 
properties: 
\begin{enumerate}
\item
\begin{align}
\psi_3\circ \e^2 & \equiv 1 \text{ in }
\bigcup_{i\in I_{\slim}}\;\{|\eta_i|<6\cdot 10^5\Delta\} 
\text{ and } \\
\psi_3\circ \e^2& \equiv 0 \text{ outside }
\bigcup_{i\in I_{\slim}}\;\{|\eta_i|<7\cdot 10^5 \Delta \}\,. \notag
\end{align}
\item
  $\supp(\psi_3)\cap\im(\e^2)\subset\widehat{W}_3
\times Q_3^\perp$. 
\item
There is a constant $\Om_3'=\Om_3'(\mathcal{M})$ such that
\begin{equation} \label{estimate3}
|(D\psi_3)_x|< \Om_3'\,x_{\r}^{-1}
\end{equation}
for all $x\in \im(\e^2)$.
\end{enumerate}
\end{lemma}
\begin{proof}
Let $\psi_3:H\ra [0,1]$ be given by 
\begin{equation}
\label{eqn-psi3}
\psi_3(x)=1-\Phi_{\frac12,1}\left(
\sum_{\{i\in I_{\slim}\;\mid\; x_i''>0\}}\Phi_{6.1\cdot 10^5 \Delta,6.5 \cdot
10^5 \Delta}\left(\frac{|x_i'|}{x_i''}\right)\cdot 
\left( 1-\Phi_{\frac12,1}\left(\frac{x_i''}{R_i}  \right) \right)
\right)\,.
\end{equation}
The rest of the proof is similar to that of Lemma \ref{psi_2}.
We omit the details.
\end{proof}

Define  $\Psi_3:H\ra H$ by
$\Psi_3(x)=x$ if 
$\pi_3(x)\not\in\widehat{W}_3$ 
and 
\begin{equation} \label{Psi3def}
\Psi_3(x)= (\psi_3(x)P_3(\pi_3(x))+(1-\psi_3(x)) \pi_3(x),
\pi_3^\perp(x))
\end{equation}
otherwise.   Put $\e^3=\Psi_3 \circ \e^2$.

\begin{lemma}
\label{lem-propertiesofPsi_3}
Under  the constraints $\Si_3<\overline{\Si}_3(\Om_3,c_{\slim})$,
$\Ga_3<\overline{\Ga}_3(\Om_3,c_{\slim})$,
$\Xi_3<\overline{\Xi}_3(c_{\slim})$ and
$c_{\edge} < \overline{c}_{\edge}(c_{\slim})$,
we have: 
\begin{enumerate}
\item $\e^3$ is smooth. 
\item  For all $p\in M$, 
\begin{equation} \label{est23}
\|\e^3(p)-\e^0(p)\|<c_{\slim}\,\r(p)\quad\mbox{and}
\quad \| D\e^3_p-D\e^0_p\|<c_{\slim}\,.
\end{equation}
\item
The restriction of $\pi_3 \circ \e^3$ to  
$\bigcup_{i\in I_{\slim}}
\{|\eta_i|< 6 \cdot 10^5 \Delta \}$
is a submersion to 
$W^0_3$.
\end{enumerate}
\end{lemma}
\begin{proof}
The proof is similar to that of Lemma \ref{lem-propertiesofPsi_2}.  
We omit the details.
\end{proof}

\subsection{Proof of Proposition \ref{prop-adjustment}  }
\label{subsecproofofprop}

Note from (\ref{Psi3def}) that $\Psi_3$ can be factored as $\Psi_3^{Q_2} \times
I_{Q_2^\perp}$ for some $\Psi_3^{Q_2} : Q_2 \rightarrow Q_2$. In
particular, $\pi_2 \circ \Psi_3 = \Psi_3^{Q_2} \circ \pi_2$.

Put $\e = \e^3$, $c_{\adjust} = c_{\slim}$ and
\begin{align}
W_1 & = (\Psi_3 \circ \Psi_2)(W^0_1) \cap
\bigcup_{i \in I_{\twostratum}} \{y \in H \: : \:
y_i'' > .9 R_i, \: |y_i'| < 5.5 R_i \}, \\
W_2 & = \Psi_3^{Q_2} (W^0_2) \cap
\bigcup_{i \in I_{\edge}} \{y \in Q_2 \: : \:
y_i'' > .9 R_i, \: |y_i'| < 5.5 \Delta R_i, \:
y_\r > 0, \: y_{E'} < 5.5 \Delta y_\r \}, \notag \\
W_3 & = W^0_3 \cap
\bigcup_{i \in I_{\slim}} \{y \in Q_3 \: : \:
y_i'' > .9 R_i, \: |y_i'| < 5.5 \cdot 10^5 \Delta R_i\}. \notag
\end{align}

The smoothness of $\e$ follows from
part (1) of Lemma \ref{lem-propertiesofPsi_3}.
Part (1) of Proposition \ref{prop-adjustment} follows from
part (2) of Lemma \ref{lem-propertiesofPsi_3}.

\begin{lemma}
$W_i$ is a $k_i$-manifold.
\end{lemma}
\begin{proof}
We will show that $W_1$ is a $2$-manifold; the proofs for $W_2$ and
$W_3$ are similar.

Choose $x \in W_1$. For
some $i \in I_{\twostratum}$,
we have $x_i'' > .9 R_i$ and $|x_i'| < 5.5 R_i$ . Putting
\begin{equation}
V_i = W_1 \cap \left\{y \in H \: : \:
y_i'' > .9 R_i, \: |y_i'| < 5.5 R_i \right\}
\end{equation}
gives a neighborhood of $x$ in $W_1$. 
As $(\pi_{H_i'}, \pi_{H_i''}) \circ (\Psi_3 \circ \Psi_2) = 
(\pi_{H_i'}, \pi_{H_i''})$,
it follows that $V_i$ is the image, under $\Psi_3 \circ \Psi_2$, of
the $2$-manifold
\begin{equation}
V_i^0 = W^0_1 \cap \{y \in H \: : \:
y_i'' > .9 R_i, \: |y_i'| < 5.5 R_i \}.
\end{equation}
If we can show that $\pi_{H_i'}$ maps $V_i^0$
diffeomorphically to its image in $H_i'$ then $V_i^0$ will be a graph
over a domain in $H_i'$, and the same will be true for $V_i$.

In view of (\ref{est01}) and the definition of $\e^0$, 
if $c_{\twostratum}$ is sufficiently small then we are ensured that
$V_i^0 = \e^1(\{|\eta_i| < 7 \}) \cap \{y \in H \: : \: |y_i'| < 5.5 R_i \}$.
From Lemma \ref{lem-2stratumestimates}(3),
Lemma \ref{approxlemma}(3) and Lemma \ref{approxlemma}(5), 
if $\Xi_1$ is sufficiently small then we are ensured that
$\pi_{H_i'}$ restricts to a proper surjective local diffeomorphism from
$V_i^0$ to $B(0, 5.5 R_i) \subset H_i'$. Hence $V_i^0$ is a 
proper covering space
of $B(0, 5.5 R_i) \subset H_i'$ and so consists of a finite number of
connected components, each mapping diffeomorphically under $\pi_i'$ to
$B(0, 5.5 R_i) \subset H_i'$. It remains to show that there is only
one connected component.  

If $V_i^0$ has more than one connected component then there are
$y_1, y_2 \in V_i^0 \cap \pi^{-1}_{H_i'}(0)$ with
$y_1 \neq y_2$. We can write $y_1 = \e^1(p_1)$ and $y_2 = \e^1(p_2)$
for some $p_1, p_2 \in \{|\eta_i| < 7 \}$.
We claim that 
there is a smooth path $\gamma$ in $M$ from $p_1$ to
$p_2$ so that $\e^1 \circ \gamma$ lies within
$B \left( y_1, \frac{1}{10} R_i \right)$. To see this, we first
note that if $\Gamma_1$ and $c_{\twostratum}$ are sufficiently small
then Lemma \ref{lem-2stratumestimates}(3) and (\ref{est01}) 
ensure that $|\eta_i(p_1)| << 1$ and
$|\eta_i(p_2)| << 1$, as otherwise we would contradict the assumption
that $(y_1)_i' = (y_2)_i' = 0$. Let $\widehat{\gamma}$ be a straight
line from $\eta_i(p_1)$ to $\eta_i(p_2)$.
Relative to the fiber bundle structure defined by $\eta_i$
(see Lemma \ref{2fibration}), let $\gamma_1$ be a lift of $\widehat{\gamma}$,
with initial point 
$p_1$.
Let $\gamma_2$ be a curve in the $S^1$-fiber
containing $p_2$, going from the endpoint of $\gamma_1$ to $p_2$. Let
$\gamma$ be a smooth concatenation of $\gamma_1$ and $\gamma_2$.
Then $\eta_i \circ \gamma$ lies in a ball whose diameter is much
smaller than one.  
If $\Gamma_1$ and $c_{\twostratum}$ are sufficiently small
then Lemma \ref{lem-2stratumestimates}(3) and (\ref{est01}) 
ensure that $\e^1 \circ \gamma$ lies in a
ball whose diameter is much smaller than $R_i$.

On the other hand, since $p_1$ and $p_2$ lie in different
connected components of $V_i^0$, any curve in
$W_1^0$ from $p_1$ to $p_2$ must go from 
$p_1$ to $\{y \in H \: : \: |y_i'| = R_i\}$. This
is a contradiction.

Thus $V_i^0$ is connected and $W_1$ is a manifold.
\end{proof}

Recall the definition of $U_1$ from Proposition \ref{prop-adjustment}.
By Lemma \ref{lem-propertiesofPsi_1}(3), the restriction of $\e^1$
to $U_1$ is a submersion
from $U_1$ to $W^0_1$. 
From Lemma \ref{lem-2stratumestimates}(3) and 
(\ref{est23}), if $\Gamma_1$ and $c_{\slim}$ are sufficiently small then
$\e = \Psi_3 \circ \Psi_2 \circ \e^1$ maps $U_1$
to $W_1 \subset (\Psi_3 \circ \Psi_2)(W^0_1)$. To see that it is a 
submersion, suppose that $|\eta_i(p)| < 5$ for some
$i \in I_{\twostratum}$. Put $x^0 = \e^0(p)$ and
$x = \e(p)$. Note that $x_i' = (x_0)_i'$.
From Lemma \ref{lem-2stratumestimates}(3) and 
Lemma \ref{approxlemma}(3), 
if $\Xi_1$ is sufficiently small then we are ensured that
$(D\pi_{H_i'})_{x^0} \circ D\e^0_p$ maps onto $T_{(x^0)_i'} H_i' \cong \R^2$.
Then $(D\pi_{H_i'})_x \circ D\e_p = 
(D\pi_{H_i'})_x \circ D(\Psi_3 \circ \Psi_2)_{x^0} \circ D\e^0_p =
(D\pi_{H_i'})_{x_0} \circ D\e^0_p$ maps onto 
$T_{x_i'} H_i' \cong \R^2$. Thus $D\e_p$ must map
$T_pM$ onto $T_{x}W_1$, showing that $\e$ is a submersion near $p$.

Next, by Lemma \ref{lem-propertiesofPsi_2}(3), the restriction of $\pi_2
\circ \e^2$ to $U_2$ is a submersion from $U_2$ to $W^0_2$.
Lemma \ref{lem-edgestratumestimates}(3) and
(\ref{est23}) imply that if $\Gamma_2$ and $c_{\slim}$ are
sufficiently small then $\pi_2 \circ \e = \pi_2 \circ \Psi_3 \circ \e^2 =
\Psi_3^{Q_3} \circ \pi_2 \circ \e^2$ maps $U_2$ to $W_2 \subset
\Psi_3^{Q_3}(W_2^0)$.
By a similar argument to the preceding paragraph, the restriction of 
$\pi_2 \circ \e$ to $U_2$ is a submersion
to $W_2$.

Finally, by Lemma \ref{lem-propertiesofPsi_3}(3), the restriction of
$\pi_3 \circ \e = \pi_3 \circ \e^3$ to $U_3$ is a submersion to $W_3 = W^0_3$.
This proves Proposition \ref{prop-adjustment}.

\section{Extracting a good decomposition of $M$}
\label{sec-extracting}

In this section we will use the map $\e$ to find a decomposition
of $M$ into fibered pieces which are compatible along 
the intersections:

\begin{proposition}
\label{prop-decompositionprop}
There is a decomposition
\begin{equation}
M=M^{\zeroball}\cup M^{\slim}\cup M^{\edge}\cup M^{\twostratum}
\end{equation}
into compact domains with disjoint interiors, where 
each  connected component of $M^{\slim}$, $M^{\edge}$, or $M^{\twostratum}$
may be endowed with a fibration structure, such that:

\begin{enumerate}
\item $M^{\zeroball}$ and $M^{\slim}$ are domains with smooth
boundary, while $M^{\edge}$ and $M^{\twostratum}$ are smooth
manifolds with corners, each point of which has a neighborhood diffeomorphic
to  $\R^{3-k}\times [0,\infty)^k$
for some $k\leq 2$.

\item Connected components of 
$M^{\zeroball}$  are diffeomorphic to one of the following: 
$S^1 \times S^2$, $S^1 \times_{\Z_2} S^2 = \R P^3 \# \R P^3$,
$T^3/\Gamma$ (where $\Gamma$ is a finite subgroup of
$\Isom^+(T^3)$ which acts freely on $T^3$), $S^3/\Gamma$
(where $\Gamma$ is a finite subgroup of $\Isom^+(S^3)$ which
acts freely on $S^3$), a solid torus $S^1 \times D^2$,
a twisted line bundle $S^2 \times_{Z_2} I$ over $\R P^2$, or
a twisted line
bundle $T^2 \times_{Z_2} I$ over a Klein bottle.

\item The components of $M^{\slim}$ 
have a fibration  with $S^2$-fibers or $T^2$-fibers.

\item Components of
$M^{\edge}$ are diffeomorphic (as manifolds with
corners) to a solid torus $S^1\times D^2$ or $I\times D^2$,
and have a fibration with $D^2$ fibers.  

\item $M^{\twostratum}$ is a smooth domain with corners with a smooth
$S^1$-fibration; in particular
the $S^1$-fibration is compatible with any corners.

\item Each fiber of the fibration $M^{\edge}\ra B^{\edge}$, lying over
a boundary point of the base $B^{\edge}$, is contained in the 
boundary of  $M^{\zeroball}$ or the boundary of $M^{\slim}$.

\item The part of $\D M^{\edge}$ which carries
an induced $S^1$-fibration is contained in $M^{\twostratum}$, and the
$S^1$-fibration induced from $M^{\edge}$ agrees with the
one inherited from $M^{\twostratum}$.

\end{enumerate}

\end{proposition}

To prove the proposition, we show that the submersions
identified in Proposition \ref{prop-adjustment}
become fibrations, when 
restricted to appropriate subsets. 
Using this, we remove fibered regions around successive
strata in the following order: $0$-stratum, slim stratum,
the edge region and the $2$-stratum.   The compatibility
of the fibrations is automatic from the compatibility of the
various projection maps $\pi_j$, for $j\in \{1,2,3,4\}$.

\subsection{The definition of  $M^{\zeroball}$}
\label{subsecdefzeroball}

For each $i\in I_{\zeroball}$, put 
\begin{equation} \label{defzeroball}
M^{\zeroball}_i = 
B \left( p_i, .35 R_i \right) \cup \e^{-1} 
\left\{ x\in H \: : \: x_i'' \ge .9 R_i \,,\;
\frac{x_i'}{x_i''}\leq \frac{4}{10} \right\}\,.
\end{equation}

\begin{lemma} \label{zeroballworks}
Under the constraints $\varsigma_{\zeroball} < 
\overline{\varsigma}_{\zeroball}$ and $c_{\adjust} < \overline{c}_{\adjust}$,
$\{M^{\zeroball}_i\}_{i\in I_{\zeroball}}$
is a disjoint collection
and each $M^{\zeroball}_i$ is a compact manifold with boundary,
which is diffeomorphic
to one of the possibilities in Proposition 
\ref{prop-decompositionprop}(2).
\end{lemma}
\begin{proof}
Note that
\begin{equation}
(\e^0)^{-1} 
\left\{ x\in H \: : \: x_i'' \ge .9 R_i \,,\;
\frac{x_i'}{x_i''}\leq \frac{4}{10} \right\} \: = \:
\{ p \in M \: : \: \zeta_i(p) \ge .9 \, ,\; \eta_i(p) \le .4\}.
\end{equation}
In particular, if $\varsigma_{\zeroball}$ is sufficiently small then
this set contains $A(p_i, .31 R_i, .39 R_i)$ and is contained in
$A(p_i, .29R_i, .41 R_i)$. Then 
if $c_{\adjust}$ is sufficiently small,
$\e^{-1} 
\left\{ x\in H\mid x_i'' \ge .9 R_i \,,\;
\frac{x_i'}{x_i''}\leq \frac{4}{10} \right\}$ contains
$A(p_i, .32R_i, .38R_i)$ and is contained in
$A(p_i, .28R_i, .42R_i)$. 

In particular, $B(p_i, .38R_i) \subset M_i^{\zeroball} \subset
B(p_i, .42R_i)$. It now follows from Lemma \ref{0-ballseparation}
that $\{M^{\zeroball}_i\}_{i\in I_{\zeroball}}$ are disjoint.

To characterize the topology of $M_i^{\zeroball}$, if
$c_{\adjust}$ is sufficiently small then we can find a smooth
function $f^0 : M \rightarrow \R$ such that \\
1. If $p \in A(p_i, .3R_i, .5R_i)$ and $x = \e(p)$ then
$f^0(p) = \frac{x_i'}{x_i''}$. \\
2. If $p \in B(p_i, .35R_i)$ then $f^0(p) \le .39$. \\
3. If $p \notin B(p_i, .5R_i)$ then $f^0(p) \ge .41$.

Put $f^1 = \eta_i$ and define $F : M \times [0,1] \rightarrow \R$ by
$F(p,t) = (1-t) f^0(p) + t f^1(p)$. Put $f^t(p) = F(p,t)$ and 
$X = (-\infty, .4]$.
If $c_{\adjust}$ and $\varsigma_{\zeroball}$ are sufficiently 
small then Lemma \ref{lemgoodannulus} implies that
for each $t \in [0,1]$, $f^t$ is transverse to $\partial X =
\{.4\}$. By Lemma \ref{lem-isotopic1}, 
$M_i^{\zeroball} = (f^0)^{-1}(X)$ is
diffeomorphic to $(f^1)^{-1}(X)$. By Lemma \ref{lem-existsetap0stratum},
the latter
is diffeomorphic to one of the possibilities in Proposition 
\ref{prop-decompositionprop}(2). This proves the lemma.
\end{proof}

We let $M^{\zeroball}=\bigcup_{i\in I_{\zeroball}}\;M^{\zeroball}_i$, and put
$M_1=M\setminus\Int(M^{\zeroball})$.  Thus $M^{\zeroball}$ and $M_1$ are smooth compact
manifolds with boundary.

\subsection{The definition of $M^{\slim}$}
\label{subsecdefslim}

We first truncate $W_3$. Put
\begin{equation}
W_3'=W_3\cap \bigcup_{i\in I_{\slim}}
\left\{ x\in Q_3\mid x_i''>.9R_i\,,\;\left|\frac{x_i'}{x_i''}\right|
<4\cdot 10^5\De \right\}
\end{equation}
and define
$U_3'=(\pi_3\circ\e)^{-1}(W_3')$.

\begin{lemma} \label{slimworks}
Under the constraints $\varsigma_{\slim} < 
\overline{\varsigma}_{\slim}(\Delta)$
and $c_{\adjust} < \overline{c}_{\adjust}$, we have
\begin{enumerate}
\item
$\bigcup_{i\in I_{\slim}}\;\{|\eta_i|\leq 3.5\cdot 10^5\De\}  
\subset U_3'\subset U_3$, 
where $U_3$ is as in Proposition \ref{prop-adjustment}.
\item
The restriction of $\pi_3\circ\e$ to $U_3'$ 
gives a proper submersion to $W_3'$.
In particular, it is a fibration. 
\item The fibers of $\pi_3\circ\e \: : \: U_3' \rightarrow W_3'$
are diffeomorphic to $S^2$ or $T^2$.
\item $M_1$ intersects $U_3'$ in a 
submanifold with boundary which is a union of fibers
of $\pi_3\circ\e:U_3'\ra W_3'$.
\end{enumerate}
\end{lemma}
\begin{proof}
For a given $i \in I_{\slim}$, suppose that $p \in M$ satisfies
$|\eta_i(p)| \leq 3.5\cdot 10^5\De$. Putting
$y = (\pi_3 \circ \e^0)(p) \in Q_3$, we have $y_i'' = R_i$ and
$\left|\frac{y_i'}{y_i''}\right|
\leq 3.5 \cdot 10^5\De$. Hence if $c_{\adjust}$ is small enough then
since $\De >> 1$,
we are ensured that, putting $x = (\pi_3 \circ \e)(p) \in Q_3$, we have
$x_i''>.9R_i$ and $\left|\frac{x_i'}{x_i''}\right|
<4\cdot 10^5\De$. As $p \in U_3$, Proposition \ref{prop-adjustment}
implies that $x \in W_3$. Hence
$\bigcup_{i\in I_{\slim}}\;\{|\eta_i|\leq 3.5\cdot 10^5\De\}  
\subset U_3'$.

Now suppose that $p \in U_3'$. 
Putting $x = (\pi_3 \circ \e)(p)$, for some $i \in I_{\slim}$ we have
$x_i''>.9R_i$ and $\left|\frac{x_i'}{x_i''}\right|
<4\cdot 10^5\De$. If $c_{\adjust}$ is small enough then we are
ensured that, putting $y = (\pi_3 \circ \e^0)(p)$, we have
$y_i'' \ge .8 R_i$ and
$\left|\frac{y_i'}{y_i''}\right|
\leq 4.5 \cdot 10^5\De$. Hence $|\eta_i(p)| \le 4.5 \cdot 10^5 \Delta$.
This shows that $U_3' \subset U_3$, proving part (1) of the lemma.

By Proposition \ref{prop-adjustment}, $\pi_3 \circ \e$ is a submersion
from $U_3$ to $W_3$. 
Hence it restricts to a surjective submersion on $U_3'$.

Suppose that $K$ is a compact subset of $W_3'$. Then
$(\pi_3 \circ \e)^{-1}(K)$ is a closed subset of
$M$ which is contained in $\overline{U}_3 = \bigcup_{i \in I_{\slim}}
\{|\eta_i| \le 5 \cdot 10^5 \Delta\}$. 
As $\{p_i\}_{i \in I_{\slim}}$ are in the slim $1$-stratum,
it follows from the definition of adapted coordinates that
$\{|\eta_i| \le 5\cdot 10^5 \Delta\}$ 
is a compact subset of $M$; cf. the proof of
Lemma \ref{lem-slimtopology}. Thus the restriction of $\pi_3 \circ \e$ 
to $U_3'$ is a proper submersion. This proves part (2) of the lemma.

To prove part (3) of the lemma,
given $x \in W_3'$, suppose that 
$p \in U_3'$ satisfies $(\pi_3 \circ \e)(p) = x$.
Choose $i \in I_{\slim}$ so that
$|\eta_i(p)| \le 4.5 \cdot 10^5 \Delta$. If $c_{\adjust}$ is sufficiently 
small then by looking at
the components in $H_i$, one sees that for any $p' \in U_3'$ satisfying
$(\pi_3 \circ \e)(p') = x$, we have
$p^\prime \in \{|\eta_i| < 5\cdot 10^5 \Delta\}$. Thus to determine
the topology of the fiber, we can just consider the restriction of
$\pi_3 \circ \e$ to $\{|\eta_i| < 5\cdot 10^5 \Delta\}$.

Let $\pi_{H_i'} : Q_3 \rightarrow H_i'$ be orthogonal
projection and put $X = \pi_{H_i'}(x) \in H_i'$.
As the restriction of 
$\pi_{H_i'} \circ \pi_3 \circ \e^0$ to $\{|\eta_i| < 5 \cdot 10^5 \Delta\}$ 
equals $\eta_i$, it follows that $\pi_{H_i'} \circ \pi_3 \circ \e^0$
is transverse there to $X$. 
By Lemma \ref{lem-slimtopology},
$\{|\eta_i| < 5 \cdot 10^5 \Delta\} \cap
(\pi_{H_i'} \circ \pi_3 \circ \e^0)^{-1}(X)$ 
is diffeomorphic to $S^2$ or $T^2$.

Consider the restriction of $(\pi_{H_i'} \circ \pi_3 \circ \e)$ to
$\{|\eta_i| < 5 \cdot 10^5 \Delta\}$.  
Proposition \ref{prop-adjustment} and Lemma \ref{lem-isotopic2}
imply that if $c_{\adjust}$ is sufficiently small then the fiber
$\{|\eta_i| < 5 \cdot 10^5 \Delta\} \cap
(\pi_{H_i'} \circ \pi_3 \circ \e)^{-1}(X)$ is
diffeomorphic to $S^2$ or $T^2$. 
In particular, it is connected. 
Now $(\pi_{H_i'} \circ \pi_3 \circ \e)^{-1}(X)$ is the preimage,
under $\pi_3 \circ \e \: : \: U_3' \rightarrow W_3'$, of the
preimage of $X$ under $\pi_{H_i'} \: : \: W_3' \rightarrow
H_i'$. From connectedness of the fiber, the preimage of
$X$ under $\pi_{H_i'} \: : \: W_3' \rightarrow
H_i'$ must just be $x$.
Hence $(\pi_3 \circ \e)^{-1}(x)$ is diffeomorphic to $S^2$ or $T^2$.
This proves part (3) of the lemma.

To prove part (4) of the lemma, given $j \in I_{\zeroball}$, 
suppose that $p \in \partial M^{\zeroball}_j$. If
$x = \e(p)$ then $x_j'' \ge .9 R_j$ and $x_j' = .4 x_j''$.
Suppose that $p \in U_3'$. If $q \in U_3'$ is a point in the same fiber of
$\pi_3 \circ \e : U_3' \rightarrow W_3'$ as $p$, put $y = \e(q) \in H$. As 
$\pi_3(x) = \pi_3(y)$, we have
$y_j'' \ge .9 R_j$ and $y_j' = .4 y_j''$. Thus $q \in \partial
M^{\zeroball}_j$. Hence
$\partial M^{\zeroball}_j$ is a union of fibers of
$\pi_3 \circ \e : U_3' \rightarrow W_3'$. 
In fact, since $\partial M^{\zeroball}_j$ is a 
connected $2$-manifold, it is a single fiber of $\pi_3 \circ \e$.
This proves part (4) of the lemma.
\end{proof}

Let $W_3''\subset W_3'$ be a compact $1$-dimensional manifold
with boundary such that $(\pi_3\circ\e)^{-1}(W_3'')$ contains
$\bigcup_{i\in I_{\slim}}\;\{|\eta_i|\leq 3.5\cdot 10^5\De\}$,
and put
$M^{\slim}=M_1\cap (\pi_3\circ\e)^{-1}(W_3'')$.  We endow
$M^{\slim}$ with the fibration induced by $\pi_3\circ\e$.

Put $M_2=M_1\setminus \Int(M^{\slim})$.

\subsection{The definition of $M^{\edge}$}
\label{subsecdefedge}

We first truncate $W_2$. Put
\begin{equation}
W_2'=W_2\cap \bigcup_{i\in I_{\edge}} \{x\in Q_2\mid
x_i'' \ge .9R_i\,,\;\left|\frac{x_i'}{x_i''}\right|<4\De\}
\end{equation}
and
\begin{equation}
U_2'=(\pi_2\circ\e)^{-1}(W_2')\cap
\left( \{ \eta_{E'} \le .35\Delta \} \cup 
\e^{-1} \{x\in H\mid x_{\r}>0\,,\;
\frac{x_{E'}}{x_{\r}}\leq 4\De\}\right)\,.
\end{equation}

\begin{lemma} \label{edgeworks}
Under the constraints $\La < \overline{\La}(\Delta)$, $\varsigma_{\edge} < 
\overline{\varsigma}_{\edge}(\Delta)$
and $c_{\adjust} < \overline{c}_{\adjust}$, we have
\begin{enumerate}
\item
$\bigcup_{i\in I_{\edge}}\;
\{|\eta_i|\leq 3.5\De\,,\;|\eta_{E'}|\leq 3.5\De\}  
\subset U_2'\subset 
U_2$, where $U_2$ is as in Proposition \ref{prop-adjustment}.
\item
The restriction of $\pi_2 \circ \e$ to $U_2'$ 
gives a proper submersion to $W_2'$.
In particular, it is a fibration.
\item The fibers of $\pi_2\circ \e \: : \: U_2' \rightarrow  W_2'$
are diffeomorphic to $D^2$. 
\item $M_2$ intersects $U_2'$ in a submanifold with corners 
which is a union of fibers
of $\pi_2\circ\e:U_2'\ra W_2'$.
\end{enumerate}
\end{lemma}
\begin{proof}
The proof is similar to that of Lemmas \ref{zeroballworks} and 
\ref{slimworks}. We omit the details.
\end{proof}

\begin{lemma}
Under the constraint $c_{\adjust} < \overline{c}_{\adjust}$,
$M_2\cap U_2'$ is compact. 
\end{lemma}
\begin{proof}
Suppose that $M_2\cap U_2'$ is not compact. As $M$ is compact,
there is a sequence 
$\{q^k\}_{k=1}^\infty \subset M_2\cap U_2'$ with a limit $q \in M$, for which
$q \notin M_2\cap U_2'$. Put $y=\e(q)$.

Since $M_2$ is closed we have $q\in M_2$ and so $q \notin U_2'$.
Since
$y_{\r}>0$ (assuming $c_{\adjust}$ is sufficiently small) we also
have 
$q \in  \{ \eta_{E'} \le .35 \Delta \} \cup 
\e^{-1} \{x\in H\mid x_{\r}>0\,,\;
\frac{x_{E'}}{x_{\r}}\leq 4\De\}$.

We know that $\pi_2(y) \in \overline{W_2'}$. 
As $q \notin U_2'$, it must be that 
$\pi_2(y) \notin W_2'$. 
Then for some $i\in I_{\edge}$, we have
$y_i'' \ge .9 R_i$ and $\left|\frac{y_i'}{y_i''}\right|=4\De$.
Now $p_i$ cannot be a slim $1$-stratum point, as otherwise
the preceding truncation step would force
$B(p_i, 1000\Delta R_i) \cap M_2 = \emptyset$, 
which contradicts the facts that $q \in M_2$
and $d(p_i, q) < 10 \Delta R_i$.

Lemma \ref{lem-smalleretai} now implies that there is a 
$j\in I_{\edge}$ such that $|\eta_j(q)|<2\De$.   If $c_{\adjust}$ is
sufficiently small then we are ensured that $y_j'' \ge .9 R_j$ and 
$\left|\frac{y_j'}{y_j''}\right|<3\De$.  Thus $\pi_2(y) \in W_2'$ and so
$q \in U_2'$, which is a contradiction.
\end{proof}

We put $M^{\edge}=U_2'\cap M_2$ and $W_2''=(\pi_2\circ\e)(M^{\edge})$.
We endow $M^{\edge}$ with the fibration
induced by $\pi_2\circ\e$.

Put $M_3=M_2\setminus \Int(M^{\edge})$.

\subsection{The definition of $M^{\twostratum}$}
\label{subsecdef2}

We first truncate $W_1$. Put
\begin{equation}
W_1'=W_1\cap \bigcup_{i\in I_{\twostratum}}
\left\{ x\in H\mid x_i''>.9\,,\;\left|\frac{x_i'}{x_i''}\right|
<4 \right\}
\end{equation}
and define
$U_1'=\e^{-1}(W_1')$.

\begin{lemma}
Under the constraints
$\varsigma_{\twostratum} < \overline{\varsigma}_{\twostratum}$ and
$c_{\adjust} < \overline{c}_{\adjust}$, we have
\begin{enumerate}
\item
$\bigcup_{i\in I_{\twostratum}}\;\{|\eta_i|\leq 3.5\}  
\subset U_1'\subset 
U_1$, where $U_1$ is as in Proposition \ref{prop-adjustment}.
\item
The restriction of $\e$ to $U_1'$ 
gives a proper submersion to $W_1'$.
In particular, it is a fibration. 
\item The fibers of $M^{\twostratum}$ are circles.
\item $M_3$ is contained in $U_1'$, and is a 
submanifold with corners which is a union of fibers
of $\e\restr_{U_1'}:U_1'\ra W_1'$.
\end{enumerate}
\end{lemma}
\begin{proof}
The proof is similar to that of Lemma \ref{slimworks}. We omit the details.
\end{proof}

We put $M^{\twostratum}=M_3$, and endow it with the fibration
$\e\restr_{M^{\twostratum}}:M^{\twostratum}\ra \e(M^{\twostratum})$.

\subsection{The proof of Proposition \ref{prop-decompositionprop}}

Proposition \ref{prop-decompositionprop} now follows from combining
the results in this section.

\section{Proof of Theorem \ref{thmmain2} for closed manifolds} 
\label{sec-proof}

Recall that we are trying to get a contradiction to
Standing Assumption \ref{assumptions}. As before, we let $M$ denote
$M^\al$ for large $\al$. Then $M$ satisfies the conclusion of
Proposition \ref{prop-decompositionprop}. To get a contradiction,
we will show that $M$ is a graph manifold.

We recall the definition of a graph manifold from Definition \ref{graphdef}. 
It is obvious that boundary components of graph manifolds are tori.
It is also obvious that if we glue two graph manifolds along boundary
components then the result is a graph manifold, provided that it
is orientable.  In addition, the connected sum of two graph manifolds
is a graph manifold.  For more information about graph manifolds,
we refer to \cite[Chapter 2.4]{Matveev}

\subsection{$M$ is a graph manifold}
Each connected component of $M^{\zeroball}$ has
boundary either $\emptyset$, $S^2$ or $T^2$. 
If there is a connected component of $M^{\zeroball}$ with empty boundary
then $M$ is
diffeomorphic to
$S^1 \times S^2$, $S^1 \times_{\Z_2} S^2 = \R P^3 \# \R P^3$,
$T^3/\Gamma$ (where $\Gamma$ is a finite subgroup of
$\Isom^+(T^3)$ which acts freely on $T^3$) or $S^3/\Gamma$
(where $\Gamma$ is a finite subgroup of $\Isom^+(S^3)$ which
acts freely on $S^3$). In any case $M$ is a graph manifold.
So we can assume that each
connected component of $M^{\zeroball}$ has nonempty boundary.

Each connected component of $M^{\slim}$ fibers over $S^1$ or $I$.
If it fibers over $S^1$ then $M$ is diffeomorphic
to $S^1 \times S^2$ or the total space of a $T^2$-bundle over $S^1$.
In either case, $M$ is a graph manifold.
Hence we can assume that each connected component of $M^{\slim}$
is diffeomorphic to $I \times S^2$ or $I \times T^2$.

\begin{lemma} \label{Euler1}
Let $M^{\zeroball}_i$ be a connected component of
$M^{\zeroball}$. If $M^{\zeroball}_i \cap M^{\slim} \neq \emptyset$
then $\partial M^{\zeroball}_i$ is a boundary component of a
connected component of $M^{\slim}$. If 
$M^{\zeroball}_i \cap M^{\slim} = \emptyset$
then we can write
$\partial M^{\zeroball}_i = A_i \cup B_i$
where 
\begin{enumerate}
\item $A_i = M^{\zeroball}_i \cap M^{\edge}$ is a disjoint union of
$2$-disks, 
\item $B_i = M^{\zeroball}_i \cap M^{\twostratum}$ is the total space
of a circle bundle and 
\item 
$A_i \cap B_i = \partial A_i \cap \partial B_i$
is a union of circle fibers.
\end{enumerate}

Furthermore, if $\partial M^{\zeroball}_i$ is a $2$-torus then
$A_i = \emptyset$, while if $\partial M^{\zeroball}_i$ is a $2$-sphere
then $A_i$ consists of exactly two $2$-disks.
\end{lemma}
\begin{proof}
Proposition \ref{prop-decompositionprop} implies all but the last
sentence of the lemma. The statement about $A_i$
follows from an Euler characteristic argument.
\end{proof}

\begin{lemma} \label{Euler2}
Let $M^{\slim}_i$ be a connected component of
$M^{\slim}$. Let $Y_i$ be one of the connected components of
$\partial M^{\slim}_i$. If $Y_i \cap M^{\zeroball} \neq \emptyset$
then $Y_i = \partial M^{\zeroball}_i$ for some connected component
$M^{\zeroball}_i$ of $M^{\zeroball}$. 

If $Y_i \cap M^{\zeroball} = \emptyset$
then we can write
$Y_i = A_i \cup B_i$
where 
\begin{enumerate}
\item $A_i = Y_i \cap M^{\edge}$ is a disjoint union of
$2$-disks,
\item $B_i = Y_i \cap M^{\twostratum}$ is the total space
of a circle bundle and 
\item 
$A_i \cap B_i = \partial A_i \cap \partial B_i$ 
is a union of circle fibers.
\end{enumerate}

Furthermore, if $Y_i$ is a $2$-torus then
$A_i = \emptyset$, while if $Y_i$ is a $2$-sphere
then $A_i$ consists of exactly two $2$-disks.
\end{lemma}
\begin{proof}
The proof is similar to that of Lemma \ref{Euler1}.
We omit the details.
\end{proof}

Hereafter we can assume that
there is a disjoint union
$M^{\zeroball} = M^{\zeroball}_{S^2} \cup M^{\zeroball}_{T^2}$,
based on what the boundaries of the connected components are.  Similarly, each
fiber of $M^{\slim}$ is $S^2$ or $T^2$, so 
there is a disjoint union
$M^{\slim} = M^{\slim}_{S^2} \cup M^{\slim}_{T^2}$.

It follows from Lemmas \ref{Euler1} and \ref{Euler2} that
each connected component of $M^{\zeroball}_{T^2} \cup
M^{\slim}_{T^2}$ is diffeomorphic to \\
1. A connected component of $M^{\zeroball}_{T^2}$, \\
2. The gluing of two connected components of $M^{\zeroball}_{T^2}$
along a $2$-torus, or \\
3. $I \times T^2$.

In case 1, the connected component is diffeomorphic to
$S^1 \times D^2$ or the total space of a twisted interval bundle over
a Klein bottle.
In any case,
we can say that $M^{\zeroball}_{T^2} \cup M^{\slim}_{T^2}$ is a graph
manifold. Put $X_1 = M - \Int(M^{\zeroball}_{T^2} \cup M^{\slim}_{T^2})$.
To show that $M$ is a graph manifold, 
it suffices to show that $X_1$ is a graph manifold.
Note that $X_1 = M^{\zeroball}_{S^2} \cup M^{\slim}_{S^2} \cup
M^{\edge} \cup M^{\twostratum}$.

Suppose that $M^{\zeroball}_i$ is a connected component of
$M^{\zeroball}_{S^2}$. From Proposition \ref{prop-decompositionprop},
$M^{\zeroball}_i$ is diffeomorphic to $D^3$ or $\R P^3 \# D^3$.
If $M^{\zeroball}_i$ is diffeomorphic to $\R P^3 \# D^3$,
let $Z_i$ be the result of replacing $M^{\zeroball}_i$ in $X_1$ by
$D^3$. Then $X_1$ is diffeomorphic to $\R P^3 \# Z_i$. 
As $\R P^3$ is a graph manifold, if $Z_i$ is a graph manifold
then $X_1$ is a graph manifold. Hence without loss of generality,
we can assume that each connected component of $M^{\zeroball}_{S^2}$
is diffeomorphic to a $3$-disk.

From Lemmas \ref{Euler1} and \ref{Euler2}, each connected
component of $M^{\zeroball}_{S^2} \cup M^{\slim}_{S^2}$ is
diffeomorphic to \\
1. $D^3$, \\
2. $I \times S^2$ or \\
3. $S^3$, the result of attaching two connected components of
$M^{\zeroball}_{S^2}$ by a connected component $I \times S^2$ of
$M^{\slim}_{S^2}$.

In case 3, $X_1$ is diffeomorphic to a graph manifold. In case 2,
if $Z$ is a 
connected component of $M^{\zeroball}_{S^2} \cup M^{\slim}_{S^2}$
which is diffeomorphic to $I \times S^2$ then we can do surgery
along $\{\frac12\} \times S^2 \subset X_1$ to replace
$I \times S^2 \subset X_1$ by a union of two $3$-disks.
Let $X_2$ be the result of performing the surgery.  Then
$X_1$ is recovered from $X_2$ by either taking a connected sum
of two connected components of $X_2$ or by taking a connected
sum of $X_2$ with $S^1 \times S^2$. In either case, if $X_2$ is a
graph manifold then $X_1$ is a graph manifold. 
Hence without loss of generality,
we can assume that each connected component of 
$(M^{\zeroball}_{S^2} \cup M^{\slim}_{S^2}) \subset X_1$ is
diffeomorphic to $D^3$.

Some connected components of $M^{\edge}$ may fiber over $S^1$.
If $Z$ is such
a connected component then it is diffeomorphic to $S^1 \times D^2$.
If $X_1 - \Int(Z)$ is a graph manifold
then $X_1$ is a graph manifold.  Hence without loss of generality,
we can assume that each connected component of 
$M^{\edge}$ is diffeomorphic to $I \times D^2$.

Let $G$ be a graph (i.e. $1$-dimensional CW-complex)
whose vertices correspond to connected components of 
$M^{\zeroball}_{S^2} \cup
M^{\slim}_{S^2}$, and whose edges correspond to connected components
of $M^{\edge}$ joining such ``vertex''
components.  
From Lemmas \ref{Euler1} and \ref{Euler2}, each vertex of $G$ has
degree two. Again from Lemmas \ref{Euler1} and \ref{Euler2},
we can label the connected components
of $M^{\zeroball}_{S^2} \cup
M^{\slim}_{S^2} \cup M^{\edge}$ by connected components of $G$.
It follows that each connected component of $M^{\zeroball}_{S^2} \cup
M^{\slim}_{S^2} \cup M^{\edge}$ is diffeomorphic to $S^1 \times D^2$.

We have now shown that $X_1$ is the result of gluing a disjoint 
collection of $S^1 \times D^2$'s to $M^{\twostratum}$,
with each gluing being
performed between the boundary of a $S^1 \times D^2$ factor and
a toral boundary component of $M^{\twostratum}$. As
$M^{\twostratum}$ is the total space of a circle bundle, it is
a graph manifold. Thus $X_1$ is a graph manifold. 
Hence we have
shown:

\begin{proposition}
Under the constraints imposed in the earlier sections,
$M$ is a
graph manifold.
\end{proposition}

\subsection{Satisfying the constraints}

We now verify that it is possible to simultaneously satisfy all the 
constraints that appeared in the construction.

We indicate a partial ordering of the parameters
which is respected by all the 
constraints appearing in
the paper.
This means 
that every constraint on a given parameter 
is an upper (or lower) bound given as a function
of other parameters which are strictly smaller in the 
partial order.  
Consequently, all constraints can be satisfied
simultaneously, since we may choose values for parameters
starting with those parameters which are minimal with respect
to the partial order, and proceeding upward. 

\begin{align} \label{parameterlist}
& \{\mathcal{M},\beta_3\}
 \prec \{c_{\slim},\Om_i,\Om_i'\}\prec\Ga_3
\prec \{\Si_3,\Xi_3\}
\prec c_{\edge}\prec\Ga_2\prec \{\Si_2 , \Xi_2\}
\prec c_{\twostratum}\prec \\
& \Ga_1\prec \{\Si_1, 
\Xi_1\}
\prec\varsigma_{\twostratum}\prec
 \be_2\prec\De \prec
\{\varsigma_{\edge},\varsigma_{E'},\varsigma_{\slim} \}
\prec \varsigma_{\zeroball} \prec \{\be_{E'}, \si_{E'}\} \prec \notag \\
&\si_E \prec \{\sigma, \Lambda\}
\prec\bar w\prec w'
\prec\be_E\prec\be_1
\prec \{\Upsilon_0, \de_0 \}\prec
\Upsilon_0'. \notag
\end{align}
This proves Theorem \ref{thmmain}.

\section{Manifolds with boundary} \label{sec-proof2}

In this section we consider manifolds with boundary.
Since our principal application is to the
geometrization conjecture,
we will only deal with manifolds whose boundary components
have a nearly cuspidal collar.  We recall that a {\em
hyperbolic cusp} is a complete manifold with boundary
diffeomorphic to $T^2\times[0,\infty)$, which is isometric
to the quotient of a horoball by an isometric  $\Z^2$-action.
More explicitly, a cusp is isometric to a quotient of the
upper half space $\R^2\times [0,\infty)\subset \R^3$, with the
metric $dz^2 + e^{-z}(dx^2+dy^2)$, by a rank-$2$ group of 
horizontal translations. (For application to the 
geometrization conjecture,
we take the cusp to have constant sectional curvature $- \: \frac14$).

\begin{theorem}
\label{thmmain2}
Let $K \ge 10$ be a fixed integer.
Fix a function $A:(0,\infty) \ra (0,\infty)$.
Then there 
is some $w_0 \in (0,c_3)$
such that the following holds.
 
Suppose that $(M, g)$ is a compact connected
orientable Riemannian $3$-manifold with boundary.   
Assume in addition that
\begin{enumerate}
\item The diameters of the connected components of $\partial M$
are bounded above by $w_0$.

\item For each component $X$ of $\D M$, there is a  hyperbolic
cusp ${\mathcal H}_X$ with boundary $\partial {\mathcal H}_X$, 
along with a 
$C^{K+1}$-embedding
 of pairs
$e : (N_{100}(\partial {\mathcal H}_X), \partial {\mathcal H}_X) 
\rightarrow (M, X)$
which is $w_0$-close to an isometry.  
\item
For every
$p \in M$ with $d(p, \partial M) \ge 10$, we have
$\vol(B(p, R_p)) \le w_0 R_p^3$.
\item For every $p \in M$, $w'\in [w_0,c_3)$, $k \in[0,K]$,
and $r \leq R_p$ such that $\vol(B(p,r))\geq w'r^3$, the inequality
\begin{equation}
\label{eqn-derivboundsmain2}
| \nabla^k\Rm| \le A(w^\prime) \:  r^{-(k+2)}
\end{equation}
holds in the ball $B(p, r)$.
\end{enumerate}
Then $M$ is a graph manifold.
\end{theorem}

In order to prove Theorem \ref{thmmain2}, we make the following
assumption.

\begin{assumption}
\label{assumptions2}
Let $K \ge 10$ be a fixed integer and let
$A': (0, \infty) \times (0, \infty) \ra (0,\infty)$ be a function.

We assume that $\{(M^\al,g^\al)\}_{\al = 1}^\infty$ 
is a sequence of connected closed Riemannian $3$-manifolds such that
\begin{enumerate}
\item The diameters of the connected components of $\partial M^\al$
are bounded above by $\frac{1}{\al}$.

\item For each component $X^\al$ of $\D M^\al$, 
there is a hyperbolic cusp ${\mathcal H}_{X^\al}$ 
along with a 
$C^{K+1}$-embedding
of pairs 
$e : (N_{100} (\partial {\mathcal H}_{X^\al}), 
\partial {\mathcal H}_{X^\al}) \rightarrow (M^\al, X^\al)$
which is $\frac{1}{\al}$-close to an isometry.  

\item For all $p \in M^\al$ with $d(p, \partial M^\al) \ge 10$,
the ratio $\frac{R_p}{r_p(1/\alpha)}$
of the curvature scale at $p$ to the $\frac{1}{\al}$-volume scale at
$p$ is bounded below by $\al$.
\item 
For all $p \in M^\al$ and $w^\prime \in [\frac{1}{\al}, c_3)$, 
let $r_p(w^\prime)$ denote the $w^\prime$-volume scale at $p$.
Then for each integer $k \in [0, K]$ and each $C \in (0, {\al})$, we have
$|\nabla^k\Rm| \le A'(C, w^\prime) \:  r_p(w^\prime)^{-(k+2)}$
on $B(p, C r_p(w^\prime))$.
\item Each $M^\al$ fails to be a graph manifold.
\end{enumerate}
\end{assumption}

As in Lemma \ref{contra}, to prove Theorem \ref{thmmain2} it
suffices to get a contradiction from Standing Assumption \ref{assumptions2}.
As before, we let $M$ denote the manifold $M^\al$ for large $\al$.
The argument to get a contradiction from 
Standing Assumption \ref{assumptions2} is a slight 
modification of the argument in the 
closed case, the main difference being the appearance of a new
family of points -- those lying in a collared region near
the boundary.

We will use the same set of the parameters as in the case of closed
manifolds, with an additional parameter $r_\partial$.
It will be placed at the end of the partial ordering in 
(\ref{parameterlist}),
after $\Upsilon_0^\prime$. 

Let $\{\partial_i M\}_{i\in I_{\partial}}$ be the collection
of boundary components of $M$, and
let $e_i : (N_{100} (\partial {\mathcal H}_i), 
\partial {\mathcal H}_i) 
\rightarrow (M, \partial_i M)$ be the embedding from 
Standing Assumption \ref{assumptions2}. Note that the
restriction of $e_i$ to $\partial {\mathcal H}_i$ is a diffeomorphism.
Put $b_i = d_{\partial {\mathcal H}_i} \in C^\infty
({\mathcal H}_i)$.
Let $\eta_i$ be a slight smoothing of $b_i \circ e_i^{-1}$
on $(b_i \circ e_i^{-1})^{-1}(1,99)$, as in
Lemma \ref{lem-generalsmoothing}.

\begin{lemma} \label{r0}
We may assume that
for all $p \in \eta_i^{-1}(5,95)$,
\begin{enumerate}
\item The curvature scale satisfies $R_p \in (1,3)$. 
\item $\r_p < r_{\partial}$.
\item There is a $(1,\beta_1)$-splitting of
$\left( \frac{1}{\r_p} M, p \right)$ for which
$\frac{1}{\r_p} \eta_i$ is an adapted coordinate of quality
$\zeta_{\slim}$.
\end{enumerate}
\end{lemma}
\begin{proof}
In view of the quality of the embedding $e_i$, it suffices to
check the claim on the constant-curvature space $b_i^{-1}(4, 96)$.
The diameter of $\partial {\mathcal H}_i$ can be assumed to be
arbitrarily small by taking $\al$ to be large enough.
The Riemannian metric on ${\mathcal H}_i$ has the form
$dz^2 + e^{-z} g_{T^2}$ for a flat metric $g_{T^2}$ on $T^2$, 
with $z \in [0, 100)$. 
The lemma follows from elementary estimates.
\end{proof}

We now select $2$-stratum balls, edge balls, slim
$1$-stratum balls and $0$-balls as in the closed case,
except with the restriction that the center points $p_i$ all
satisfy $d(p_i, \partial M) \ge 10$.

Given $i \in I_{\partial}$, let $B_i$ be the connected component of
$M - \eta_i^{-1}(90)$ containing $\partial_i M$.
 
\begin{lemma}
If $r_{\partial} < \overline{r}_{\partial}(\Upsilon_0^\prime)$ then
$M$ is diffeomorphic to $I \times T^2$ or
$\{B_i\}_{i \in I_{\partial}} \cup \{B(p_i, r^0_{p_i}) 
\}_{i \in I_{\zeroball}}$ is a disjoint collection of open sets.
\end{lemma}
\begin{proof}
We can assume that $M$ is not diffeomorphic to $I \times T^2$.

Suppose first that $B_i \cap B_j \neq \emptyset$ for some
$i,j \in I_{\partial}$ with $i \neq j$. Then
$\eta_i^{-1}(5,90)$ must intersect $\eta_j^{-1}(5,90)$.
It follows easily that $M = N_{10}(B_i) \cup N_{10}(B_j)$ is
diffeomorphic to $I \times T^2$, which is a contradiction.
Thus $B_i \cap B_j = \emptyset$.

Next, suppose that $B_i \cap B(p_j, r^0_{p_j}) \neq \emptyset$
for some $i \in I_{\partial}$ and $j \in I_{\zeroball}$.
If 
$r_\partial<\overline{r}_\partial(\Upsilon_0')$
then by  Lemma \ref{r0} we will have
$\Upsilon_0' \r_{p_j} < \frac{1}{100}$. 
Hence $r^0_{p_j} < \frac{1}{100}$ and the triangle inequality implies that
$p_j \in \eta_i^{-1}(5, 95)$. However, from Lemma \ref{r0}(3),
this contradicts the fact that $p_j$ is a $0$-stratum point.

Finally, if $i,j \in I_{\zeroball}$ and $i \neq j$ then
$B(p_i, r^0_{p_i}) \cap B(p_j, r^0_{p_j}) = \emptyset$ from
Lemma \ref{0-ballseparation}(1).
\end{proof}

Hereafter we assume that $M$ is not diffeomorphic to $I \times T^2$,
which is already a graph manifold.

For each $i \in I_{\partial}$, let $H_i$ be a copy of
$\R^2$. Put $H_{\partial} = \bigoplus_{i \in I_{\partial}} H_i$.
We also put
\begin{itemize}
\item $Q_1=H \bigoplus H_{\partial}$,
\item $Q_2=H_{\zeroball}\bigoplus H_{\slim}
\bigoplus H_{\edge} \bigoplus H_{\partial}$,
\item $Q_3=H_{\zeroball}\bigoplus H_{\slim} \bigoplus H_{\partial}$,
\item $Q_4=H_{\zeroball} \bigoplus H_{\partial}$.
\end{itemize}

For $i \in I_{\partial}$, let $\zeta_i  \in C^\infty(M)$ be the
extension by zero of $\Phi_{20,30,80,90} \circ \eta_i$ to $M$. Define
$\e^0_i : M \rightarrow H_i$ by
$\e^0_i(p) = \left( \eta_i(p) \zeta_i(p), \zeta_i(p) \right)$.
We now go through Sections \ref{sec-mapping} and \ref{sec-adjusting}, 
treating $H_{\partial}$ in parallel to $H_{\zeroball}$.
Next, in analogy to (\ref{defzeroball}), for each $i \in I_{\partial}$ we put
\begin{equation}
M^{\partial}_i = 
N_{35}(\partial_i M) \cup \e^{-1} 
\left\{ x\in H \: : \: x_i'' \ge .9 \,,\;
\frac{x_i'}{x_i''}\leq 40 \right\}\,.
\end{equation}
Then $M^{\partial}_i$ is diffeomorphic to $I \times T^2$.
We now go through the argument of Section \ref{sec-proof},
treating each $M^{\partial}_i$ as if it were an element of
$M^{\zeroball}_{T^2}$ without a core. As in 
Section \ref{sec-proof}, we conclude that $M$ is a graph manifold.
This proves Theorem \ref{thmmain2}.

\section{Application to the geometrization conjecture} \label{geom}

We now use the terminology of 
\cite{Kleiner-Lott} and \cite{Perelman2}.
Let $(M, g(\cdot))$ be a Ricci flow with surgery whose initial
$3$-manifold is compact. We normalize the metric by putting
$\widehat{g}(t) = \frac{g(t)}{t}$.  
Let $(M_t, \widehat{g}(t))$ be the time-$t$ manifold.
(If $t$ is a surgery time then we take $M_t$ to be the post-surgery
manifold.) 
We recall that the $w$-thin part $M^-(w,t)$ of $M_t$ is defined to be the set
of points $p \in M_t$ so that either $R_p = \infty$ or
$\vol(B(p, R_p)) < w R_p^3$. 
The $w$-thick part $M^+(w,t)$ of $M_t$ is $M_t - M^-(w,t)$.

The following theorem is proved in \cite[Section 7.3]{Perelman2}; see also
\cite[Proposition 90.1]{Kleiner-Lott}.

\begin{theorem} \cite{Perelman2} \label{hyperbolic}
There is a finite collection
$\{(H_i, x_i)\}_{i=1}^k$ of pointed complete finite-volume
Riemannian $3$-manifolds with constant sectional curvature
$- \: \frac14$ and, for large $t$, a decreasing function
$\beta(t)$ tending to zero and a family of maps
\begin{equation}
f_t \: : \: \bigsqcup_{i=1}^k H_i \supset \bigsqcup_{i=1}^k
B \left( x_i, \frac{1}{\beta(t)} \right) \rightarrow M_t
\end{equation}
such that 
\begin{enumerate}
\item $f_t$ is $\beta(t)$-close to being an isometry. 
\item The image of $f_t$ contains $M^+(\beta(t), t)$. 
\item The image under $f_t$ of a cuspidal torus of
$\{H_i\}_{i=1}^k$ is incompressible in $M_t$.
\end{enumerate}
\end{theorem}

Given a sequence $t^\alpha \rightarrow \infty$, let
$Y^\alpha$ be the truncation of $\bigsqcup_{i=1}^k H_i$ obtained
by removing horoballs at distance approximately 
$\frac{1}{2 \beta(t^\alpha)}$
from the basepoints $x_i$. Put 
$M^\alpha = M_{t^\alpha} -
f_{t^\al}(Y_{t^\alpha})$.

\begin{theorem} \cite{Perelman2} \label{graph}
For large $\alpha$, $M^\alpha$ is a graph manifold.
\end{theorem}

\begin{proof} 
We check that the hypotheses of Theorem \ref{thmmain2} are satisfied
for large $\al$.
Conditions (1) and (2) of Theorem \ref{thmmain2} follow from
the almost-isometric embedding of $\bigsqcup_{i=1}^k
\left( B(x_i, \frac{1}{\beta(t^\alpha)}) - 
B(x_i, \frac{1}{2 \beta(t^\alpha)}) \right) \subset \bigsqcup_{i=1}^k H_i$ in 
$M^\alpha$. 

Next,
Theorem \ref{hyperbolic} says that for any $\bar{w} > 0$,
for large $\alpha$ the
$\bar{w}$-thick part of $M_{t^\alpha}$ has already
been removed in forming $M^\alpha$.
Thus Condition (3) of
Theorem \ref{thmmain2} holds.

From Ricci flow arguments,
for each $w^\prime \in (0, c_3)$ there are
$\overline{r}(w') > 0$ and $K_k(w^\prime) < \infty$
so that for large $\al$ the following holds:
for every $p \in M^\al$, $w^\prime \in (0, c_3)$, $k \in [0,K]$ and
$r \le \min(R_p, \overline{r}(w'))$, the inequality
$| \nabla^k\Rm| \le K_k(w^\prime) \:  r^{-(k+2)}$
holds in the ball $B(p, r)$
\cite[Lemma 92.13]{Kleiner-Lott}. Hence to verify Condition (4) of
Theorem \ref{thmmain2}, at least for large $\al$, we
must show that if $p \in M^\al$ then the conditions
$r \le R_p$ and $\vol(B(p, r)) \ge w' r^3$ imply that
$r \le \overline{r}(w')$. 

Suppose not, i.e. we have $\overline{r}(w') < r \le R_p$. 
Then $\Rm \Big|_{B(p, r)} \: \ge
\: - \frac{1}{r^2}$. Using the fact that
$\vol(B(p, r)) \ge w' r^3$, the Bishop-Gromov inequality gives an
inequality of the form
$\vol(B(p, \overline{r}(w'))) \ge w'' \overline{r}^3(w')$
for some $w'' = w''(w') > 0$. 

We also have
$\Rm \Big|_{B(p, \overline{r}(w'))} \: \ge
\: - \frac{1}{\overline{r}^2(w')}$. Then from
\cite[Lemma 7.2]{Perelman} or \cite[Lemma 88.1]{Kleiner-Lott},
for large $\al$ we can assume that the sectional curvatures on 
$B(p, \overline{r}(w'))$ are arbitrarily close to $- \: \frac14$.
In particular, $R_p \le 5$. Then
\begin{equation}
\vol(B(p,R_p)) \ge \vol(B(p,r)) \ge w' r^3 = w' \left( \frac{r}{R_p}
\right)^3 R_p^3 \ge  w' \left( \frac{\overline{r}(w')}{5}
\right)^3 R_p^3.
\end{equation}
If $\al$ is sufficiently large then we conclude that $p \in 
f_{t^\al}(Y_{t^\alpha})$,
which is a contradiction.

We now take $A(w^\prime)$ to be a number so that
Condition (4) of Theorem \ref{thmmain2} holds for all $M^\al$.
From the preceding discussion, there is a finite such number.
Then for large $\al$, all of the hypotheses of Theorem 
\ref{thmmain2} hold.  The theorem follows.
\end{proof}

Theorems \ref{hyperbolic} and \ref{graph}, along with the
description of how $M_t$ changes under surgery
\cite[Section 3]{Perelman2},\cite[Lemma 73.4]{Kleiner-Lott}, imply 
Thurston's geometrization conjecture.

\section{Local collapsing without derivative bounds}
\label{sec-shioyayamaguchi}

In this section, we explain how one can remove the
bounds on derivatives of curvature
from the hypotheses of
Theorem \ref{thmmain}, to obtain:

\begin{theorem}
\label{thm-generalcollapsing}
There exists a $w_0\in (0,c_3)$ such that if $M$ is a closed,
orientable, Riemannian  $3$-manifold satisfying
\begin{equation}
\vol(B(p,R_p))<w_0R_p^3
\end{equation}
for every $p\in M$, then $M$ is a graph manifold.
\end{theorem} 

The bounds on the derivatives of curvature
are only used to obtain pointed 
$C^{K}$-limits
of sequences 
at the (modified) volume scale.  This occurs in Lemmas \ref{lem-edgetopology}
and \ref{lemgoodannulus}.  We explain how to adapt the statements
and proofs.

\subsection*{Modifications in Lemma \ref{lem-edgetopology}}  The statement
of the Lemma
does not require modification. 
In the proof, the map
$\phi$ will be a Gromov-Hausdorff approximation rather than a 
$C^{K+1}$-map
close to an isometry, and $Z$
will be a complete $2$-dimensional nonnegatively curved Alexandrov space.
As critical point theory for functions works the same way for
Alexandrov spaces as for Riemannian manifolds, 
and $2$-dimensional Alexandrov spaces
are topological manifolds, the statement and proof of 
Lemma \ref{surfacelemma} remain valid for $2$-dimensional Alexandrov
spaces.  The main difference in the proof 
of Lemma \ref{lem-edgetopology} is the method for
verifying the fiber topology.  For this, we use:

\begin{theorem}[Linear local contractibility \cite{Grove-Petersen}]
\label{thm-Grove-Petersen}
For every $w\in (0,\infty)$ and every positive
integer $n$, there exist  $r_0\in (0,\infty)$ and
$C\in (1,\infty)$  with the following property.
If $B(p,1)$ is a unit ball with compact closure in a 
Riemannian $n$-manifold, 
$\Rm \Big|_{B(p,1)} \ge -1$ 
and $\vol(B(p,1))\geq w$ then
the inclusion $B(p,r)\ra B(p,Cr)$ is null-homotopic
for every $r\in (0,r_0)$.
\end{theorem}
This uniform contractibility may be used to promote a  Gromov-Hausdorff
approximation $f_0$ to a nearby continuous map $f$: one first restricts
$f_0$ to the $0$-skeleton of a fine triangulation, and then
extends it inductively to higher skeleta simplex by simplex, 
using the controlled contractibility radius.  

\begin{lemma}
With notation from the proof of Lemma
\ref{lem-edgetopology}, the fiber $F=\eta_p^{-1}(\{0\})$ is homotopy
equivalent to
$B(\star_Z,4\De)\subset Z$. 
\end{lemma} 
\begin{proof}
Let $\widehat{\phi}$ be a quasi-inverse to the Gromov-Hausdorff approximation
$\phi$. 

To produce a map $F\ra B(\star_Z,4\De)$
we take  $\pi_Z\circ\phi\restr_F$, promote it to 
a continuous map as above, and then  use the absence of critical
points of $d_Y$ near $S(\star_Z,4\De)$ to homotope this
to a map  taking values in $B(\star_Z,4\De)$.
 
To get the map $B(\star_Z,4\De)\ra F$,
we apply the above procedure to promote $\widehat{\phi} \restr_F$
to a nearby continuous map $B(\star_Z,4\De)\ra \frac{1}{\r_p}M$.  Then using the
fibration structures defined by $\eta_p$  and
$(\eta_p,\eta_{E'})$, we may perturb this to a map
taking values in $F$.  

The compositions of these maps are close to the identity maps; using
a relative version of the
approximation procedure one shows that these are homotopic to identity maps.
\end{proof}

Thus we conclude that the fiber is a contractible compact
$2$-manifold with boundary, so it is a $2$-disk.

\subsection*{Modifications in Lemma \ref{lemgoodannulus}} 
In the statement of the lemma, $N_p$ is a $3$-dimensional
nonnegatively curved Alexandrov space instead of a nonnegatively
curved Riemannian manifold, and ``diffeomorphism'' is replaced by
``homeomorphism''.  

In the proof, the pointed
$C^{K}$-convergence
is replaced by pointed Gromov-Hausdorff
convergence to a $3$-dimensional nonnegatively curved Alexandrov space
$N$; otherwise, we retain the notation from the proof.   
We need:

\begin{theorem}[The Stability Theorem \cite{Perelman,Kapovitch}]
Suppose $\{(M_k,\star_k)\}$ is a sequence of Riemannian 
$n$-manifolds, such that the sectional curvature is bounded
below by a  ($k$-independent) function of the distance to the 
basepoint $\star_k$.  Let $X$ be an $n$-dimensional 
Alexandrov space with curvature bounded below, and assume that
$\phi_k:(X,\star_\infty) \ra(M_k,\star_k)$
is a $\de_k$-pointed Gromov Hausdorff approximation,
where $\de_k\ra 0$.  Then for every 
$R\in (0,\infty)$, $\eps\in (0,\infty)$, and every sufficiently
large $k$, there is a pointed map 
$\psi_k:(B(\star_\infty,R+\eps),\star_\infty)\ra (M_k,\star_k)$
which is a homeomorphism onto an open subset containing $B(\star_k,R)$,
where
$
d_{C^0} \left( \psi_k,\phi_k\restr_{B(\star_\infty,R+\eps)} \right)<\eps$.
\end{theorem}
Using critical point theory as before,
we get that the limiting Alexandrov space $N$ is homeomorphic
to the balls $B(p_\infty,R'')$ for $R''\in(\frac12 R',2R')$,
and there are no critical points for $d_{p_\infty}$ or $d_{p_j}$
in the respective annuli $A(p_\infty,\frac{R''}{10^3},10R'')\subset N$ and
$A(p_j, \frac{R''}{10^3},10R'')\subset \frac{1}{\r_{p_j}}M_j$.
The Stability Theorem produces a homeomorphism $\psi$
from the closed ball
$\ol{B(p_\infty,R'')}\subset M$  to a subset close to the ball 
$B(p_j,R'')\subset\frac{1}{\r_{p_j}}M_j$;
in particular, restricting $\psi$ to the sphere $S(p_\infty,R'')$
we obtain a Gromov-Hausdorff approximation from the surface
$S(p_\infty,R'')$ to the surface $S(p_j,R'')$.  Appealing to 
uniform contractibility (Theorem \ref{thm-Grove-Petersen}),
and using homotopies guaranteed by the absence of critical
points we get that $\psi\restr_{S(p_\infty,R'')}$ is close
to a homotopy equivalence.  As in the proof of Theorem
\ref{lem-existsetap0stratum}, we conclude that 
$\psi(\ol{B(p_\infty,R'')})$ is isotopic to $\ol{B(p_j,R'')}$.

Finally, we appeal to the classification of complete, 
noncompact, orientable, nonnegatively curved Alexandrov
spaces $N$, when $N$ is a noncompact topological $3$-manifold, from 
\cite{Shioya-Yamaguchi1} to conclude that the
list of possible topological types is the same as in the 
smooth case.

\begin{theorem}[Shioya-Yamaguchi \cite{Shioya-Yamaguchi1}]
If $X$ is a noncompact, orientable,
$3$-dimensional nonnegatively curved Alexandrov space
which is a topological manifold,
then $X$ is homeomorphic to one of the following:
$\R^3$, $S^1\times \R^2$, $S^2\times \R$, $T^2\times\R$,
or a twisted line bundle over $\R P^2$ or the Klein bottle.

\end{theorem} 

When $N$ is compact, we may
apply the main theorem of \cite{Simon} to see that the 
topological classification is the same as in the smooth case.
Alternatively, using the splitting theorem,
one may reduce to the case when $N$ has finite fundamental
group and use the elliptization conjecture (now a theorem
via  Ricci flow due to finite extinction time results).

\begin{remark}
Theorem \ref{thm-generalcollapsing} implies the 
collapsing result stated in the appendix of 
\cite{Shioya-Yamaguchi}.   Note that 
Theorem \ref{thm-generalcollapsing} is
strictly stronger, since the curvature scale need not
be small compared to the diameter.   However, we remark that
the argument of \cite{Shioya-Yamaguchi} also gives the 
stronger result, if one uses \cite{Simon} or the elliptization
conjecture as above.
\end{remark}

\appendix

\section{Choosing ball covers} \label{secballcovers}

Let $M$ be a 
complete Riemannian manifold and let $V$ be a bounded subset of $M$.
Given $p \in V$ and $r > 0$, we write $B(p, r)$ for the 
metric ball in $M$ around $p$ of radius $r$.
Let ${\mathcal R} \: : \: V \rightarrow \R$ be a (not necessarily 
continuous)
function with range in some compact positive interval.
For $p \in V$, we denote ${\mathcal R}(p)$ by ${\mathcal R}_p$.
Put $S_1 = V$, $\rho_1 \: = \: \sup_{p \in V} {\mathcal R}_p$ and
$\rho_\infty \: = \: \inf_{p \in V} {\mathcal R}_p$.
Choose a point $p_1 \in V$ so that ${\mathcal R}_{p_1} \ge \frac12 \rho_1$.
Inductively, for $i \ge 1$, let $S_{i+1}$ be the subset of $V$ consisting of
points $p$ such that $B(p, {\mathcal R}_p)$ is disjoint from
$B(p_1, {\mathcal R}_{p_1}) \cup \ldots \cup B(p_i, {\mathcal R}_{p_i})$.
If $S_{i+1} = \emptyset$ then stop. If $S_{i+1} \neq \emptyset$,
put $\rho_{i+1} \: = \: \sup_{p \in S_{i+1}} {\mathcal R}_p$ and
choose a point $p_{i+1} \in S_{i+1}$ so that 
${\mathcal R}_{p_{i+1}} \ge \frac12 \rho_{i+1}$.
This process must terminate after a finite number of steps, as 
the $\rho_1$-neighborhood of $V$ 
cannot contain an infinite number of
disjoint balls with radius at least $\frac12 \rho_\infty$.

\begin{lemma} \label{ballcovers}
$\{B(p_i, {\mathcal R}_{p_i})\}$ 
is a finite disjoint collection of balls such that
$V \subset \bigcup_i B \left( p_i, 3{\mathcal R}_{p_i} \right)$. 
Furthermore, given $q \in V$, there is some $N$ so
that $q \in B \left( p_N, 3{\mathcal R}_{p_N} \right)$ and
${\mathcal R}_q \le 2 {\mathcal R}_{p_N}$.
\end{lemma}
\begin{proof}
Given $q \in V$, we know that
$B(q, {\mathcal R}_q)$ intersects 
$\bigcup_i B(p_i, {\mathcal R}_{p_i})$.
Let $N$ be the smallest
number $i$ such that $B(q, {\mathcal R}_q)$ intersects 
$B(p_i, {\mathcal R}_{p_i})$.
Then $q \in S_{N}$ and so $\rho_N \ge {\mathcal R}_q$. Thus 
${\mathcal R}_{p_N} \ge 
\frac12 \rho_N \ge \frac12 {\mathcal R}_q$. As 
$B(q, {\mathcal R}_q)$ intersects 
$B(p_N, {\mathcal R}_{p_N})$,
we have $d(q, p_N) <
{\mathcal R}_q + {\mathcal R}_{p_N} \le 3 {\mathcal R}_{p_N}$.
\end{proof}

\section{Cloudy submanifolds}
\label{sec-cloudy}

In this section we define the notion of a cloudy $k$-manifold.
This is a subset of a Euclidean space with the property that near
each point, it looks coarsely close to an affine subspace of the
Euclidean space. The result of this appendix is that any
cloudy $k$-manifold can be well interpolated by a smooth $k$-dimensional
submanifold of the Euclidean space.

If $H$ is a Euclidean space, let $\Gr(k, H)$ denote the Grassmannian
of codimension-$k$ subspaces of $H$.  It is metrized by saying that
for $P_1, P_2 \in \Gr(k, H)$, if $\pi_1, \pi_2 \in \End(H)$ are
orthogonal projection onto $P_1$ and $P_2$, respectively, then
$d(P_1, P_2)$ is the operator norm of 
$\pi_1 - \pi_2$. If $H^\prime$ is
another Euclidean space then there is an isometric embedding
$\Gr(k, H) \rightarrow \Gr(k, H \oplus H^\prime)$. If $X$ is a
$k$-dimensional submanifold of $H$ then the 
normal
 map of $X$
is the map $X\rightarrow \Gr(k, H)$ which assigns to
$p\in X$ the normal space of $X$ at $p$.

\begin{definition}
\label{def-cloudykmanifold}
Suppose $C,\,\de\in(0,\infty)$,  $k\in\N$, 
and $H$ is a Euclidean space.
  A {\em $(C,\de)$ cloudy $k$-manifold in $H$} is a triple $(\widetilde{S},S,r)$, where
$S\subset \widetilde{S}\subset H$ is a pair of subsets,
and 
$r \: : \: \widetilde{S} \ra(0,\infty)$ is a (possibly discontinuous) function such that: 
\begin{enumerate}
\item For all $x,y \in \widetilde{S}$,
$|r(y) - r(x)| \le C (|x-y| +  r(x))$.
\item For all $x\in S$, the rescaled pointed subset 
$(\frac{1}{r(x)} \widetilde{S},x)$
is $\de$-close in the pointed Hausdorff distance to
$(\frac{1}{r(x)} A_x,x)$, where $A_x$ is a $k$-dimensional 
affine subspace of $H$.  Here, as usual, 
$\frac{1}{r(x)} \widetilde{S}$ means the subset $\widetilde{S}$
equipped with the distance function of $H$ rescaled by 
$\frac{1}{r(x)}$.
\end{enumerate}

We will sometimes say informally that a pair $(\widetilde{S},S)$ is a {\em cloudy $k$-manifold}
if it can be completed to a triple 
$(\widetilde{S},S,r)$ which is a $(C,\de)$ cloudy $k$-manifold for
some $(C,\de)$. We will write $A^0_x \subset H$ for the $k$-dimensional
linear subspace parallel to $A_x$
and we will write $\pi_{A^0_x}$ for orthogonal projection onto
$A^0_x$.
Let $P_{A_x} \: : \: H \rightarrow H$ be the nearest point projection to
$A_x$, given by 
$P_{A_x}(y) = x + \pi_{A^0_x}(y-x)$
\end{definition}

\begin{lemma} \label{approxlemma}
For all $k,K \in  \Z^+$, $\eps\in(0,\infty)$ and 
$C < \infty$,
there  is a $\de = \de(k,K,\eps,C) > 0$ with the following
property.   Suppose
 $(\widetilde{S},S,r)$ is a $(C,\de)$ cloudy $k$-manifold
in a Euclidean space $H$, and for every $x\in S$ we denote
by $A_x$ an affine subspace as in Definition 
\ref{def-cloudykmanifold}.
Then  
 there is a $k$-dimensional smooth submanifold
$W\subset H$ such that
\begin{enumerate}
\item 
For all $x\in S$, 
the pointed Hausdorff distance from 
$(\frac{1}{r(x)} \widetilde{S},x)$ to $(\frac{1}{r(x)} W,x)$
is at most $\eps$.
\item $W \subset N_{\eps r}( \widetilde{S})$.  
\item For all $x \in S$,  the restriction 
of the normal map of $W$ to
$B(x, r(x)) \cap W$ has image contained in
an $\epsilon$-ball of $A_x^\perp$ in $\Gr(k, H)$.
\item If $I$ is a multi-index with $|I| \le K$ then the $I^{th}$ covariant
derivative of the second fundamental form of $W$ at $w$ is bounded in norm by
$\epsilon \, r(x)^{-(|I|+1)}$.
\item $W \cap N_{r}(S)$ is properly embedded in $N_{r}(S)$.
\item 
The nearest point map
$P \: : \: N_{r}(S) \rightarrow W$ is a well-defined smooth submersion.
\item
If $I$ is a multi-index with $1 \le |I| \le K$ then for all $x \in S$, 
the restriction of 
$P - P_{A_x}$
to $B(x, r(x))$ has $I^{th}$ derivative bounded in norm by
$\epsilon \, r(x)^{-(|I|-1)}$.
\end{enumerate}
\end{lemma}

\begin{figure}[h]
\begin{center}
 
\input{B2.pspdftex}
 
\caption{}
\end{center}
\end{figure}
\begin{proof}
With the notation of Lemma \ref{ballcovers}, put 
$V = S$ and ${\mathcal R} = r$. Let
$T$ be the finite collection of points $\{p_i\}$ from 
the conclusion of Lemma \ref{ballcovers}.
Then $\{ B(\widehat{x}, r_{\widehat{x}}) \}_{\widehat{x} \in T}$ is a disjoint collection of balls
such that for any $x \in S$, there is some
$\widehat{x} \in T$ with $x \in B(\widehat{x}, 3 r(\widehat{x}))$ and
$r(x) \le 2 r(\widehat{x})$. Hence
$B(x, r(x)) \subset B(\widehat{x}, r(x) + 3 r(\widehat{x})) \subset B(\widehat{x},5 r(\widehat{x}))$.
This shows that 
$\bigcup_{x \in S} B(x, r(x)) \subset \bigcup_{\widehat{x} \in T} B(\widehat{x},5 r(\widehat{x}))$.

For each $\widehat{x} \in T$, let $A_{\widehat{x}} \subset H$ be the
$k$-dimensional affine subspace from 
Definition \ref{def-cloudykmanifold}, so that
$\left( \frac{1}{r(\widehat{x})} \widetilde{S}, \widehat{x} \right)$ is 
$\delta$-close in the pointed Hausdorff topology to
$(\frac{1}{r(\widehat{x})} A_{\widehat{x}},\widehat{x})$.
Here $\delta$ is a parameter which will eventually be made
small enough so the proof works. Let
$A^0_{\widehat{x}} \subset H$ be the $k$-dimensional 
linear subspace which is parallel to
$A_{\widehat{x}}$.  Let 
$p_{\widehat{x}} \: : \: H \rightarrow H$ be orthogonal projection onto the
orthogonal complement of $A^0_{\widehat{x}}$.

In view of the assumptions of the lemma, a 
packing argument shows that for
any $l < \infty$,
for sufficiently small $\delta$ 
there is a number $m = m(k,C, l)$ so that for all 
$\widehat{x} \in T$, there
are at most $m$ elements of $T$ in 
$B(\widehat{x}, l r(\widehat{x}))$. Fix a nonnegative
function $\phi \in C^\infty(\R)$ which is identically 
one on $[0,1]$ and
vanishes on $[2, \infty)$. For $\widehat{x} \in T$, 
define $\phi_{\widehat{x}} \: : \:
H \rightarrow \R$ by 
$\phi_{\widehat{x}}(v) \: = \: \phi \left( \frac{|v-\widehat{x}|}{10 
r(\widehat{x})} \right)$.
Let $E(k,H)$ be the set of pairs consisting of a codimension-$k$
plane in $H$ and a point in that plane.
That is, $E(k,H)$ is the total space of the universal bundle over
$\Gr(k,H)$. 
Given 
$v \in \bigcup_{\widehat{x}\in {T}}\;B(\widehat{x}, 5r(\widehat{x}))$,
put $O_v \: = \: \frac{\sum_{\widehat{x} \in T} \phi_{\widehat{x}}(v) \:
p_{\widehat{x}}}{\sum_{\widehat{x} \in T} 
\phi_{\widehat{x}}(v)}$.
Note that for small $\delta$,
there is a uniform upper bound on the number of nonzero terms in the
summation, in terms of $k$ and $C$; hence the rank of $O_v$ is also
bounded in terms of $k$ and $C$.

If $\delta$ is sufficiently
small then since the projection operators $p_{\widehat{x}}$ that occur
with a nonzero coefficient in the summation are uniformly
norm-close to each other,
the self-adjoint operator $O_v$ will have
$k$ eigenvalues near $0$, with the rest of the spectrum being near $1$.
Let $\nu(v)$ be the orthogonal complement of the span of the eigenvectors
corresponding to the $k$ eigenvalues of $O_v$ near $0$.
Let $Q_v$ be orthogonal projection onto $\nu(v)$.

Recall that $\bigcup_{x \in S} B(x, r(x)) \subset \bigcup_{\widehat{x} \in T} 
B(\widehat{x}, 5 r(\widehat{x}))$.
Define $\eta \: : \:  \bigcup_{\widehat{x} \in T} 
B(\widehat{x}, 5 r(\widehat{x}))  \rightarrow H$ by
\begin{equation}
\eta(v) \: = \: \frac{
\sum_{\widehat{x} \in T} \phi_{\widehat{x}}(v) \: Q_v(v-\widehat{x})}{
\sum_{\widehat{x} \in T} \phi_{\widehat{x}}(v)}.
\end{equation}
Define $\pi:\bigcup_{x\in S}\;B(x,r(x))\ra E(k,H)$ by
$\pi(v) \: = \: (\nu(v), \eta(v))$. 

If $\de$ is sufficiently small then $\pi$ is uniformly transverse to the zero-section of
$E(k, H)$.
Hence the inverse 
image under $\pi$ of the zero section  will be a $k$-dimensional 
submanifold $W$. The map $P$ is defined as in the statement of the lemma.

The conclusions
of the lemma follow from a convergence argument.  For example, for conclusion
(3), suppose that there is a sequence $\delta_j \rightarrow 0$ and
a collection of counterexamples to conclusion (3). Let $x_j \in S_j$ be the
relevant point.  In view of the multiplicity bounds, we can assume without loss of
generality that the dimension of the Euclidean space is uniformly bounded
above. Hence after passing to a subsequence, we can pass to the case when
$\dim(H_j)$ is constant in $j$. Then $\lim_{j \rightarrow \infty} (\frac{1}{r(x_j)}S_j,x_j)$
exists in the pointed Hausdorff topology and is a $k$-dimensional
plane $(S_\infty, x_\infty)$. The map $\nu_\infty$ is a constant map and
$\eta$ is an orthogonal projection. Then $W_\infty$ is a flat
$k$-dimensional manifold, which gives a contradiction.  The verifications
of the other conclusions of the lemma are similar.
\end{proof}

\section{An isotopy lemma}

\begin{lemma}
\label{lem-isotopic1}
Suppose that $F:Y\times [0,1]\ra N$ is a smooth map between manifolds, with
slices $\{f^t:Y\ra N\}_{t\in [0,1]}$, and let 
$X\subset N$ be a submanifold with boundary $\D X$.
If 
\begin{itemize}
\item $f^t$ is transverse to both $X$ and $\D X$ for every $t\in [0,1]$ and
\item 
 $F^{-1}(X)$ is compact
\end{itemize}
then $(f^0)^{-1}(X)$ is isotopic in 
$Y$ to $(f^1)^{-1}(X)$.

\end{lemma}
\begin{proof}
Suppose first that $\D X = \emptyset$.
Now $F(y,t) = f^t(y)$. For $v \in T_yY$ and $c \in \R$, we can write 
$DF(v + c \frac{\partial}{\partial t}) =
Df^t(v) + c \frac{\partial f^t}{\partial t}(y)$.
We know that if $F(y,t) = x \in X$ then 
$T_{(y,t)} (F^{-1}(X)) = (DF_{y,t})^{-1}(T_xX)$. 

By assumption, for each $t \in [0,1]$, if $f_t(y) = x \in X$ then we have
\begin{equation} \label{trans}
\Image (Df_t)_y + T_xX = T_xN.
\end{equation}

We want to show that projection onto the $[0,1]$-factor gives a submersion 
$F^{-1}(X)
\ra [0,1]$. Suppose not. Then for some $(y,t) \in F^{-1}(X)$, we have
$T_{(y,t)} F^{-1}(X) \subset T_y Y$. That is, putting $F(y,t)=x$, whenever
$v \in T_yY$ and $c \in \R$ satisfy 
$Df^t(v) + 
c \frac{\partial f^t}{\partial t}(y) \in T_xX$ then we must have $c = 0$.
However, for any $c \in \R$, equation (\ref{trans}) implies that we can solve
$Df_t(-v) + w = c \frac{\partial f^t}{\partial t}(y)$ for some 
$v \in T_yY$ and
$w \in T_xX$. This is a contradiction.

Thus we have a submersion from the compact set $F^{-1}(X)$ to $[0,1]$.
This submersion must have a product structure, from which the lemma follows.
 
The case when $\D X \neq \emptyset$ is similar.
\end{proof}

\begin{lemma}
\label{lem-isotopic2}
Suppose that $Y$ is a smooth manifold, $(X,\D X)\subset \R^k$ is a
smooth submanifold,  $f:Y\ra \R^k$ is transverse to both 
$X$ and $\D X$, and $\widehat{X}= f^{-1}(X)$ is compact.  Then for 
any compact
subset $Y^\prime \subset Y$ whose interior contains $\widehat{X}$, there is an
$\eps>0$ such that if $f':Y\ra \R^k$ and
$\|f'-f\|_{C^1(Y^\prime)}< \eps$
then $f'^{-1}(X)$ is isotopic to $f^{-1}(X)$.
\end{lemma}
\begin{proof}
This follows from Lemma \ref{lem-isotopic1}.
\end{proof}

\bibliography{collapse2}
\bibliographystyle{alpha}
\end{document}

%% file: 4.24.pspdftex
\begin{picture}(0,0)%
\includegraphics{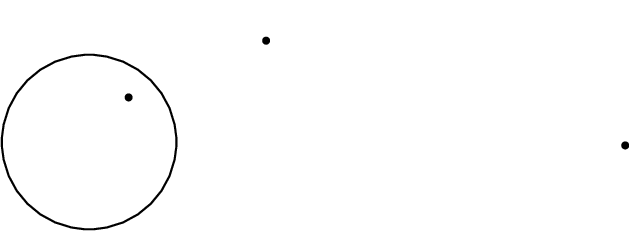}%
\end{picture}%
\setlength{\unitlength}{3947sp}%
\begingroup\makeatletter\ifx\SetFigFont\undefined%
\gdef\SetFigFont#1#2#3#4#5{%
  \reset@font\fontsize{#1}{#2pt}%
  \fontfamily{#3}\fontseries{#4}\fontshape{#5}%
  \selectfont}%
\fi\endgroup%
\begin{picture}(5033,1850)(2676,-4878)
\put(3392,-4208){\makebox(0,0)[b]{\smash{{\SetFigFont{12}{14.4}{\familydefault}{\mddefault}{\updefault}{\color[rgb]{0,0,0}$\star$}%
}}}}
\put(3376,-3736){\makebox(0,0)[lb]{\smash{{\SetFigFont{12}{14.4}{\familydefault}{\mddefault}{\updefault}{\color[rgb]{0,0,0}$m_i$}%
}}}}
\put(4501,-3211){\makebox(0,0)[lb]{\smash{{\SetFigFont{12}{14.4}{\familydefault}{\mddefault}{\updefault}{\color[rgb]{0,0,0}$m_i'$}%
}}}}
\put(7501,-4096){\makebox(0,0)[lb]{\smash{{\SetFigFont{12}{14.4}{\familydefault}{\mddefault}{\updefault}{\color[rgb]{0,0,0}$p_{j+}$}%
}}}}
\end{picture}%

%% file: 9.7.pspdftex
\begin{picture}(0,0)%
\includegraphics{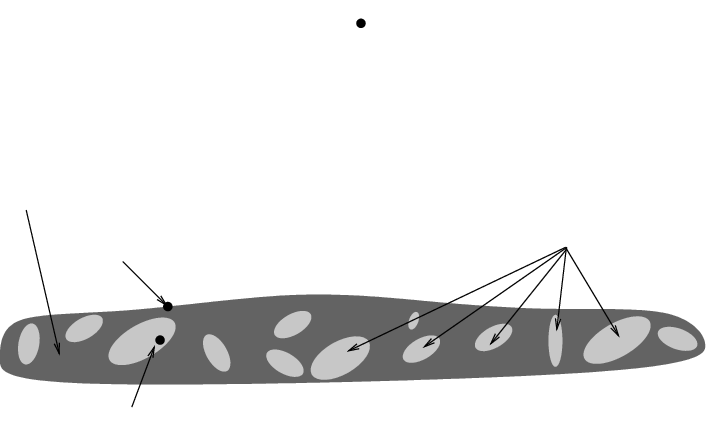}%
\end{picture}%
\setlength{\unitlength}{3947sp}%
\begingroup\makeatletter\ifx\SetFigFont\undefined%
\gdef\SetFigFont#1#2#3#4#5{%
  \reset@font\fontsize{#1}{#2pt}%
  \fontfamily{#3}\fontseries{#4}\fontshape{#5}%
  \selectfont}%
\fi\endgroup%
\begin{picture}(5643,3351)(2199,-5824)
\put(5161,-2656){\makebox(0,0)[lb]{\smash{{\SetFigFont{12}{14.4}{\familydefault}{\mddefault}{\updefault}{\color[rgb]{0,0,0}$p$}%
}}}}
\put(6616,-4366){\makebox(0,0)[lb]{\smash{{\SetFigFont{12}{14.4}{\familydefault}{\mddefault}{\updefault}{\color[rgb]{0,0,0}$E$}%
}}}}
\put(2281,-4096){\makebox(0,0)[lb]{\smash{{\SetFigFont{12}{14.4}{\familydefault}{\mddefault}{\updefault}{\color[rgb]{0,0,0}$E'$}%
}}}}
\put(3121,-4501){\makebox(0,0)[lb]{\smash{{\SetFigFont{12}{14.4}{\familydefault}{\mddefault}{\updefault}{\color[rgb]{0,0,0}$p'$}%
}}}}
\put(3181,-5746){\makebox(0,0)[lb]{\smash{{\SetFigFont{12}{14.4}{\familydefault}{\mddefault}{\updefault}{\color[rgb]{0,0,0}$q$}%
}}}}
\end{picture}%

%% file: 9.12.pspdftex
\begin{picture}(0,0)%
\includegraphics{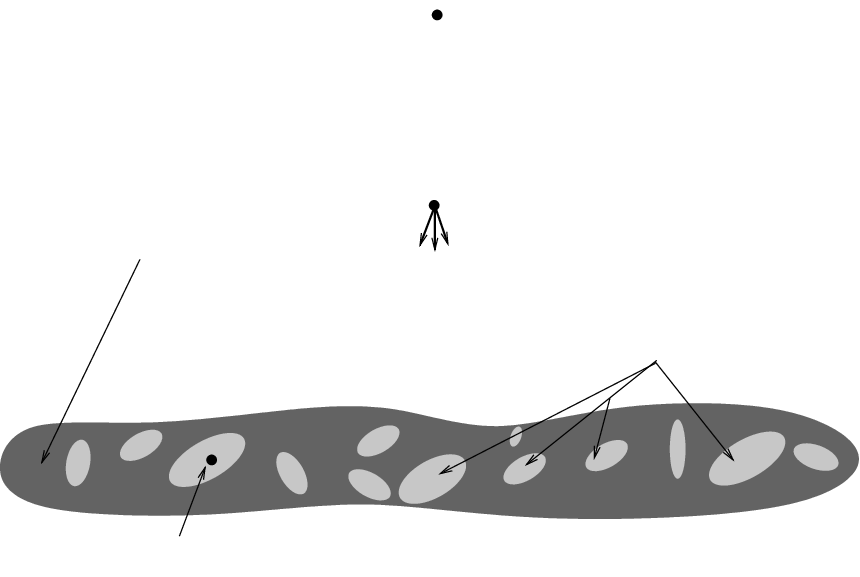}%
\end{picture}%
\setlength{\unitlength}{3947sp}%
\begingroup\makeatletter\ifx\SetFigFont\undefined%
\gdef\SetFigFont#1#2#3#4#5{%
  \reset@font\fontsize{#1}{#2pt}%
  \fontfamily{#3}\fontseries{#4}\fontshape{#5}%
  \selectfont}%
\fi\endgroup%
\begin{picture}(6871,4611)(1548,-6139)
\put(5026,-3061){\makebox(0,0)[b]{\smash{{\SetFigFont{12}{14.4}{\familydefault}{\mddefault}{\updefault}{\color[rgb]{0,0,0}$x$}%
}}}}
\put(5176,-1711){\makebox(0,0)[lb]{\smash{{\SetFigFont{12}{14.4}{\familydefault}{\mddefault}{\updefault}{\color[rgb]{0,0,0}$y$}%
}}}}
\put(5071,-3751){\makebox(0,0)[b]{\smash{{\SetFigFont{12}{14.4}{\familydefault}{\mddefault}{\updefault}{\color[rgb]{0,0,0}$V_x$}%
}}}}
\put(2956,-6061){\makebox(0,0)[b]{\smash{{\SetFigFont{12}{14.4}{\familydefault}{\mddefault}{\updefault}{\color[rgb]{0,0,0}$p$}%
}}}}
\put(6811,-4366){\makebox(0,0)[b]{\smash{{\SetFigFont{12}{14.4}{\familydefault}{\mddefault}{\updefault}{\color[rgb]{0,0,0}$E$}%
}}}}
\put(2701,-3511){\makebox(0,0)[b]{\smash{{\SetFigFont{12}{14.4}{\familydefault}{\mddefault}{\updefault}{\color[rgb]{0,0,0}$E'$}%
}}}}
\end{picture}%

%% file: 9.17.pspdftex
\begin{picture}(0,0)%
\includegraphics{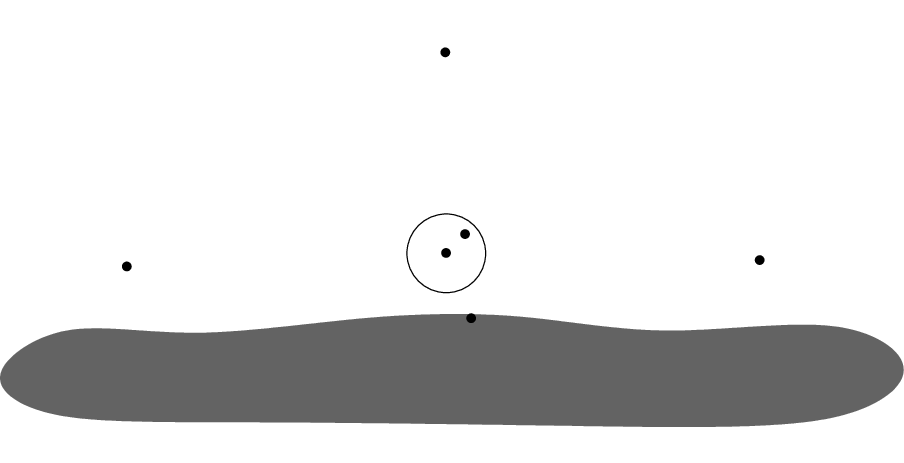}%
\end{picture}%
\setlength{\unitlength}{3947sp}%
\begingroup\makeatletter\ifx\SetFigFont\undefined%
\gdef\SetFigFont#1#2#3#4#5{%
  \reset@font\fontsize{#1}{#2pt}%
  \fontfamily{#3}\fontseries{#4}\fontshape{#5}%
  \selectfont}%
\fi\endgroup%
\begin{picture}(7229,3687)(1423,-5305)
\put(5285,-3334){\makebox(0,0)[b]{\smash{{\SetFigFont{12}{14.4}{\familydefault}{\mddefault}{\updefault}{\color[rgb]{0,0,0}$y$}%
}}}}
\put(2446,-3601){\makebox(0,0)[b]{\smash{{\SetFigFont{12}{14.4}{\familydefault}{\mddefault}{\updefault}{\color[rgb]{0,0,0}$x_1^-$}%
}}}}
\put(5071,-1801){\makebox(0,0)[b]{\smash{{\SetFigFont{12}{14.4}{\familydefault}{\mddefault}{\updefault}{\color[rgb]{0,0,0}$x_2^+$}%
}}}}
\put(7426,-3526){\makebox(0,0)[b]{\smash{{\SetFigFont{12}{14.4}{\familydefault}{\mddefault}{\updefault}{\color[rgb]{0,0,0}$x_1^+$}%
}}}}
\put(8146,-5236){\makebox(0,0)[b]{\smash{{\SetFigFont{12}{14.4}{\familydefault}{\mddefault}{\updefault}{\color[rgb]{0,0,0}$E'$}%
}}}}
\put(4921,-3751){\makebox(0,0)[b]{\smash{{\SetFigFont{12}{14.4}{\familydefault}{\mddefault}{\updefault}{\color[rgb]{0,0,0}$x$}%
}}}}
\put(5236,-4072){\makebox(0,0)[b]{\smash{{\SetFigFont{12}{14.4}{\familydefault}{\mddefault}{\updefault}{\color[rgb]{0,0,0}$z$}%
}}}}
\end{picture}%

%% file: B2.pspdftex
\begin{picture}(0,0)%
\includegraphics{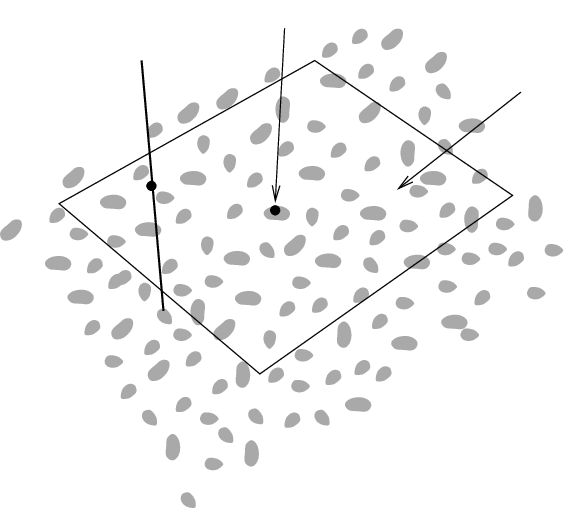}%
\end{picture}%
\setlength{\unitlength}{3947sp}%
\begingroup\makeatletter\ifx\SetFigFont\undefined%
\gdef\SetFigFont#1#2#3#4#5{%
  \reset@font\fontsize{#1}{#2pt}%
  \fontfamily{#3}\fontseries{#4}\fontshape{#5}%
  \selectfont}%
\fi\endgroup%
\begin{picture}(4505,4065)(2795,-4152)
\put(5111,-246){\makebox(0,0)[b]{\smash{{\SetFigFont{12}{14.4}{\familydefault}{\mddefault}{\updefault}{\color[rgb]{0,0,0}$\hat x$}%
}}}}
\put(6981,-761){\makebox(0,0)[b]{\smash{{\SetFigFont{12}{14.4}{\familydefault}{\mddefault}{\updefault}{\color[rgb]{0,0,0}$A_{\hat x}$}%
}}}}
\put(3931,-4001){\makebox(0,0)[b]{\smash{{\SetFigFont{12}{14.4}{\familydefault}{\mddefault}{\updefault}{\color[rgb]{0,0,0}$S$}%
}}}}
\put(3901,-481){\makebox(0,0)[b]{\smash{{\SetFigFont{12}{14.4}{\familydefault}{\mddefault}{\updefault}{\color[rgb]{0,0,0}$\nu(v)$}%
}}}}
\end{picture}%